\documentclass[12pt]{amsart}

\usepackage{tikz-cd}
\usepackage[a4paper,margin=23mm,top=30mm]{geometry}

\usepackage{amsfonts, amsmath, amssymb, amsgen, amsthm, amscd}
\usepackage{newtxtext,newtxmath}
\usepackage[utf8]{inputenc} 

\usepackage{color}

\usepackage[all]{xy}

\usepackage{hyperref}
\usepackage{mathtools,slashed}

\usepackage{enumerate}

\def\a{\alpha}
\def\b{\beta}
\def\d{\delta}
\def\D{\Delta}
\def\g{\gamma}

\def\Om{\Omega}
\def\s{\sigma}

\def\t{\theta}

\def\ve{\varepsilon}
\def\vp{\varphi}

\def\Id{\mathop{\rm Id}\nolimits}

\def\ot{\otimes}

\def\nb{\nabla}

\def\rt{\triangleright}
\def\lt{\triangleleft}
\def\acl{\blacktriangleright\hspace{-4pt}\vartriangleleft }

\def\cl{\blacktriangleright\hspace{-4pt} < }

\def\D{\Delta}

\def\Om{\Omega}

\def\Id{\mathop{\rm Id}\nolimits}

\newcommand{\ps}[1]{~\hspace{-4pt}_{^{(#1)}}}

\newcommand{\ns}[1]{~\hspace{-4pt}_{_{{<#1>}}}}

\usepackage{enumerate}

\newcommand{\G}[1]{\mathfrak{#1}}
\newcommand{\C}[1]{\mathcal{#1}}
\newcommand{\B}[1]{\mathbb{#1}}

\newcommand{\Hom}{{\rm Hom}}

\newcommand{\ie}{{\it i.e.\/}\ }

\newcommand{\brt}{\blacktriangleright}

\renewcommand{\leq}{\leqslant}
\renewcommand{\geq}{\geqslant}

\numberwithin{equation}{section}

\newtheorem{theorem}{Theorem}[section]
\newtheorem{proposition}[theorem]{Proposition}
\newtheorem{lemma}[theorem]{Lemma}
\newtheorem{corollary}[theorem]{Corollary}
\theoremstyle{definition}

\newtheorem{remark}[theorem]{Remark}
\newtheorem{example}[theorem]{Example}
\newtheorem{definition}[theorem]{Definition}

\title{Hom-Lie-Hopf algebras}

\author{S. Halıcı}
\address{Pamukkale University, Denizli, Turkey}
\email{shalici@pau.edu.tr}

\author{A. Karataş}
\address{Pamukkale University, Denizli, Turkey}
\email{adnank@pau.edu.tr}

\author{S. Sütlü}
\address{Işık University, Istanbul, Turkey}
\email{serkan.sutlu@isikun.edu.tr}

\begin{document}

\begin{abstract}
We studied both the double cross product and the bicrossproduct constructions for the Hom-Hopf algebras of general $(\a,\b)$-type. This allows us to consider the universal enveloping Hom-Hopf algebras of Hom-Lie algebras, which are of $(\a,\Id)$-type. We show that the universal enveloping Hom-Hopf algebras of a matched pair of Hom-Lie algebras form a matched pair of Hom-Hopf algebras. We observe also that, the semi-dualization of a double cross product Hom-Hopf algebra is a bicrossproduct Hom-Hopf algebra. In particular, we apply this result to the universal enveloping Hom-Hopf algebras of a matched pair of Hom-Lie algebras to obtain Hom-Lie-Hopf algebras.
\end{abstract}

\maketitle
\tableofcontents

\section*{Introduction}

An important class of (infinite dimensional, non-commutative and non-cocommutative) Hopf algebras in the theory of quantum groups consists of the bicrossed-product Hopf algebras (in the sense of \cite[Chap. 6]{Majid-book}) of the form $R(\G{h}) \acl U(\G{g})$, where $(\G{g},\G{h})$ is a matched pair of Lie algebras, $U(\G{g})$ is the universal enveloping (Hopf) algebra of $\G{g}$, and $R(\G{h})$ is the (Hopf) algebra of representative functions of $\G{h}$, see for instance \cite{Hoch59}. In this case, there is a Lie algebra structure on $\G{g} \bowtie \G{h}:=\G{g} \oplus \G{h}$, and a double-crossed product Hopf algebra structure (in the sense of \cite[Chap. 7]{Majid-book}) on the universal enveloping Hom-algebra $U(\G{g}\bowtie\G{h}) \cong U(\G{g}) \bowtie U(\G{h})$. Then dualizing $U(\G{h})$ to $R(\G{h})$, inspired by \cite{Majid-book, Maji90, Kac68}, there emerges the bicrossed-product Hopf algebra $R(\G{h}) \acl U(\G{g})$. 

The bicrossed-product Hopf algebras of the form $R(\G{h}) \acl U(\G{g})$ were studied first in \cite{MoscRang09}, and then in \cite{RangSutl}, see also \cite{RangSutl-I-arxiv}. Later on, in \cite{RangSutl-III}, it was observed that many of the (co)representation theoretical properties of  $R(\G{h}) \acl U(\G{g})$ are controlled by the Lie algebra $\G{g}\bowtie \G{h}$, which in turn, are controlled by the individual Lie algebras $\G{g}$ and $\G{h}$. Therefore, one gets the advantage of analysing a highly nontrivial (nonlinear)  object $R(\G{h}) \acl U(\G{g})$ in terms of a much milder (and linear) object $\G{g}\bowtie \G{h}$. As such, a strategy arises to attack any problem (representation theory, homology, etc.) related to $R(\G{h}) \acl U(\G{g})$.

Two concrete, and successful, applications of this strategy were presented in \cite{RangSutl-III}. It was noted that the Connes-Moscovici Hopf algebras $\C{H}_n$, that appear in local index formula of the transversally elliptic operators on foliations \cite{ConnMosc98,ConnMosc}, are examples of Lie-Hopf algebras associated to the (infinite dimensional) Lie algebras $W_n$ of formal vector fields on $\B{R}^n$, \cite{Fuks-book}. Then, the classification of the stable anti-Yetter-Drinfeld modules (SAYD-modules in short, the coefficient spaces for the Hopf-cyclic cohomology) over $\C{H}_n$ was reduced to the classification of the modules over $W_n$, which was known to be trivial. It was thus proved that the only SAYD-module on $\C{H}_n$ is the one dimensional one introduced in \cite{ConnMosc98}.

What is also established in \cite{RangSutl-III} is a quasi-isomorphism between the Hopf-cyclic cohomology of the Hopf algebra $\C{H}_{1S}$ (the Schwarzian quotient of the Connes-Moscovici Hopf algebra $\C{H}_1$, which also happens to be the Lie-Hopf algebra associated to the Lie algebra $s\ell_2$), and the Lie-algebra cohomology of  $s\ell_2$.

Now, on the other hand, the $\sigma$-derivations of these infinite dimensional Lie algebras $W_n$, see for instance \cite{GelfFuks70,Jako75}, give rise to Lie algebra structures so that the Jacobi identity is deformed by a homomorphism, called Hom-Lie algebras in \cite{HartLarsSilv06}. Upon discovering their fundamental importance towards the development of a non-associative geometry (in the study of quantum groups, or the extension of the non-commutative formalism of the spectral action principle to a non-associative
framework, see for instance \cite{HassShapSutl15} and the references therein), the emergence of Hom-associative algebras (Hom-algebras in short) \cite{MakhSilv08}, Hom-bialgebras \cite{MakhSilv10-II}, and Hom-Hopf algebras \cite{MakhSilv09} followed shortly.

However, although the original definition of a Hom-bialgebra $\C{B}$ involve two homomorphisms; an $\a:\C{B} \to \C{B}$ twisting the associativity of the multiplication, and a $\b:\C{B} \to \C{B}$ twisting the coassociativity of the comultiplication, the study of these objects seems to take two certain directions: one is towards the equality $\b = \a$ of the structure maps, and the other is the course towards $\b = \a^{-1}$. As such, both directions exclude the most fundamental ingredient; the universal enveloping Hom-Hopf algebras of Hom-Lie algebras. Given a Hom-Lie algebra $(\G{g},\a)$, its universal enveloping Hom-Hopf algebra is a bialgebra of type $(\a,\Id)$, that is, its coalgebra structure is coassociative. 

The central role of the Lie-Hopf algebras in the associative setting (just as the Connes-Moscovici Hopf algebras $\C{H}_n$ appearing in the computation of the local index theory of transverse elliptic operators), and in particular $\C{H}_n$'s being a Hopf algebra associated to the Lie algebra $W_n$, whose deformations ignited the Hom-associative objects, prompted us to investigate the construction of Hom-associative Hopf algebras out of Hom-Lie algebras. Needless to say, we had to develop all the machinery suitable for general Hom-bialgebras of type $(\a,\b)$, not only those of the type $(\a,\a)$, or of the type $(\a,\a^{-1})$.

More precisely, given a matched pair $(\G{g},\G{h})$ of Hom-Lie algebras, we observe that their universal enveloping Hom-Hopf algebras $\C{U}(\G{g})$ and $\C{U}(\G{h})$ form a matched pair of Hom-Hopf algebras. As such, we get a ``double cross product'' Hom-Hopf algebra $\C{U}(\G{g})\bowtie\C{U}(\G{h})$. Next, in view of the duality of Hom-Hopf algebras, we show that $\C{U}(\G{h})^\circ$ and $\C{U}(\G{g})$ form a mutual pair of Hom-Hopf algebras, leading to a ``bicrossproduct'' Hom-Hopf algebra $\C{U}(\G{h})^\circ \acl\C{U}(\G{g})$. These are the Hom-Hopf algebras that we refer as the Hom-Lie-Hopf algebras. Further study on Hom-Lie-Hopf algebras, such as the (co)representations, (co)homology, etc., is spared for sequential papers.

The paper is organized as follows.

In order to fix the terminology, and for the convenience of the reader, we review in Section \ref{sect-Hom-Hopf-algebras} the Hom-associative structures we shall need in the sequel. More precisely, we begin with Hom-algebras, and Hom-coalgebras in Subsection \ref{subsect-Hom-(co)algebras-duality}. What is also discussed in Subsection \ref{subsect-Hom-(co)algebras-duality} is the duality between Hom-algebras and Hom-coalgebras along the lines of \cite{Abe-book}. This allows us to interpret the dual as the Hom-associative analogue of the (classical)  representative functions. In Subsection \ref{subsect-Hom-Lie-Hom-Hopf-algebras}, on the other hand, we proceeded to recall Hom-bialgebras and Hom-Hopf algebras, as well as the Hom-Lie algebras with their universal enveloping Hom-Hopf algebras.

Section \ref{sect-Hom-Hopf-symmetries} is devoted to the Hom-associative analogues of the Hopf algebra symmetries. In Subsection \ref{subsect-module-alg} we review the Hom-module algebra symmetry, which is needed to construct the underlying Hom-algebra structure of a bicrossproduct Hom-Hopf algebra. In Subsection \ref{subsect-module-coalg}, on the other hand, we consider Hom-module coalgebras, which are the building blocks of double cross product Hom-Hopf algebras. Subsection \ref{subsect-comodule-alg} deals with the comodule algebra symmetry, that is, the essense of the Hom-Hopf-Galois extensions, and finally in Subsection \ref{subsect-comodule-coalg} we consider the Hom-comodule coalgebra symmetry for the underlying Hom-coalgebra structure of a bicrossproduct Hom-Hopf algebra.

In Section \ref{sect-Double-cross-product-Hom-Hopf} we shall study the double cross product Hom-Hopf algebras. To this end, Subsection \ref{subsect-double-cross-product-construction} is reserved for the matched pairs of Hom-Hopf algebras, and their double cross product. More precisely, we succeed to obtain the correct compatibility conditions for a matched pair of Hom-Hopf algebras of type $(\a,\b)$ in Proposition \ref{prop-matched-pair-double-cross-prod}. Then, focusing on the universal enveloping Hom-Hopf algebras of matched pairs of Hom-Lie algebras in Subsection \ref{subsect-univ-envlp-Hom-Hopf-Hom-Lie}, we show in Proposition \ref{prop-univ-envlp-mutual-pair} that the universal enveloping Hom-Hopf algebras of a matched pair of Hom-Lie algebras form a matched pair of Hom-Hopf algebras.

Finally, in Section \ref{sect-bicrossproduct-Hom-Hopf} we turn our attention to mutual pair of Hom-Hopf algebras, and their bicrossproducts. To this end, in Subsection \ref{subsect-bicrossproduct-construction} we review the bicrossproduct construction - for Hom-Hopf algebras of $(\a,\b)$-type. The underlying algebra structure of a bicrossproduct Hom-Hopf algebra is given in Proposition \ref{prop-bicrossprod-alg}, and then in Proposition \ref{prop-bicrossprod-coalg} the coalgebra structure of a bicrossproduct is obtained. Finally in Proposition \ref{prop-bicrossed-prod-Hom-Hopf} we achieve to obtain the general compatibility conditions for a mutual pair of Hom-Hopf algebras. Next, in Subsection \ref{subsect-semidualization}, reviewing the Hom-action / Hom-coaction duality, we prove the main result of the paper in Proposition \ref{prop-semidual}, giving way to the Hom-Lie-Hopf algebras of Corollary \ref{coroll-Hom-Lie-Hopf}.

\subsection*{Notation and Conventions} 

We shall use the Sweedler's notation for the comultiplication. More precisely, suppressing the summation, $\D(c) = c\ps{1}\ot c\ps{2}$ if the coalgebra structure is associative, and $\D(c) = c\ns{1}\ot c\ns{2}$ if it the coalgebra structure is Hom-associative. Similarly, we use $\nb(w) = w\ps{-1} \ot w\ps{0}$ and $\nb(w) = w\ps{0} \ot w\ps{1}$ for the associative (left and right resp.) coactions, and $\nb(w) = w\ns{-1} \ot w\ns{0}$ and $\nb(w) = w\ns{0} \ot w\ns{1}$ for the Hom-associative (again, left and right resp.) coactions. As for the multiplications, we use the symbols $\bullet$, $\star$, or $\ast$ for Hom-associative multiplications, while unadorned multiplications refer to the associative ones. Finally, we occasionally employed the notation $(a,b)$ for the tensor product $a\ot b$ in order to save some space in calculations.


\section{Hom-Hopf algebras}\label{sect-Hom-Hopf-algebras}

\subsection{Hom-algebras, Hom-coalgebras, duality}\label{subsect-Hom-(co)algebras-duality}

In this subsection we shall recall the notions of Hom-algebra, Hom-coalgebra, Hom-bialgebra, Hom-Hopf algebra and Hom-Lie algebra from \cite{MakhSilv08,MakhSilv09}.

\subsubsection{Hom-algebras} 

We discuss here the Hom-algebras, the basics examples of Hom-algebras, and Hom-modules.

\begin{definition}
A ``Hom-associative algebra'' (Hom-algebra in short) is defined to be a triple $(\C{A},\a,\mu)$, where $V$ is a vector space, $\mu : \C{A} \ot \C{A} \to \C{A}$ and $\a:\C{A}\to \C{A}$ are linear maps satisfying 
\begin{equation} \label{Hom-Assoc}
\mu (\alpha (x) \otimes \mu (y \otimes z))=\mu(\mu(x \otimes y)\otimes \alpha(z)),
\end{equation}
for any $x,y,z \in \C{A}$.
\end{definition}
A Hom-associative algebra $(\C{A},\mu,\a)$ is called ``multiplicative'' if the linear map $\a:\C{A}\to \C{A}$ respects $\mu:\C{A}\ot \C{A}\to \C{A}$ (\ie multiplicative).

The equality \eqref{Hom-Assoc} is called the ``Hom-associativity condition'', and it is equivalent to the commutativity of the diagram
\[
\xymatrix{
\C{A} \ot \C{A} \ot \C{A} \ar[d]_{\a\ot \mu} \ar[r]^{\mu\ot\a} & \C{A}\ot \C{A} \ar[d]^{\mu} \\
\C{A}\ot \C{A} \ar[r]_{\mu} & \C{A}.
}
\]

The fundamental example is one that indicates that the Hom-algebras are deformations of the associative algebras.

\begin{example}\label{ex-assoc-alg-to-hom-assoc-alg}
Any associative algebra $A$, and an algebra homomorphism $\a:A\to A$, give rise to a Hom-algebra $(A_\a,\mu\circ\a,\a)$, where $A_\a:=A$, and $\mu:A\ot A\to A$ is the (associative) multiplication on $A$, \cite[Coroll. 2.6(1)]{Yau09}.
\end{example}

A Hom-associative algebra of the form given in Example \ref{ex-assoc-alg-to-hom-assoc-alg} above is called a ``twist'' in \cite{Gohr10}. On the other hand, a Hom-associative algebra $\C{A}$ is called ``strongly degenerate'' if there exists $a,b \in \C{A}$ such that $a\bullet x = b\bullet x$, and $x\bullet a = x\bullet b$ for all $x \in \C{A}$, \cite{Gohr10}. There is, then, the following result; \cite[Prop. 1]{Gohr10}.

\begin{proposition}\label{twist-strong-degenerate}
Let $(\C{A},\a)$ be a Hom-associative algebra so that $\a:\C{A} \to \C{A}$ is surjective. Then, $\C{A}$ is either a twist, or strongly degenerate.
\end{proposition}

Further terminology on Hom-algebras is in order. Following \cite{Laur-GengMakhTele18}, a Hom-algebra $(\C{A},\mu,\a)$ is called ``unital'' is there is a homomorphism $\eta:k\to \C{A}$ such that the diagrams
\begin{equation}\label{Hom-unit}
\xymatrix{
k\ot \C{A} \ar[dr]_{\a}\ar[r]^{\eta\ot\Id} & \C{A}\ot \C{A} \ar[d]_{\mu} &\ar[l]_{\Id\ot\eta}\ar[dl]^\a \C{A} \ot k & & \C{A}\ar[r]^\a & \C{A} \\
 & \C{A} & & & k\ar[u]^\eta\ar[ur]_\eta & 
}
\end{equation}
are commutative. Given two (unital) Hom-algebras $(\C{A},\mu,\eta,\a)$ and $(\C{A}',\mu',\eta',\a')$, a linear map $f:\C{A}\rightarrow \C{A}'$ is a morphism of Hom-algebras if the diagrams
\[
\xymatrix{
\C{A}\ot \C{A} \ar[d]_\mu\ar[r]^{f\ot f} & \C{A}'\ot \C{A}' \ar[d]^{\mu'} & &\C{A}\ar[r]^f &\C{A}'  & &\C{A}\ar[r]^f \ar[d]_\a& \C{A}'\ar[d]^{\alpha '}  \\
\C{A} \ar[r]_f & \C{A}' & & k\ar[u]^\eta\ar[ur]_{\eta'} & & & \C{A}\ar[r]_f & \C{A}'
}
\]
are commutative. 

\begin{remark}
If $(\C{A},\mu,\eta,\a)$ is a unital Hom-algebra, then
\begin{align*}
& \mu(\a(a) \ot \a(a')) = \mu(\a(a) \ot \mu(a' \ot \eta(1))) = \mu(\mu(a \ot a') \ot \a(\eta(1))) = \\
& \mu(\mu(a \ot a') \ot \eta(1)) = \a(\mu(a \ot a')),
\end{align*}
for any $a,a' \in \C{A}$.
\end{remark}

\begin{definition}
Let $(\C{A},\mu,\eta,\a)$ be a unital Hom-algebra. An element $a \in \C{A}$ is said to be ``Hom-invertible'' if there is an $x^{-1}\in \C{A}$, and an $n \in \B{N}$ such that 
\begin{equation}\label{eqn-Hom-inv}
(\a^n\circ \mu) (x \ot x^{-1}) = (\a^n\circ \mu) (x^{-1} \ot x) = \eta(1).
\end{equation}
\end{definition}

The tensor product of Hom-algebras is considered as follows; if $(\C{A},\mu,\a)$ and $(\C{A}',\mu',\a')$ are two Hom-algebras, then their tensor product is also a Hom-algebra with the structure given by $(\C{A}\ot\C{A}',(\mu\ot\mu')\circ(\Id \ot\tau\ot\Id),\a\ot\a')$, where
\[
\tau: \C{A}' \ot \C{A} \to \C{A} \ot \C{A}', \qquad \tau(a'\ot a) := a \ot a'
\]
is the switch morphism. Finally, let us recall the notion of a (left) module over a Hom-algebra. 

\begin{definition}
Let $(\C{A},\mu,,\eta,\a)$ be a unital Hom-algebra, and $M$ a vector space. Let also $\rho:\C{A} \ot M \to M$ and $\gamma:M\to M$ be linear maps. Then the triple $(M,\rho,\gamma)$ is called a ``(left) Hom-module over the Hom-algebra $\C{A}$'' if the diagrams
\[
\xymatrix{
\C{A}\ot \C{A}  \ot M \ar[d]_{\a\ot \rho} \ar[r]^{\,\,\,\,\,\,\,\,\mu\ot \gamma} & \C{A} \ot M \ar[d]^{\rho}  && k\ot M \ar[r]^{\eta\ot \Id} \ar[d]_{\g} & \C{A}\ot M \ar[dl]^{\rho}\\
\C{A} \ot M \ar[r]_{\rho} & M &&M &
}
\]
are commutative.
\end{definition}
A ``right Hom-module over the Hom-algebra $\C{A}$'' is defined similarly. As expected, any Hom-algebra $(\C{A},\mu,\a)$ is both a left and a right Hom-module over itself.

\begin{definition}
Let $(\C{A},\mu, \eta,\a)$ be a (unital) Hom-algebra, and $(M,\rho,\g)$, $(M',\rho',\g')$ be two (left) Hom-modules over $\C{A}$. Then a map $f:M\to M'$ is called a ``morphism of Hom-modules'', if $f\circ\g = \g' \circ f$, and
\[
\xymatrix{
\C{A}  \ot M \ar[d]_{\Id\ot f} \ar[r]^{\,\,\,\,\,\,\,\,\rho} &  M \ar[d]^f \\
\C{A} \ot M' \ar[r]_{\rho'} & M'
}
\]
is commutative.
\end{definition}

\subsubsection{Hom-coalgebras}

We now recall Hom-coalgebras, basic examples of Hom-coalgebras, and Hom-comodules.

\begin{definition}
A ``Hom-coassociative coalgebra'' (Hom-coalgebra in short) is a triple $(\C{C},\D,\b)$ consisting of a linear space $\C{C}$, and linear maps $\Delta:\C{C}\to \C{C}\ot \C{C}$ and $\beta:\C{C}\to \C{C}$ such that the diagram
\begin{equation}\label{Hom-coassoc}
\xymatrix{
\C{C} \ot \C{C} \ot \C{C}   & \ar[l]_{\,\,\,\,\,\,\,\,\,\,\,\,\,\,\,\,\D\ot\b}\C{C}\ot \C{C}  \\
\C{C}\ot \C{C} \ar[u]^{\b\ot \D}  & \ar[l]^{\D} \C{C} \ar[u]_{\D} 
}
\end{equation}
is commutative.
\end{definition}
A Hom-coassociative coalgebra $\C{C}$ is called ``comultiplicative'' if $\b:\C{C} \to \C{C}$ is a coalgebra endomorphism (\ie comultiplicative).

Following \cite{Laur-GengMakhTele18}, a Hom-coalgebra $(\C{C},\D,\b)$ is called ``counital'' if there is $\ve:\C{C}\to k$ such that the diagrams 
\begin{equation}\label{Hom-counit}
\xymatrix{
 k\ot \C{C} & \ar[l]_{\ve\ot \Id}\C{C}\ot \C{C}\ar[r]^{\Id\ot\ve} & \C{C}\ot k & & \C{C}\ar[r]^\b\ar[rd]_\ve & \C{C}\ar[d]^\ve \\
  & \ar[ul]^\b \C{C}\ar[u]^{\D}\ar[ur]_{\b} & & &  & k
}
\end{equation}
are commutative. Given two (counital) Hom-coalgebras $(\C{C},\D,\ve,\b)$ and $(\C{C}',\D',\ve',\b')$, a linear map $f:\C{C} \rightarrow \C{C}'$ is a morphism of Hom-coalgebras if the diagrams
\[
\xymatrix{
\C{C}\ot \C{C} \ar[r]^{f\ot f} & \C{C}'\ot \C{C}'  & &\C{C}\ar[d]_\ve\ar[r]^f &\C{C}' \ar[dl]^{\ve'} & & \C{C}\ar[r]^f\ar[d]_\b & \C{C}' \ar[d]^{\b'} \\
\C{C}\ar[u]^\D \ar[r]_f & \C{C}'\ar[u]_{\D'} & & k & & & \C{C}\ar[r]_f & \C{C}'
}
\]
are commutative. As for the tensor product of Hom-coalgebras, let us note that given two (counital) Hom-coalgebras $(\C{C},\D,\ve,\b)$ and $(\C{C}',\D',\ve',\b')$, we have a (counital) Hom-coalgebra $(\C{C}\ot \C{C}', (\Id\ot \tau \ot \Id)\circ (\D\ot \D'),\ve\ot\ve',\b\ot\b')$, where
\[
\tau: \C{C} \ot \C{C}' \to \C{C}' \ot \C{C}, \qquad \tau(c\ot c') := c' \ot c
\]
is the switch morphism.

Just as it was the case for the Hom-algebras, the first example we recall implies that Hom-coalgebras are deformations of coassociative counital coalgebras.

\begin{example}
Let $(C,\D,\ve)$ be a coassociative counital coalgebra, and $\b:C\to C$ be a coalgebra morphism. Then for $C^\b:=C$, $(C^\b, \D\circ\b, \ve\circ \b, \b)$ is a coassociative Hom-coalgebra, \cite[Thm. 3.16]{MakhSilv10-II}.
\end{example}

Let us next recall the notion of a Hom-comodule over a Hom-coalgebra, see for instance \cite{MakhSilv10-II}.

\begin{definition}
Let $(\C{C},\D,\b)$ be a Hom-coalgebra, and $M$ a vector space. Let also $\nb:M\to M\ot \C{C}$ and $\t:M\to M$ be linear maps. Then the triple $(M,\nb,\t)$ is called a ``right Hom-comodule over the Hom-coalgebra $\C{C}$'' if the diagrams
\begin{equation}\label{Hom-comodule-compt}
\xymatrix{
M\ot \C{C} \ot \C{C} &\ar[l]_{\,\,\,\,\,\,\,\,\,\nb\ot \b} M\ot \C{C} && M\ot \C{C} \ar[r]^{\Id\ot \ve}&M\ot k\\
M\ot \C{C} \ar[u]^{\t\ot \D} & \ar[l]^\nb M\ar[u]_\nb &&M\ar[u]^\nb \ar[ur]_\t &
}
\end{equation}
are commutative.
\end{definition}
The left Hom-comodule over the Hom-coalgebra $\C{C}$ is defined similarly. Parallel to the case of Hom-modules, any Hom-coalgebra $(\C{C},\D,\b)$ is both a left and a right Hom-comodule over itself.

\begin{definition}
Let $(\C{C},\D,\ve,\b)$ be a (counital) Hom-coalgebra, and $(M,\nb,\t)$, $(M',\nb',\t')$ be two (right) Hom-comodules over $\C{C}$. Then a map $f:M\to M'$ is called a ``morphism of Hom-comodules'', if $f \circ \t = \t' \circ f$, and
\[
\xymatrix{
M'\ot \C{C} &\ar[l]_{\,\,\,\,\,\,\,\,\,\nb'} M'\\
M\ot \C{C} \ar[u]^{f\ot \Id} & \ar[l]^\nb M\ar[u]_f
}
\]
is commutative.
\end{definition}

\subsubsection{Duality}

Let us now discuss briefly the duality between Hom-algebras and Hom-coalgebras.

To this end, let us recall that given a (coassociative) coalgebra $(C,\D)$, its algebraic dual $C^\ast := \Hom(C,k)$ is an (associative) algebra via the convolution multiplication, namely;
\[
(f \ast g)(c) := f(c\ps{1})g(c\ps{2}),
\]
for any $f,g \in C^\ast$, and any $c \in C$. Moreover, if coalgebra $(C,\D,\ve)$ is counital, then the dual algebra $C^\ast$ is unital, the unit element being the counit $\ve:C\to k$, that is,
\[
(f \ast \ve)(c) := f(c\ps{1})\ve(c\ps{2}) = f(c) = \ve(c\ps{1})f(c\ps{2}) = (\ve\ast f)(c),
\]
for any $f \in C^\ast$, and any $c \in C$, see for instance \cite[Subsect. 2.1.2]{Abe-book}.

On the other hand, starting with an (associative) algebra $(A,\mu)$, the transpose
\[
\mu^\ast:A^\ast \to (A\ot A)^\ast
\]
of the multiplication map $\mu:A\ot A\to A$ does not in general factors through $A^\ast \ot A^\ast$. However, it does on the restricted dual
\[
A^\circ = \{f\in \Hom(A,k)\mid \exists \, I\subseteq \ker(f), \,\, \text{\rm such that}\,\, I\trianglelefteq A,\,\,\dim(A/I)<\infty\},
\]
see \cite[Thm. 2.2.12]{Abe-book}, that is,
\[
\mu^\ast:A^\circ \to A^\circ \ot A^\circ \subseteq (A\ot A)^\ast, \qquad \mu^\ast(f) =: f\ps{1}\ot f\ps{2},
\]
such that
\[
(f\ps{1}\ot f\ps{2})(a\ot a') := f(aa').
\]
Let us note also that if $\dim(A)<\infty$, then $A^\circ = A^\ast$. We shall now investigate these dualities for Hom-algebras, and Hom-coalgebras.

As for the former, the result goes verbatim on the level of Hom-algebras. We record below a slightly different version of \cite[Prop. 3.20]{MakhSilv10-II}.

\begin{proposition}\label{prop-Hom-C-A-duality}
Let $(\C{C},\D,\b)$ be a Hom-coalgebra, so that $\b:\C{C}\to \C{C}$ is an isomorphism of Hom-coalgebras, and $(\C{A},\mu,\a)$ a Hom-algebra. Then  $\Hom(\C{C},\C{A})$ is a Hom-algebra with the convolution product given by
\[
\mu^\star(f \ot g)(c) := \mu(f(\b^{-2}(c\ns{1}))\ot g(\b^{-2}(c\ns{2}))),
\]
and the homomophism
\[
\a^\star(f) := \a \circ f\circ\b^{-1},
\]
for any $f,g \in \Hom(\C{C},\C{A})$, and any $c \in \C{C}$. Moreover, if $(\C{C},\D,\ve,\b)$ is a counital Hom-coalgebra, and $(\C{A},\mu,\eta,\a)$ is a unital Hom-algebra, then $\Hom(\C{C},\C{A})$ is a unital Hom-algebra with the unit 
\[
\eta^\star:k \to \Hom(\C{C},\C{A}), \qquad \eta^\star(1) := \eta \circ \ve.
\]
\end{proposition}

\begin{proof}
For any $f,g,h \in \Hom(\C{C},\C{A})$, and any $c \in \C{C}$,
\begin{align*}
& \mu^\star (\alpha^\star (f) \otimes \mu^\star (g \otimes h)) (c) = \mu(\alpha^\star (f)(\b^{-2}(c\ns{1})) \ot \mu^\star (g \otimes h)(\b^{-2}(c\ns{2}))) = \\
& \mu(\a(f(\b^{-3}(c\ns{1}))) \ot \mu( g(\b^{-4}(c\ns{2}\ns{1})) \ot h(\b^{-4}(c\ns{2}\ns{2})))) = \\
& \mu(\a(f(\b^{-4}(c\ns{1}\ns{1}))) \ot \mu( g(\b^{-4}(c\ns{1}\ns{2})) \ot h(\b^{-3}(c\ns{2})))) = \\
& \mu(\mu(f(\b^{-4}(c\ns{1}\ns{1})) \ot  g(\b^{-4}(c\ns{1}\ns{2}))) \ot \a(h(\b^{-3}(c\ns{2})))) = \\
& \mu^\star(\mu^\star(f \otimes g)\otimes \alpha^\star(h))(c),
\end{align*}
where the second equality follows from the assumption, the third equality follows from the commutativity of \eqref{Hom-coassoc}, and the fourth equality is nothing but the Hom-associativity condition \eqref{Hom-Assoc}. As a result, $(\Hom(\C{C},\C{A}),\mu^\star,\a^\star)$ is a Hom-algebra. We now verify the condition for the unit, \ie the commutativity of the diagrams \eqref{Hom-unit}. We have,
\begin{align*}
& \mu^\star(\eta^\star(1) \ot f)(c) = \mu^\star(\eta \circ \ve \ot f)(c) = \mu(\eta(\ve(\b^{-2}(c\ns{1})))\ot f(\b^{-2}(c\ns{2}))) = \\
& \mu(\eta(1)\ot f(\b^{-1}(c))) = \a(f(\b^{-1}(c))) = \a^\star(f)(c),
\end{align*}
where the third equality is the counitality, \ie the commutativity of \eqref{Hom-counit}, and the last equality is the unitality, \ie the commutativity of \eqref{Hom-unit}. Hence, $(\Hom(\C{C},\C{A}),\mu^\star,\eta^\star,\a^\star)$ is a unital Hom-algebra.
\end{proof}

We thus conclude the following.

\begin{corollary}\label{coroll-dual-hom-alg}
For any Hom-coalgebra $(\C{C},\D,\b)$, where $\b:\C{C}\to \C{C}$ is an isomorphism of Hom-coalgebras, the algebraic dual $\C{C}^\ast:=\Hom(\C{C},k)$ is a Hom-algebra with the product given by
\begin{equation}\label{dual-multp-star}
\mu^\star(f \ot g)(c) := f(\b^{-2}(c\ns{1}))g(\b^{-2}(c\ns{2})),
\end{equation}
and the homomophism
\[
\a^\star(f) := f\circ\b^{-1},
\]
for any $f,g \in \C{C}^\ast$, and any $c \in \C{C}$. Moreover, if $(\C{C},\D,\ve,\b)$ is counital, then $\C{C}^\ast$ is a unital Hom-algebra with the unit 
\[
\eta^\star:k \to \C{C}^\ast, \qquad \eta^\star(1) := \ve.
\]
\end{corollary}

As an immediate application, we shall consider the following analogue of Proposition \ref{twist-strong-degenerate} in order to shed further light to the assumption in the hypothesis of Proposition \ref{prop-Hom-C-A-duality}.

\begin{proposition}
Let $(\C{C},\D_\C{C},\b)$ be a Hom-coassociative coalgebra so that $\b:\C{C} \to \C{C}$ is an isomorphism of Hom-coalgebras. Then $\C{C}$ is either a ``cotwist''; \ie $\C{C} = C^{\b'}$ for some coalgebra isomorphism $\b':C\to C$ over a coassociative coalgebra $(C,\D)$, or it is ``strongly codegenerate''; \ie there is $f,g \in \C{C}^\ast$ with $f \neq g$ such that $f(\b^{-2}(c\ns{1}))\b^{-2}(c\ns{2}) = g(\b^{-2}(c\ns{1}))\b^{-2}(c\ns{2})$, and $\b^{-2}(c\ns{1})f(\b^{-2}(c\ns{2})) = \b^{-2}(c\ns{1})g(\b^{-2}(c\ns{2}))$ for any $c \in \C{C}$.
\end{proposition}

\begin{proof}
If $(\C{C},\D_\C{C},\b)$ is a Hom-coassociative coalgebra so that $\b:\C{C} \to \C{C}$ is bijective, then (in view of Corollary \ref{coroll-dual-hom-alg}) the dual $\C{C}^\ast$ is a Hom-associative algebra with (the surjective) $(\b^{-1})^\ast:\C{C}^\ast\to \C{C}^\ast$.

Now, if $(\C{C}^\ast,(\b^{-1})^\ast)$ is not strongly codegenerate, then for any $f,g \in \C{C}^\ast$ with $f(\b^{-2}(c\ns{1}))\b^{-2}(c\ns{2}) = g(\b^{-2}(c\ns{1}))\b^{-2}(c\ns{2})$, or $\b^{-2}(c\ns{1})f(\b^{-2}(c\ns{2})) = \b^{-2}(c\ns{1})g(\b^{-2}(c\ns{2}))$ for any $c \in \C{C}$ implies $f = g$. But then, $\mu^\star(f \ot h) = \mu^\star(g \ot h)$ or $\mu^\star(h \ot f) = \mu^\star(h \ot g)$ for any $h \in \C{C}^\ast$ implies $f=g$; that is, $(\C{C}^\ast,(\b^{-1})^\ast)$ is not strongly degenerate. As such, it follows from Proposition \ref{twist-strong-degenerate} (more precisely, from the proof of \cite[Prop. 1]{Gohr10}) that $\C{C}^\ast$ is a twist; \ie $\C{C}^\ast = C^\ast_{(\b^{-1})^\ast}$. In other words,
\[
\mu^\star(f \ot g) = (\b^{-1})^\ast(f)(\b^{-1})^\ast(g),
\]
and hence
\[
\mu^\star(f \ot g)(c) = (\b^{-1})^\ast(f)(c\ps{1})(\b^{-1})^\ast(g)(c\ps{2}) = f(\b^{-1}(c\ps{1}))g(\b^{-1}(c\ps{2}))
\]
for some $\D:C\to C$ with $\D(c):= c\ps{1} \ot c\ps{2}$. It then follows from $(\b^{-1})^\ast:C^\ast \to C^\ast$ being an isomorphism of algebras that $\b^{-1}:C\to C$ satisfies $\D\circ \b = (\b \ot \b)\circ \D$, and similarly it follows from the associativity $f(gh) = (fg)h$ of $C^\ast$ for any $f,g,h \in C^\ast$ that (when evaluated over a $c\in C$)
\[
f(c\ps{1}\ps{1})g(c\ps{1}\ps{2})h(c\ps{2}) = f(c\ps{1})g(c\ps{2}\ps{1})h(c\ps{2}\ps{2}),
\]
that is, 
\[
c\ps{1}\ps{1}\ot c\ps{1}\ps{2}\ot c\ps{2} = c\ps{1}\ot c\ps{2}\ps{1}\ot c\ps{2}\ps{2},
\]
or equivalently, $(C,\D)$ is coassociative, and $\D_\C{C} = \D\circ\b$.
\end{proof}

As for the restricted dual, we shall mimic \cite[Sect. 2.2]{Abe-book}. To this end, we keep in mind that for a Hom-algebra $(\C{A},\mu,\a)$, in general, $\C{A}^\ast:=\Hom(\C{A},k)$ is not necessarily a left (or right) Hom-$\C{A}$-module via the coregular actions, \cite{HassShapSutl15}. Instead, we shall adopt the following point of view.

\begin{proposition}
Let $(\C{A},\mu,\a)$ be a multiplicative Hom-algebra. Then, the pair $(\C{A}^\ast,(\a^{-1})^\ast)$ is a left Hom-$\C{A}$-module via
\begin{equation}\label{left-action-on-A-star}
\rho_L:\C{A}\ot \C{A}^\ast \to \C{A}^\ast, \qquad \rho_L(a\ot f)(a') := f(\a^{-2}(\mu(a'\ot a))),
\end{equation}
and is a right Hom-$\C{A}$-module via
\begin{equation}\label{right-action-on-A-star}
\rho_R:\C{A}^\ast\ot \C{A} \to \C{A}^\ast, \qquad \rho_R(f\ot a)(a') := f(\a^{-2}(\mu(a\ot a'))).
\end{equation}
\end{proposition}

\begin{proof}
Let us use, for the sake of simplicity, the notation $a \rt f := \rho_L(a\ot f)$, for any $a\in \C{A}$, and any $f\in \C{A}^\ast$. We see at once that
\begin{align*}
& (\a(a) \rt (a' \rt f)) (a'') = (a' \rt f)(\a^{-2}(a''\bullet \a(a))) = (a' \rt f)(\a^{-2}(a'')\bullet\a^{-1}( a)) = \\
& f(\a^{-2}([(\a^{-2}(a'')\bullet\a^{-1}( a)] \bullet a')) = f([(\a^{-4}(a'')\bullet\a^{-3}( a)] \bullet \a^{-2}(a')) = \\
& f(\a^{-3}(a'')\bullet[(\a^{-3}( a)\bullet \a^{-3}(a')] ) = f(\a^{-3}(a''\bullet ( a\bullet a')) ) = \\
& (\a^{-1})^\ast(f)(\a^{-2}(a''\bullet ( a\bullet a')) ) = ( a\bullet a') \rt ((\a^{-1})^\ast(f))(a''),
\end{align*}
for any $a,a',a'' \in \C{A}$, and any $f\in \C{A}^\ast$. The first claim is thus obtained. The latter follows similarly.
\end{proof}

Following \cite[Sect. 2.2]{Abe-book}, given any $f \in \C{A}^\ast$, let us now define
\begin{equation}\label{A-vee-f}
\C{A}^\ast_f := \{\rho_L(a\ot f) \mid a \in \C{A}\},
\end{equation}
and
\begin{equation}\label{f-A-vee}
{}_f\C{A}^\ast := \{\rho_R(f\ot a) \mid a \in \C{A}\}.
\end{equation}
The following is an analogue of \cite[Thm. 2.2.6]{Abe-book}.

\begin{proposition}\label{prop-dim-fin}
Let $(\C{A},\mu,\a)$ be a multiplicative Hom-algebra, and $f\in \C{A}^\ast$ be given. Let also $\C{A}^\ast_f$ and ${}_f\C{A}^\ast$ be as in \eqref{A-vee-f} and \eqref{f-A-vee} respectively. Then $\dim(\C{A}^\ast_f) < \infty$ if and only if $\dim({}_f\C{A}^\ast) < \infty$.
\end{proposition}

\begin{proof}
Let $\dim(\C{A}^\ast_f) < \infty$, and let $\{f_1,f_2,\ldots,f_n\}$ be a basis. As such, for any $a\in \C{A}$,
\[
\rho_L(a\ot f) = \sum_{i=1}^ng_i(a)f_i,
\]
for some $g_i(a) \in k$, $1 \leq i \leq n$. Then, for any $a' \in \C{A}$,
\begin{align*}
& \rho_R(f\ot a') (a) = f(\a^{-2}(a'\bullet a)) = \rho_L(a\ot f) (a') =   \\
&  \sum_{i=1}^ng_i(a)f_i(a') = \left(\sum_{i=1}^nf_i(a')g_i\right)(a),
\end{align*}
from which we conclude that ${}_f\C{A}^\ast \subseteq {\rm Span}(\{g_1,\ldots, g_n\})$, and hence $\dim({}_f\C{A}^\ast) < \infty$. The converse implication is similar.
\end{proof}

Keeping in mind the discussion above on the tensor products of Hom-algebras (being Hom-algebras), we now consider a linear map
\[
\pi: \C{A}^\ast \ot \C{A}^\ast \to (\C{A}\ot \C{A})^\ast, \qquad \pi(f \ot g)(a\ot a') :=f(a)g(a'),
\]
which is injective. Indeed, for any $\sum_{i=1}^n g_i \ot f_i \in \C{A}^\ast \ot \C{A}^\ast$, without loss of generality $\{f_1,\ldots,f_n\}$ being linearly independent, such that $\pi(\sum_{i=1}^n g_i \ot f_i) = 0$, we have
\[
\pi\left(\sum_{i=1}^n g_i \ot f_i\right)(a\ot a') = \sum_{i=1}^n g_i(a) f_i(a') = \left(\sum_{i=1}^n g_i(a) f_i\right)(a') = 0,
\]
for any $a,a' \in \C{A}$. As a result, $\sum_{i=1}^n g_i(a) f_i = 0$, hence the linear independence of $\{f_1,\ldots,f_n\}$ implies $g_i(a) = 0$ for $1\leq i\leq n$. Hence $g_i = 0$ for $1\leq i\leq n$, which, in turn, results in $\sum_{i=1}^n g_i \ot f_i = 0$.

Let us next define
\begin{equation}\label{delta}
\d: \C{A}^\ast \to (\C{A}\ot \C{A})^\ast, \qquad \d(f)(a\ot a') := f(\a^{-2}(a\bullet a')) = f(\a^{-2}(a)\bullet \a^{-2}(a')).
\end{equation}
The following is an analogue of \cite[Thm. 2.2.7]{Abe-book}.

\begin{proposition}\label{prop-delta-pi}
Let $(\C{A},\mu,\a)$ be a multiplicative Hom-algebra, and $f\in \C{A}^\ast$ be given. Then, $\d(f) \in \pi(\C{A}^\ast \ot \C{A}^\ast)$ if and only if $\dim(\C{A}^\ast_f) <\infty$.
\end{proposition}

\begin{proof}
Let $\d(f) \in \pi(\C{A}^\ast \ot \C{A}^\ast)$, that is, let there be $f_i,g_i \in \C{A}^\ast$, $1 \leq i\leq n$, such that $\pi(\sum_{i=1}^n\,f_i\ot g_i) = \d(f)$. Since $\d(f)(a \ot a') = \rho_L(a'\ot f)(a)$, we arrive at 
\[
\rho_L(a'\ot f)(a) = \pi(\sum_{i=1}^n\,f_i\ot g_i)(a\ot a') = \sum_{i=1}^n\,f_i(a) g_i(a') = \left(\sum_{i=1}^n g_i(a')f_i\right)(a),
\]
in other words, $\C{A}^\ast_f \subseteq {\rm Span}(\{f_1,\ldots, f_n\})$, and hence $\dim(\C{A}^\ast_f) < \infty$.

Conversely, let $\dim(\C{A}^\ast_f) <\infty$ with a basis $\{f_1,\ldots, f_n\}$. Then for any $a'\in \C{A}$, $\rho_L(a'\ot f) = \sum_{i=1}^n\,g_i(a')f_i$, for some $f_i \in \C{A}^\ast$, $1\leq i \leq n$. Hence, for any $a\in \C{A}$,
\[
\d(f)(a\ot a') = \rho_L(a'\ot f)(a) = \sum_{i=1}^n\,g_i(a')f_i(a) = \pi(\sum_{i=1}^nf_i \ot g_i )(a\ot a').
\]
\end{proof}

Finally, the notion of the restricted dual appears in the following analogue of \cite[Thm. 2.2.8]{Abe-book}.

\begin{proposition}\label{prop-delta-pi-circ}
Let $(\C{A},\mu,\a)$ be a multiplicative Hom-algebra, and let
\[
\C{A}^\circ := \{f\in \C{A}^\ast \mid \dim(\C{A}^\ast_f) <\infty\}.
\]
Then, for any $f\in \C{A}^\circ$, $\d(f) \in \pi(\C{A}^\circ \ot \C{A}^\circ)$.
\end{proposition}

\begin{proof}
Given $f\in \C{A}^\circ$, and hence having $\dim(\C{A}^\ast_f) < \infty$, let $\{f_1,\ldots,f_n\}$ be a basis for $\C{A}^\ast_f$. As a result, there are $g_1,\ldots,g_n \in \C{A}^\ast$ so that for any $a'\in \C{A}$
\[
\rho_L(a'\ot f) = \sum_{i=1}^n\,g_i(a')f_i.
\]

On the other hand, it follows from \eqref{delta} that
\[
\d(f) (a \ot a') = \rho_L(a'\ot f)(a) = \sum_{i=1}^n\,g_i(a')f_i(a) = \pi\left(\sum_{i=1}^n\,f_i \ot g_i\right) (a \ot a').
\]
Now, noticing from \eqref{left-action-on-A-star} and \eqref{right-action-on-A-star} that 
\[
\rho_L(a'\ot f)(a) = \sum_{i=1}^n\,g_i(a')f_i(a) = \rho_R(f\ot a)(a'),
\]
and considering $\{a_1,\ldots,a_n\} \subseteq \C{A}$ such that $f_i(a_j) = \d_{ij}$, see for instance \cite[Lemma 2.2.9]{Abe-book} and \cite[Lemma 1.1]{Hochschild-book}, we obtain $g_j = \rho_R(f \ot a_j) \in \C{A}^\ast$, $1\leq j \leq n$. Furthermore, for any $b\in \C{A}$, we see at once that
\[
\rho_R(\rho_R(f\ot a) \ot b) = \rho_R((\a^{-1})^\ast(f) \ot (a \bullet\a^{-1}(b))) = \sum_{i=1}^n\,f_i((\a^{-1}(b)\bullet a)g_i\circ \a^{-1},
\]
that is, $\rho_R(\rho_R(f\ot a) \ot b) \subseteq {\rm Span}(g_1\circ \a^{-1}, \ldots, g_n\circ \a^{-1})$. In other words, $\dim{}_{\rho_R(f\ot a)}\C{A}^\ast < \infty$, for any $a \in \C{A}$. It then follows from Proposition \ref{prop-dim-fin} that $\dim\C{A}^\ast_{\rho_R(f\ot a)} < \infty$, for any $a \in \C{A}$, that is, $\rho_R(f\ot a) \in \C{A}^\circ$ for any $a\in \C{A}$. In particular, we thus obtain $g_j = \rho_R(f \ot a_j) \in \C{A}^\circ$ for $1\leq j \leq n$. Summing up, we conclude
\begin{equation}\label{delta-f-I}
\d(f) = \pi\left(\sum_{i=1}^n\,f_i \ot g_i\right) \in \pi(\C{A}^\ast \ot \C{A}^\circ).
\end{equation}
Similarly we may obtain
\begin{equation}\label{delta-f-II}
\d(f)  \in \pi(\C{A}^\circ \ot \C{A}^\ast).
\end{equation}
Finally, \eqref{delta-f-I} and \eqref{delta-f-II} together yields the claim.
\end{proof}

As a result, we now have the restricted dual Hom-coalgebra.

\begin{corollary}\label{coroll-dual-hom-coalg}
Given a Hom-algebra $(\C{A},\mu,\a)$, its ``restricted dual'' $(\C{A}^\circ,\pi^{-1}\circ\d,(\a^{-1})^\ast )$ is a Hom-coalgebra. If $(\C{A},\mu,\eta,\a)$ is a unital Hom-algebra, then $(\C{A}^\circ,\pi^{-1}\circ\d,\eta^\ast,(\a^{-1})^\ast )$ is a counital Hom-coalgebra.
\end{corollary}

\begin{proof}
Let us first note that Proposition \ref{prop-delta-pi-circ} yields the coaction $\pi^{-1}\circ\d:\C{A}^\circ \to \C{A}^\circ \ot \C{A}^\circ$. Let us verify the commutativity of \eqref{Hom-coassoc}. Indeed, for any $f \in \C{A}^\circ$, setting $(\pi^{-1}\circ\d)(f) := f\ns{1}\ot f\ns{2}$, on one hand we have
\begin{equation}\label{(a-ot-D)D}
((\a^{-1})^\ast \ot \pi^{-1}\circ\d)(\pi^{-1}\circ\d)(f) = (\a^{-1})^\ast (f\ns{1}) \ot f\ns{2}\ns{1} \ot f\ns{2}\ns{2},
\end{equation}
and on the other hand,
\begin{equation}\label{(D-ot-a)D}
(\pi^{-1}\circ\d \ot (\a^{-1})^\ast )(\pi^{-1}\circ\d)(f) = f\ns{1}\ns{1} \ot f\ns{1}\ns{2} \ot (\a^{-1})^\ast (f\ns{2}).
\end{equation}
Evaluating both \eqref{(a-ot-D)D} and \eqref{(D-ot-a)D} on an arbitrary $a\ot a'\ot a'' \in \C{A}\ot \C{A} \ot \C{A}$, we see that
\begin{align*}
& \left((\a^{-1})^\ast (f\ns{1}) \ot f\ns{2}\ns{1} \ot f\ns{2}\ns{2}\right)(a\ot a'\ot a'') = f\ns{1}(\a^{-1}(a))f\ns{2}(\a^{-2}(a' \bullet a'')) = \\
&  f(\a^{-3}(a)\bullet [\a^{-4}(a')\bullet \a^{-4}(a'')]) = f([\a^{-4}(a)\bullet \a^{-4}(a')]\bullet \a^{-3}(a''))\\
& f\ns{1}(\a^{-2}(a\bullet a')f\ns{2}(\a^{-1}(a'')) = \left(f\ns{1}\ns{1} \ot f\ns{1}\ns{2} \ot (\a^{-1})^\ast(f\ns{2})\right)(a\ot a'\ot a''),
\end{align*}
where the third equality is a result of \eqref{Hom-Assoc}. As for the counitality, we observe that
\begin{align*}
& (\Id\ot\eta^\ast)((\pi^{-1}\circ\d)(f))(a) = f\ns{1}(a)\eta^\ast(f\ns{2}) = \\ 
& f(\a^{-2}(a\bullet \eta(1))) = f(\a^{-2}(\a(a))) = f(\a^{-1}(a)) = (\a^{-1})^\ast(f)(a),
\end{align*}
and that
\[
(\eta^\ast \circ(\a^{-1})^\ast)(f) = \eta^\ast((\a^{-1})^\ast(f)) = (\a^{-1})^\ast(f)(\eta(1)) = f(\a^{-1}(\eta(1))) = f(\eta(1)) = \eta^\ast(f),
\]
hence the claim.
\end{proof}

\subsection{Hom-Lie algebras and Hom-Hopf algebras}\label{subsect-Hom-Lie-Hom-Hopf-algebras}

\subsubsection{Hom-Hopf algebras}

We now put together the notions of Hom-algebras and Hom-coalgebras, with necessary compatibility conditions, to arrive at Hom-bialgebras and Hom-Hopf algebras.

\begin{definition}\label{def-Hom-bialg-Hom-Hopf}
A ``Hom-bialgebra'' is a 5-tuple $(\C{B},\mu, \alpha, \Delta, \beta)$, where
\begin{itemize}
\item[(i)] $(\C{B},\mu, \alpha)$ is a Hom-algebra,
\item[(ii)] $(\C{B},\Delta, \beta)$ is a Hom-coalgebra,
\item[(iii)] the Hom-coalgebra structure maps are Hom-algebra morphisms, \ie
\begin{enumerate}
\item $\D \circ \mu = (\mu \ot \mu)\circ(\Id \circ\tau\circ \Id) \circ(\D \ot \D)$, 
\item $\D\circ \a = (\a \ot \a) \circ\D$,
\item $\b \circ \mu = \mu \circ (\b \ot \b)$,
\item $\b \circ \a = \a \circ \b$.
\end{enumerate}
\end{itemize} 
\end{definition}

A (co)unital Hom-bialgebra is defined similarly. More explicitly, it is a 7-tuple $(\C{B},\mu, \eta, \alpha, \Delta,\ve, \beta)$, where
\begin{itemize}
\item[(i)] $(\C{B},\mu, \eta, \alpha)$ is a unital Hom-algebra,
\item[(ii)] $(\C{B},\Delta,\ve, \beta)$ is a counital Hom-coalgebra,
\item[(iii)] the Hom-coalgebra structure maps are unital Hom-algebra morphisms, \ie
\begin{enumerate}
\item $\D(\eta(1)) = \eta(1) \ot \eta(1)$, 
\item $\D \circ \mu = (\mu \ot \mu)\circ(\Id \circ\tau\circ \Id) \circ(\D \ot \D)$, 
\item $\D\circ \a = (\a \ot \a) \circ\D$,
\item $\ve\circ \eta = \Id_k$, 
\item $\ve\circ\mu = \ve\ot \ve$,
\item $\ve\circ\a = \ve$,
\item $\b \circ \eta = \eta$,
\item $\b \circ \mu = \mu \circ (\b \ot \b)$,
\item $\b \circ \a = \a \circ \b$.
\end{enumerate}
\end{itemize} 

\begin{example}
Let $(B,\mu,\eta,\D,\ve)$ be a bialgebra, and $\a:B\to B$, $\b:B\to B$ be two commuting ($\a\circ \b = \b \circ \a$) bialgebra homomorphisms. Then, for $B_\a^\b:=B$, the 7-tuple $(B,\a\circ\mu,\eta,,\a,\b\circ\D,\ve,\b)$ is a Hom-bialgebra.
\end{example}

Let us also consider the following examples from \cite{MakhSilv10-II}.

\begin{example}
Let $G$ be a group, $(kG,\mu)$ the group algebra of the group $G$, and $\g:G\to G$ a group homomorphism. Let also
\[
\g:kG \to kG, \qquad \g(\sum_{g\in G}c_g\,g) := \sum_{g\in G}c_g\,\g(g).
\]
be the associated algebra homomorphism. Then, following \cite[Coroll. 2.6(1)]{Yau09}, \ie Example \ref{ex-assoc-alg-to-hom-assoc-alg} above, we have the Hom-algebra $(kG,\g\circ\mu,\g)$. Furthermore, setting
%
%
\[
\D:kG \to kG\ot kG, \qquad \D(g) := \g(g) \ot \g(g)
\]
we arrive at a Hom-bialgebra $(kG,\mu,\g,\D,\g)$, see \cite[Ex. 3.32]{MakhSilv10-II}.
\end{example}

\begin{example}
Let $\C{B} := k[\{X_{ij}\mid 1 \leq i,j\leq n\}]$ be the polynomial algebra with $n^2$ indeterminates. We note that $(\C{B},\mu,\D)$ has the structure of a (co)associative bialgebra with the polynomial algebra and the matrix coalgebra structures. Now let $\g:\C{B}\to \C{B}$ be a bialgebra map. Accordingly, $(\C{B}, \g\circ \mu, (\g\ot \g)\circ \D, \g)$ is a Hom-bialgebra, see \cite[Ex. 3.33]{MakhSilv10-II}.
\end{example}

Parallel to the associative case, a Hom-Hopf algebra is a Hom-bialgebra with an antipode. Given a Hom-bialgebra $(\C{B},\mu,\a,\D,\b)$, the antipode is defined to be the inverse of the identity homomorphism $\Id \in \Hom(\C{B},\C{B})$ with respect to the convolution multiplication
\begin{equation}\label{conv-multp-bialg}
\mu^\star:\Hom(\C{B},\C{B}) \ot \Hom(\C{B},\C{B}) \to \Hom(\C{B},\C{B}), \qquad \mu^\star(f\ot g) := \mu\circ (f \ot g) \circ (\b^{-2}\ot \b^{-2}) \circ \D.
\end{equation}
Since the (multpilicative) identity $\eta \circ \ve \in \Hom(\C{B},\C{B})$ - see Proposition \ref{prop-Hom-C-A-duality} - of the convolution multiplication \eqref{conv-multp-bialg} requires the existence of a unit $\eta:k\to \C{B}$ and a counit $\ve:\C{B}\to k$, one has to start with a (co)unital Hom-bialgebra in order to define a Hom-Hopf algebra.

\begin{definition}
A ``Hom-Hopf algebra'', or an ``$(\a,\b)$-Hom-Hopf algebra'', is a 8-tuple $(\C{H},\mu,\eta,\a,\D,\ve,\b,S)$ such that
\begin{itemize}
\item[(i)] $(\C{H},\mu,\eta,\a,\D,\ve,\b)$ is a (co)unital Hom-bialgebra,
\item[(ii)] there is a linear map $S:\C{H}\to \C{H}$, called ``antipode'', satisfying 
\begin{enumerate}
\item $\mu \circ (S \otimes \Id) \circ \Delta  = \mu \circ (\Id \otimes S) \circ \Delta = \eta \circ \ve$,  
\item $S\circ \a = \a \circ S$,
\item $S \circ \b = \b \circ S$.
\end{enumerate}
\end{itemize}
\end{definition}

\begin{example}
Let $(H,\mu,\nu,\D,\ve,S)$ be a Hopf algebra, and let $\a:H\to H$ and $\b:H\to H$ be two commuting ($\a\circ \b = \b \circ \a$) bialgebra homomorphisms. Then setting $H_\a^\b := H$, we may say that $(H_\a^\b,\a\circ \mu,\eta, \a,\D\circ \b,\ve,\b,S)$ is a Hom-Hopf algebra.
\end{example}

The antipode of a Hom-Hopf algebra has the similar properties to that of a Hopf algebra. Namely, we have the following.

\begin{proposition}\label{prop-antipode-properties}
The following properties hold for any Hom-Hopf algebra $(\C{H},\mu,\eta,\a,\D,\ve,\b,S)$:
\begin{itemize}
\item[(i)] The antipode $S:\C{H}\to \C{H}$ is unique,
\item[(ii)] $S\circ \eta = \eta$,
\item[(iii)] $\ve \circ S = \ve$,
\item[(iv)] $S\circ \mu = \mu \circ(S\ot S) \circ \tau$,
\item[(v)] $\D\circ S = \tau \circ (S\ot S) \circ \D$.
\end{itemize}
\end{proposition}

\begin{proof}
The proofs of (i), (ii) and (iii) are exactly as in \cite[Prop. 3.23]{MakhSilv10-II}. We shall proceed directly towards (iii) and (iv) following \cite[Thm. 2.14]{Abe-book}. Let us begin with (iv). Considering the (unital) convolution Hom-algebra $(\Hom(\C{H}\ot \C{H},\C{H}), \mu^\star, \eta^\star,\a^\star)$, where $\C{H}\ot \C{H}$ is regarded as a tensor product (counital) Hom-coalgebra, and $\C{H}$ as a unital Hom-algebra, we set $\psi, \vp, \Phi \in \Hom(\C{H}\ot \C{H},\C{H})$ so that
\[
\psi(h\ot h') := h\bullet h', \qquad \vp(h\ot h') := S(h')\bullet S(h), \qquad \Phi(h\ot h') := S(h\bullet h').
\]
Now, on one hand we have
\begin{align*}
& \mu^\star(\Phi \ot \psi)(h\ot h') = \Phi(\b^{-2}(h\ns{1}) \ot \b^{-2}(h'\ns{1}))\bullet \psi(\b^{-2}(h\ns{2})\ot \b^{-2}(h'\ns{2})) = \\
& \Big[S\big(\b^{-2}(h\ns{1}) \bullet \b^{-2}(h'\ns{1})\big)\Big]\bullet \Big[\b^{-2}(h\ns{2})\bullet \b^{-2}(h'\ns{2})\Big] = \\
& S\Big(\big(\b^{-2}(h) \bullet \b^{-2}(h')\big)\ns{1}\Big) \big(\b^{-2}(h) \bullet \b^{-2}(h')\big)\ns{2} = \\
& \ve\big(\b^{-2}(h) \bullet \b^{-2}(h')\big)\eta(1) = \ve(h')\ve(h)\eta(1) = \eta^\star(1)(h \ot h'),
\end{align*}
and on the other hand,
\begin{align*}
& \mu^\star(\psi \ot \vp)(h\ot h') = \psi(\b^{-2}(h\ns{1}) \ot \b^{-2}(h'\ns{1}))\bullet \vp(\b^{-2}(h\ns{2}) \ot \b^{-2}(h'\ns{2})) = \\
& \Big[\b^{-2}(h\ns{1})\bullet \b^{-2}(h'\ns{1})\Big]  \bullet \Big[S\big(\b^{-2}(h'\ns{2})\big)\bullet S\big(\b^{-2}(h\ns{2})\big)\Big] = \\
& \a(\b^{-2}(h\ns{1})) \bullet \bigg[\b^{-2}(h'\ns{1})\bullet\Big(\a^{-1}(S\big(\b^{-2}(h'\ns{2})\big))\bullet \a^{-1}(S\big(\b^{-2}(h\ns{2})\big))\Big)\bigg]  =\\
& \a(\b^{-2}(h\ns{1})) \bullet\bigg[\Big(\a^{-1}(\b^{-2}(h'\ns{1}))\bullet\a^{-1}(S\big(\b^{-2}(h'\ns{2})\big)) \Big)\bullet S\big(\b^{-2}(h\ns{2})\big)\bigg]  =\\
& \a(\b^{-2}(h\ns{1})) \bullet\bigg[\ve(h')\eta(1)\bullet S\big(\b^{-2}(h\ns{2})\big)\bigg]  = \a(\b^{-2}(h\ns{1})) \bullet\bigg[\ve(h') \a(S\big(\b^{-2}(h\ns{2})\big))\bigg] =\\
&  \ve(h')\ve(h)\eta(1) = \eta^\star(1)(h \ot h').
\end{align*}
In view of the definition of a (co)unital Hom-bialgebra, we obtain $\vp = \Phi$ as they are both inverse to $\psi \in \Hom(\C{H}\ot \C{H},\C{H})$. Thus (iii) follows. 

As for (v), we start with the (unital) convolution Hom-algebra $(\Hom(\C{H},\C{H}\ot \C{H}), \mu^\star, \eta^\star, \a^\star)$, where we regard $\C{H}$ as a counital Hom-coalgebra and $\C{H}\ot \C{H}$ a tensor product unital Hom-algebra. This time we consider $\tau \circ (S\ot S) \circ \D, \D\circ S \in \Hom(\C{H},\C{H}\ot \C{H})$, where $\tau:\C{H}\ot \C{H} \to \C{H}\ot \C{H}$ is the flip map. Then, in view of the definition of a (co)unital Hom-bialgebra, we have on one hand
\begin{align*}
& \mu^\star((\D \circ S) \ot \D)(h) = (\D\circ S)(\b^{-2}(h\ns{1}))\bullet \D(\b^{-2}(h\ns{2})) = \\
& \Big[S(\b^{-2}(h\ns{1}))\ns{1}  \ot S(\b^{-2}(h\ns{1}))\ns{2}\Big]  \bullet \Big[\b^{-2}(h\ns{2}\ns{1}) \ot \b^{-2}(h\ns{2}\ns{2})\Big] = \\
& S(\b^{-2}(h\ns{1}))\ns{1} \bullet \b^{-2}(h\ns{2}\ns{1})  \ot S(\b^{-2}(h\ns{1}))\ns{2}  \bullet \b^{-2}(h\ns{2}\ns{2}) = \\
& \D(S(\b^{-2}(h\ns{1})) \bullet \b^{-2}(h\ns{2})) =  \ve(h)\eta(1) \ot \eta(1) = \eta^\star(1)(h),
\end{align*}
and on the other hand
\begin{align*}
& \mu^\star(\D \ot (\tau \circ (S\ot S) \circ \D))(h) = \D(\b^{-2}(h\ns{1}))\bullet (\tau \circ (S\ot S) \circ \D)(\b^{-2}(h\ns{2})) = \\
& \bigg(\b^{-2}(h\ns{1}\ns{1}) \ot \b^{-2}(h\ns{1}\ns{2})\bigg) \bullet \bigg(S(\b^{-2}(h\ns{2}\ns{1})) \ot S(\b^{-2}(h\ns{2}\ns{2}))\bigg) = \\
& \b^{-2}(h\ns{1}\ns{1}) \bullet S(\b^{-2}(h\ns{2}\ns{1})) \ot \b^{-2}(h\ns{1}\ns{2}) \bullet S(\b^{-2}(h\ns{2}\ns{2})) = \\
& \b^{-2}(h\ns{1}\ns{1}) \bullet S(\b^{-1}(h\ns{2})) \ot \b^{-3}(h\ns{1}\ns{2}\ns{1}) \bullet S(\b^{-3}(h\ns{1}\ns{2}\ns{2})) = \\
& \b^{-2}(h\ns{1}\ns{1}) \bullet S(\b^{-1}(h\ns{2})) \ot \ve(h\ns{1}\ns{2}) \eta(1) = \b^{-1}(h\ns{1}) \bullet S(\b^{-1}(h\ns{2})) \ot \eta(1) = \\
& \ve(h)\eta(1) \ot \eta(1) = \eta^\star(1)(h),
\end{align*}
where, on the fourth equality we used the fact that
\begin{align*}
& \b(h\ns{1}) \ot h\ns{2}\ns{1} \ot h\ns{2}\ns{2} = h\ns{1}\ns{1} \ot h\ns{1}\ns{2} \ot \b(h\ns{2}) \overset{\D\ot\b\ot\Id}{\Rightarrow} \\
& \b(h\ns{1}\ns{1}) \ot \b(h\ns{1}\ns{2}) \ot \b(h\ns{2}\ns{1}) \ot h\ns{2}\ns{2} = \D(h\ns{1}\ns{1}) \ot \b(h\ns{1}\ns{2}) \ot \b(h\ns{2}) = \\
& \b(h\ns{1}\ns{1}) \ot \D(h\ns{1}\ns{2}) \ot \b(h\ns{2}) = \b(h\ns{1}\ns{1}) \ot h\ns{1}\ns{2}\ns{1} \ot h\ns{1}\ns{2}\ns{2} \ot \b(h\ns{2}) \overset{\b^{-1}\ot\b^{-1}\ot\b^{-1}\ot\Id}{\Rightarrow} \\
& h\ns{1}\ns{1} \ot h\ns{1}\ns{2} \ot h\ns{2}\ns{1} \ot h\ns{2}\ns{2} = h\ns{1}\ns{1} \ot \b^{-1}(h\ns{1}\ns{2}\ns{1}) \ot \b^{-1}(h\ns{1}\ns{2}\ns{2}) \ot \b(h\ns{2}).
\end{align*}
\end{proof}

We do note that the last two items of Proposition \ref{prop-antipode-properties} are subject also of \cite[Prop. 2.9]{CaenGoyv11}. 

\begin{example}
Given a Hom-Hopf algebra $(\C{H},\mu,\eta,\a,\D,\ve,\b,S)$, the 8-tuple $(\C{H},\mu^{op},\eta,\a,\D,\ve,\b,S^{-1})$ is also a Hom-Hopf algebra, which is denoted by $\C{H}^{op}$. Similarly, $(\C{H},\mu,\eta,\a,\D^{op},\ve,\b,S^{-1})$ becomes a Hopf algbra, denoted by $\C{H}^{cop}$.
\end{example}

\begin{proposition}\label{prop-dual-Hom-Hopf}
Let $(\C{H},\mu,\eta,\a,\D,\ve,\b,S)$ be a Hom-Hopf algebra. Then, its restricted dual $(\C{H}^\circ,\D^\ast,\ve,(\b^{-1})^\ast, \pi^{-1}\circ\d,\eta^\ast,(\a^{-1})^\ast, S^\ast)$ is a Hom-Hopf algebra.
\end{proposition}

\begin{proof}
It follows from Corollary \ref{coroll-dual-hom-alg} that $(\C{H}^\circ,\D^\ast,\ve,(\b^{-1})^\ast)$ is a unital Hom-algebra, and from Corollary \ref{coroll-dual-hom-coalg} that $(\C{H}^\circ, \pi^{-1}\circ\d,\eta^\ast,(\a^{-1})^\ast)$ is a counital Hom-coalgebra. We shall, thus, check the bialgebra compatibilities (1)-(9) now. The very first compatibility (1) follows along the lines of
\[
(\pi^{-1}\circ\d)(\ve) (h \ot h')= \ve (\a^{-2}(h \bullet h')) = \ve (h \bullet h') = (\ve \ot \ve)(h\ot h'),
\]
that is,
\[
(\pi^{-1}\circ\d)(\ve)  = \ve \ot \ve.
\]
Denoting the multiplication in $\C{H}^\circ$ by $\star:\C{H}\ot \C{H}\to \C{H}$, and $\pi^{-1}\circ \d(f) =: f\ns{1}\ot f\ns{2}\in \C{H}^\circ\ot \C{H}^\circ$, we next see that
\begin{align*}
& (\pi^{-1}\circ\d)(f \star g)(h\ot h') = (f \star g)(\a^{-2}(h\bullet h')) = (f \star g)(\a^{-2}(h)\bullet \a^{-2}(h')) = \\
& f(\b^{-2}(\a^{-2}(h\ns{1})\bullet \a^{-2}(h'\ns{1})))g(\b^{-2}(\a^{-2}(h\ns{2})\bullet \a^{-2}(h'\ns{2}))) = \\
& f(\b^{-2}(\a^{-2}(h\ns{1}))\bullet \b^{-2}(\a^{-2}(h'\ns{1})))g(\b^{-2}(\a^{-2}(h\ns{2}))\bullet \b^{-2}(\a^{-2}(h'\ns{2}))) =\\
& f\ns{1}(\b^{-2}(h\ns{1}))f\ns{2}(\b^{-2}(h'\ns{1}))g\ns{1}(\b^{-2}(h\ns{2}))g\ns{2}(\b^{-2}(h'\ns{2})) = \\
& (f\ns{1}\star g\ns{1})(h)(f\ns{2}\star g\ns{2})(h').
\end{align*}
Hence, 
\[
(\pi^{-1}\circ\d)(f\star g)  = (\pi^{-1}\circ\d)(f)\star (\pi^{-1}\circ\d)(g),
\]
for any $f,g\in \C{H}^\circ$, which is (2). We proceed into (3) through
\begin{align*}
&\Big((\pi^{-1}\circ \d)\circ (\b^{-1})^\ast \Big)(f) (h \ot h') = (\pi^{-1}\circ \d)(f\circ \b^{-1})(h\ot h' ) = (f\circ \b^{-1})(\a^{-2}(h\bullet h')) = \\
& f(\a^{-2}(\b^{-1}(h)\bullet \b^{-1}(h'))) = (\pi^{-1}\circ \d)(f)(\b^{-1}(h)\bullet \b^{-1}(h')) = \\
& \Big(((\b^{-1})^\ast\ot (\b^{-1})^\ast)\circ (\pi^{-1}\circ \d)\Big)(f) (h\ot h'),
\end{align*}
and we thus obtain
\[
(\pi^{-1}\circ\d)\circ (\b^{-1})^\ast = ((\b^{-1})^\ast \ot (\b^{-1})^\ast) \circ(\pi^{-1}\circ\d).
\]
The property (4) follows rather easily as
\[
\Big(\eta^\ast\circ \ve\Big)(\s) = \ve(\eta(\s)) = \s,
\]
for any $\s \in k$, that is,
\[
\eta^\ast\circ \ve = \Id_k.
\]
Next comes the observation of the compatibility (5). To this end, we see that
\begin{align*}
& \eta^\ast(f\star g) = f\star g(\eta(1)) = f(\b^{-2}(\eta(1)))g(\b^{-2}(\eta(1))) = f(\eta(1))g(\eta(1)) = \eta^\ast(f)\eta^\ast(g).
\end{align*}
The item (6) is equally straightforward. Indeed,
\[
\Big(\eta^\ast\circ(\b^{-1})^\ast\Big)(f) = (f\circ \b^{-1})(\eta(1)) = f(\b^{-1}(\eta(1))) = f(\eta(1)) =\eta^\ast (f),
\]
in other words,
\[
\eta^\ast\circ(\b^{-1})^\ast = \eta^\ast.
\]
Similarly, the compatibility (7), that is,
\[
(\a^{-1})^\ast \circ \ve = \ve,
\] 
follows at once from
\[
\Big((\a^{-1})^\ast \circ \ve\Big)(h) = \ve(\a^{-1}(h)) = \ve(h).
\]
As for item (8), we have
\begin{align*}
& (\a^{-1})^\ast(f\star g) (h) = (f\star g)(\a^{-1}(h)) = f(\b^{-2}(\a^{-1}(h\ns{1})))g(\b^{-2}(\a^{-1}(h\ns{2}))) = \\
& (\a^{-1})^\ast(f)(\b^{-2}(h\ns{1}))(\a^{-1})^\ast(g)(\b^{-2}(h\ns{2})) = \Big((\a^{-1})^\ast(f) \star (\a^{-1})^\ast(g)\Big)(h),
\end{align*}
that is,
\[
(\a^{-1})^\ast(f\star g) = (\a^{-1})^\ast(f) \star (\a^{-1})^\ast(g).
\]
Finally, 
\[
(\a^{-1})^\ast \circ (\b^{-1})^\ast = (\b^{-1})^\ast \circ (\a^{-1})^\ast
\]
that is, the compatibility (9), follows through
\[
(\a^{-1})^\ast \circ (\b^{-1})^\ast = (\b^{-1} \circ \a^{-1})^\ast = (\a^{-1} \circ \b^{-1})^\ast = (\b^{-1})^\ast \circ (\a^{-1})^\ast.
\]
We conclude by the compatibilities with the antipode. To this end, we first observe that
\begin{align*}
& \Big(S^\ast(f\ns{1})\star f\ns{2}\Big)(h) = S^\ast(f\ns{1})(\b^{-2}(h\ns{1}))f\ns{2}(\b^{-2}(h\ns{2})) = \\
& f\ns{1}(S(\b^{-2}(h\ns{1})))f\ns{2}(\b^{-2}(h\ns{2})) = f(\a^{-2}(S(\b^{-2}(h\ns{1}))\bullet \b^{-2}(h\ns{2}))) = \\
& \ve(h)f(\eta(1)) = \Big(\eta^\ast\circ \ve\Big)(f)(h),
\end{align*}
that is,
\[
S^\ast(f\ns{1})\star f\ns{2} = \eta^\ast(f) \ve.
\]
Similarly, we may observe that
\[
f\ns{1} \star S^\ast(f\ns{2})= \eta^\ast(f) \ve.
\]
The remaining two properties of the antipode follows immediately. Indeed,
\[
S^\ast \circ (\b^{-1})^\ast = (\b^{-1}\circ S)^\ast = (S \circ \b^{-1})^\ast = (\b^{-1})^\ast \circ S^\ast,
\]
and quite similarly,
\[
S^\ast \circ (\a^{-1})^\ast = (\a^{-1}\circ S)^\ast = (S \circ \a^{-1})^\ast = (\a^{-1})^\ast \circ S^\ast.
\]
\end{proof}

\subsubsection{Hom-Lie algebras and their enveloping Hom-Hopf algebras}

The notion of Hom-Lie algebra has been introduced in \cite{HartLarsSilv06}, and is recalled below.

\begin{definition}
A triple $(\G{g},[\,,\,],\phi)$ consisting of a vector space $\G{g}$, an anti-symmetric bilinear map $[\,,\,]:\G{g}\ot \G{g} \to \G{g}$, and a linear map $\phi:\G{g}\to \G{g}$ is called a ``Hom-Lie algebra'' if the Hom-Jacobi identity
\[
[\phi(\xi),[\xi',\xi'']] + [\phi(\xi''),[\xi,\xi']] + [\phi(\xi'),[\xi'',\xi']] = 0
\]
is satisfied, for all $\xi,\xi',\xi'' \in \G{g}$. In case the linear mapping $\phi:\G{g}\to \G{g}$ is an endomorphism, that is, it satisfies 
\[
\phi([\xi,\xi']) = [\phi(\xi),\phi(\xi')]
\]
for all $\xi,\xi'\in \G{g}$, then the Hom-Lie algebra $(\G{g},[\,,\,],\phi)$ is called ``multiplicative''.
\end{definition}
In the sequel, we shall consider the multiplicative Hom-Lie algebras only.

Given two Hom-Lie algebras $(\G{g},\phi)$ and $(\G{h},\a)$ - suppressing their brackets, a ``morphism'' of Hom-Lie algebras is defined to be a linear map $f:\G{g}\to \G{h}$ such that
\[
f([\xi,\xi']) = [f(\xi),f(\xi')]
\]
and that
\[
f(\phi(x)) = \a(f(x)),
\]
for all $\xi,\xi' \in \G{g}$.

As is noted in \cite[Ex. 1.4]{Laur-GengMakhTele18}, see also \cite{MakhSilv08,Yau09}, given any Lie algebra $(\G{g},[\,,\,])$, and a Lie algebra endomorphism $\phi:\G{g}\to \G{g}$, the triple $(\G{g},[\,,\,]_\phi,\phi)$ is a Hom-Lie algebra, where 
\[
[\xi,\xi']_\phi := \phi([\xi,\xi'])
\]
for any $\xi,\xi' \in \G{g}$. Similarly, given a Hom-algebra $(\C{A},\phi)$, the triple $(\C{A},[\,,\,],\phi)$ becomes a Hom-Lie algebra via
\[
[x,y] := xy - yx
\]
for all $x,y \in \C{A}$.

The efforts of associating a universal enveloping (Hom-Hopf) algebra to a given Hom-Lie algebra has been initiated in \cite{Yau08,Yau10}, constructing a Hom-bialgebra out of a Hom-Lie algebra. In case the Hom-Lie algebra is multiplicative, then a Hom-Hopf algebra construction (in terms of weighted trees) has been achieved in \cite{Laur-GengMakhTele18}.  Let us conclude this section with a brief review of this universal enveloping Hom-Hopf algebra construction.

To this end, we begin with the notion of a ``weighted $n$-tree'', consisting of a pair $(\vp,a_1,\ldots,a_n)$ of an $n$-tree $\vp\in T_n$, $T_n$ being the set of all planar binary trees of $n$-leaves, and an $n$-tuple $(a_1,\ldots,a_n)\in \B{N}^n$. As such, a weighted $n$-tree may be pictured as an $n$-tree with the numbers $a_1,\ldots,a_n \in \B{N}$ on top of the leaves, see for instance \cite[Subsect. 4.1]{Laur-GengMakhTele18}, or alternatively \cite[Subsect. 2.3]{Yau08}. Following the terminology of \cite{Laur-GengMakhTele18}, we shall denote the set of all weighted $n$-trees by $B_n$, and
\[
B:= \{{\bf 1}\}\,\cup\,\bigcup_{n\geq 1}\, B_n,
\]
where {\bf 1} is an external object. We also let $\B{T}$ denote the vector space generated by $B$. 

The ``grafting'' operation (adjoining the roots of two planar binary trees, say a $n$-tree and a $m$-tree, on a common root to get a $(n+m)$-tree) extend trivially to the level of weighted trees; $\vee: B_n\times B_m \to B_{n+m}$. Furthermore, the grafting extends to $\B{T}$  along with the operation
\begin{equation}\label{a-map}
\G{a}:B_n\to B_n, \qquad \vp:=(\vp,s_1,\ldots,s_n)\mapsto (\vp,s_1+1,\ldots,s_n+1)
\end{equation}
according to the convensions
\begin{itemize}
\item[(i)] $\G{a}({\bf 1}) = {\bf 1}$, \\
\item[(ii)] ${\bf 1} \vee {\bf 1} = {\bf 1}$, \\
\item[(iii)] $\vp\vee {\bf 1} = {\bf 1} \vee \vp = \G{a}(\vp)$.
\end{itemize}
Let us note that the grafting operation does not satisfy any Hom-associativity according to \eqref{a-map}, that is, given arbitrary $\vp,\vp',\vp'' \in T$, 
\begin{equation}\label{gen-ideal-I}
(\vp \vee \vp') \vee \G{a}(\vp'') - \G{a}(\vp) \vee (\vp' \vee \vp'')
\end{equation}
does not necessarily vanish. Nevertheless, $\C{I}$ being the space generated by the elements of the form \eqref{gen-ideal-I}, the quotient $\B{T} / \C{I}$ does. In other words, $(\B{T} / \C{I}, \vee, {\bf 1}, \G{a})$ is a unital Hom-algebra. 

Next, given any $\vp:=(\vp,s_1,\ldots,s_n)$, and any 
\[
 I =: \{t_1,\ldots, t_m\} \subseteq \{1,2,\ldots,n\},
\]
let $\vp_I \in T_m$ be the tree whose braches corresponding to $\{1,2,\ldots,n\} \backslash I$ is replaced with ${\bf 1}$. Then, a coproduct $\D:\B{T}\to \B{T}\ot \B{T}$ is defined as
\[
\D(\vp) := \sum_{ \underset{I \cap J = \emptyset}{I\cup J = \{1,2,\ldots,n\}}}\, \vp_I \ot \vp_J, \qquad \D({\bf 1}) := {\bf 1} \ot {\bf 1},
\]
while $\ve:\B{T}\to k$ given by
\[
\ve({\bf 1}) = 1, \qquad \ve(\vp) = 0
\]
for any $\vp \in B_n$, yields a coproduct. The coproduct and the counit pass to the quotient $\B{T} / \C{I}$, and we arrive at a Hom-bialgebra $(\B{T} / \C{I}, \vee,{\bf 1}, \G{a}, \D,\ve,\Id)$. Finally, together with the antipode $S:\B{T}\to \B{T}$ given by
\begin{itemize}
\item[(i)] $S({\bf 1}) = {\bf 1}$, \\
\item[(ii)] $S(\vp_1, s) = - (\vp_1, s)$, where $T_1 = \{\vp_1\}$, \\
\item[(iii)] $S(\vp \vee \vp') = S(\vp')\vee S(\vp)$ for any $\vp,\vp' \in B$,
\end{itemize}
the 8-tuple $(\B{T} / \C{I}, \vee,{\bf 1}, \G{a}, \D,\ve,\Id, S)$ is a Hom-Hopf algebra.

In the presence of a Hom-Lie algebra $(\G{g},\phi)$ now, one begins with the space 
\[
\B{T}^\G{g} := k{\bf 1} \,\oplus\, \bigoplus_{n\geq 1}\,\Big(B_n \, \oplus\,\G{g}^{\ot\,n}\Big),
\]
elements of which may be represented by
\[
(\vp,s_1,\ldots,s_n,\xi_1,\ldots,\xi_n), \qquad \vp\in T_n,\quad s_1,\ldots, s_n \in \B{N},\quad \xi_1,\ldots,\xi_n \in \G{g}.
\]
They may equally be pictured as weighted trees with the elements of $\G{g}$ on top of the weighted leaves. As is remarked in \cite[Subsect. 4.2]{Laur-GengMakhTele18}, there are extensions 
\begin{align*}
& \vee:\B{T}^\G{g} \ot \B{T}^\G{g} \to \B{T}^\G{g}, \\
&(\vp,s_1,\ldots,s_n,\xi_1,\ldots,\xi_n) \vee (\vp',s'_1,\ldots,s'_n,\xi'_1,\ldots,\xi'_n) := \\
& \hspace{5cm} (\vp\vee \vp',s_1,\ldots,s_n, s'_1,\ldots,s'_n,\xi_1,\ldots,\xi_n,\xi'_1,\ldots,\xi'_n), \\
& \G{a}:\B{T}^\G{g}\to \B{T}^\G{g}, \qquad \G{a}(\vp,s_1,\ldots,s_n,\xi_1,\ldots,\xi_n) := (\vp,s_1,\ldots,s_n,\phi(\xi_1),\ldots,\phi(\xi_n)), \\
& \D:\B{T}^\G{g}  \to \B{T}^\G{g} \ot \B{T}^\G{g}, \\
& \D(\vp,s_1,\ldots,s_n,\xi_1,\ldots,\xi_n) := \\
& \sum_{\underset{\{t_1,\ldots,t_m\} \cap \{p_1,\ldots,p_{n-m}\} =\emptyset}{\{t_1,\ldots,t_m\} \cup \{p_1,\ldots,p_{n-m}\} = \{1,\ldots,n\}}}\,(\vp_{\{t_1,\ldots,t_m\}},\xi_{t_1},\ldots,\xi_{t_m}) \ot (\vp_{\{p_1,\ldots,p_{n-m}\}},\xi_{p_1},\ldots,\xi_{p_{n-m}}), \\
& \ve:\B{T}^\G{g}  \to k, \qquad \ve({\bf 1}) = 1 \qquad \ve(\vp,s_1,\ldots,s_n,\xi_1,\ldots,\xi_n) = 0, \quad \forall\,(\vp,s_1,\ldots,s_n)\in B_n, \quad n\geq 1.
\end{align*}
In order to be able to obtain a Hom-associative algebra, there is a need to an analogue of $\C{I}$. This is the (co)ideal given by
\[
\C{I}^\G{g}:= \bigoplus_{n\geq 1}\, \Big(\C{I}\,\cap\,B_n\Big)\ot \G{g}^{\ot\,n}.
\]
Then, the 8-tuple $(\B{T}^\G{g} / \C{I}^\G{g}, \vee,{\bf 1}, \G{a}, \D,\ve,\Id, S)$ is a Hom-Hopf algebra. Finally, there comes the Hopf-ideal $\C{J}^\G{g}$ of $\B{T}^\G{g} / \C{I}^\G{g}$, see \cite[Prop. 4.13]{Laur-GengMakhTele18}, generated by the elements of the form
\begin{itemize}
\item[(i)]  $(\vp_1,s,\xi) - (\vp_1,0,\phi^s(\xi))$, and \\
\item[(ii)] $(\vp_2,0,0,\xi_1,\xi_2) - (\vp_2,0,0,\xi_2,\xi_1) - (\vp_1,0,[\xi_1,\xi_2])$, where we recall that $T_1=\{\vp_1\}$, and $T_2=\{\vp_2\}$, $\vp_2 := \vp_1\vee\vp_1$.
\end{itemize}
Accordingly, the Hom-Hopf algebra
\[
\C{U}(\G{g}) := \Big((\B{T}^\G{g} / \C{I}^\G{g}) / \C{J}^\G{g}, \vee, {\bf 1}, \phi, \D,\ve,\Id,S\Big)
\]
is called the ``universal enveloping Hom-Hopf algebra'' of the Hom-Lie algebra $(\G{g},\phi)$.

\section{Hom-Hopf algebra symmetries}\label{sect-Hom-Hopf-symmetries}

The present section is reserved for the Hom-associative analogues of the four fundamental Hopf symmetries in non-commutative geometry; namely the module algebra, module coalgebra, comodule algebra, and comodule coalgebra. These symmetries will be instrumental later in the constructions of ``double cross product Hom-Hopf algebras'' and ``bicrossproduct Hom-Hopf algebras''.

\subsection{Hom-module algebra}\label{subsect-module-alg}~

We shall pursue the motivation given at \cite[Thm. 1.1]{Yau08-II}, see also \cite[Prop. 1.10]{MakhPana15}, to seek appropriate compatibility conditions for the ``module algebra'' symmetry via the Hom-Hopf algebras. However, the examples coming from the (restricted) duals of universal enveloping Hom-algebras (of Hom-Lie algebras), which are associative as algebras (that is, those with $\a = \Id$), does not satify the hypothesis of \cite[Thm. 1.2]{Yau08-II}  (see also \cite[Prop. 1.10]{MakhPana15} and \cite[Prop. 2.13]{MakhPana16}). As a result, we take the following point of view.

Let $(H,\mu,\eta,\D,\ve)$ be a bialgebra, and $(A,\mu_A,\eta_A)$ be an algebra so that $A$ is a (left) $H$-module algebra via $\cdot:H\ot A \to A$. Let also $\phi,\psi:H\to H$ be two bialgebra homomorphsims, via which we construct the Hom-bialgebra $H_\phi^\psi$. Let, finally, $\a,\gamma:A\to A$ be two (commuting) algebra homomorphisms, so that,
\[
\gamma(h\cdot a) = \phi(h)\cdot \gamma(a), \qquad \a(h\cdot a) = \psi(h)\cdot \a(a)
\]
for any $a\in A$, and any $h\in H$. Thus, setting
\[
\rt:H_\phi^\psi \ot A_\a\to A_\a, \qquad h\rt a := \phi(h)\cdot \gamma(a),
\]
we endow $(A_\a,\gamma)$ with a (left) $H_\phi^\psi$-Hom-module structure. Indeed,
\begin{align}\label{computation-module}
\begin{split}
& (h\bullet h') \rt \g(a) = \phi(hh') \rt \g(a) = \phi^2(hh') \cdot \g^2(a) = \\
& \phi^2(h) \cdot \Big(\phi^2(h') \cdot \g^2(a)\Big) =  \phi(h) \rt \Big(\phi(h') \cdot \g(a)\Big) = \phi(h) \rt \Big(h' \rt a\Big),
\end{split}
\end{align}

On the next step, we define the diagonal action as
\begin{align*}
& \rt:H_\phi^\psi \ot A_\a\ot A_\a\to A_\a\ot A_\a, \\ 
& h\rt (a\ot a') := \phi(\psi(h\ps{1}))\cdot \gamma(a) \ot \phi(\psi(h\ps{2}))\cdot \gamma(a') = h\ns{1} \rt a \ot h\ns{2} \rt a'.
\end{align*}
Then, along the lines of 
\begin{align}\label{computation-hom-module-diagonal}
\begin{split}
& (h\bullet h') \rt \g\big(a \ot a'\big) = (h\bullet h') \rt \big(\a(a) \ot \a(a')\big) = \\
& \phi(hh') \rt \big(\g(a) \ot \g(a')\big) =  \psi(\phi^2(h\ps{1}h'\ps{1}))\cdot \g^2(a) \ot \psi(\phi^2(h\ps{2}h'\ps{2}))\cdot \g^2(a') = \\
& \psi(\phi^2(h\ps{1}))\cdot \Big(\psi(\phi^2(h'\ps{1})) \cdot \g^2(a)\Big) \ot \psi(\phi^2(h\ps{2}))\cdot \Big(\psi(\phi^2(h'\ps{2})) \cdot\g^2(a') \Big) = \\
& \phi^2(\psi(h\ps{1}))\cdot \Big(\phi^2(\psi(h'\ps{1})) \cdot \g^2(a)\Big) \ot \phi^2(\psi(h\ps{2}))\cdot \Big(\phi^2(\psi(h'\ps{2})) \cdot\g^2(a') \Big) = \\
& \phi(\psi(h\ps{1}))\rt \Big(\phi(\psi(h'\ps{1})) \cdot \g(a)\Big) \ot \phi(\psi(h\ps{2}))\rt \Big(\phi(\psi(h'\ps{2})) \cdot\g(a') \Big) = \\
& \phi(\psi(h\ps{1}))\rt \Big(\psi(h'\ps{1}) \rt a\Big) \ot \phi(\psi(h\ps{2}))\rt \Big(\psi(h'\ps{2}) \rt a' \Big) = \\
& \phi(h) \rt \Big(h' \rt (a \ot a')\Big),
\end{split}
\end{align}
we see that $(A_\a \ot A_\a, \gamma\ot \gamma)$ is a (left) $H_\phi^\psi$-Hom-module. We proceed to the observation that $\mu_\a:(A_\a \ot A_\a, \gamma\ot \gamma) \to (A_\a, \gamma)$ is a morphism of $H_\phi^\psi$-Hom-modules. Indeed,
\begin{align*}
& \mu_\a(h \rt (a \ot a')) = \big(h\ns{1} \rt a\big) \bullet \big(h\ns{2} \rt a'\big) = \big(\psi(h\ps{1}) \rt a\big) \bullet \big(\psi(h\ps{2}) \rt a'\big) = \\
& \a\big(\psi(h\ps{1}) \rt a\big) \a\big(\psi(h\ps{2}) \rt a'\big) =  \a\big(\phi(\psi(h\ps{1})) \cdot\gamma(a)\big) \a\big(\phi(\psi(h\ps{2})) \cdot\gamma(a')\big) = \\
& \Big(\phi(\psi^2(h\ps{1})) \cdot\a\gamma(a)\Big) \Big(\phi(\psi^2(h\ps{2})) \cdot\a(\gamma(a'))\Big) =  \phi(\psi^2(h)) \cdot \a(\gamma(a))\a(\gamma(a')) = \\
& \phi(\psi^2(h)) \cdot \a(\gamma(aa')) = \phi(\psi^2(h)) \cdot \gamma(a\bullet a') = \psi^2(h) \rt (a \bullet a') = \psi^2(h) \rt \mu_\a(a \ot a').
\end{align*}

\begin{remark}
Considering
\[
\brt:H_\phi^\psi \ot A_\a\to A_\a, \qquad h\brt a := \psi^2(h)\rt a,
\]
we may endow $(A_\a,\g)$ with another (left) $H_\phi^\psi$-Hom-module structure. Indeed,
\begin{align*}
& (h \bullet h') \brt \g(a) = \psi^2(h \bullet h') \rt \g(a) = \phi(\psi^2(h)) \rt \Big(\psi^2(h')\rt a\Big) = \\
& \phi(\psi^2(h)) \rt \Big(h'\brt a\Big) = \phi(h) \brt \Big(h'\brt a\Big).
\end{align*}
\end{remark}

Finally, we set
\[
\rt: H_\phi^\psi \ot k\to k, \qquad h\rt r := \ve(h) r,
\]
and observe that 
\begin{align*}
& \eta(h \rt r) = \eta(\ve(h)r) = \ve(h)\eta(r) =  \ve(\phi(h))\eta(r) =  \\
& \phi(h)\cdot \eta(r)  = \phi(h)\cdot \a(\eta(r)) = h \rt \eta(r).
\end{align*}

Accordingly, we introduce the following definition of ``Hom-module algebra''.

\begin{definition}\label{def-Hom-mod-alg}
Let $(\C{H},\mu,\eta,\phi,\D,\ve,\psi)$ be a Hom-bialgebra, and $(\C{A},\mu_\C{A},\eta_\C{A},\a)$ be a Hom-algebra together with a Hom-algebra endomorphism $\g:\C{A}\to \C{A}$ so that $(\C{A},\g)$ is a left $\C{H}$-Hom-module via $\rt:\C{H}\ot \C{A} \to \C{A}$. Then the pair $(\C{A},\g)$ is called a $\C{H}$-Hom-module algebra if the Hom-algebra structure maps are $\C{H}$-Hom-module morphisms, \ie
\begin{align}
& \a(h \rt a) = \psi(h) \rt \a(a), \label{Hom-mod-alg-00} \\
& \psi^2(h)\rt (a \bullet a') = (h\ns{1} \rt a) \bullet (h\ns{2} \rt a'), \label{Hom-mod-alg-I}\\
& h \rt \eta(1) = \ve(h)\eta(1), \label{Hom-mod-alg-II}
\end{align}
for any $h \in \C{H}$, and any $a,a'\in \C{A}$.
\end{definition}

\begin{remark}
In case $\phi = \psi$, the module Hom-algebra compatibility reduces to the one given in \cite{MakhPana15,Yau08-II}. 
\end{remark}

\subsection{Hom-module coalgebra}\label{subsect-module-coalg}~

This time we shall investigate appropriate compatibility conditions for the ``module coalgebra'' symmetry for Hom-bialgebras. In particular, we would like to cover the universal enveloping Hom-Lie-algebras; which are coassociative as coalgebras (in other words $\b = \Id$). Keeping in mind the similar considerations as in Subsection \ref{subsect-module-alg}, we shall take the following point of view.

Let $(H,\mu,\eta,\D,\ve)$ be a bialgebra, and $(C,\D_C,\ve_C)$ be a coalgebra so that $C$ is a (left) $H$-module coalgebra via $\cdot:H\ot C \to C$. Let also $\phi,\psi:H\to H$ be two bialgebra homomorphsims, via which we construct the Hom-bialgebra $H_\phi^\psi$. Let, finally, $\b,\gamma:C\to C$ be two (commuting) coalgebra homomorphisms, so that,
\[
\gamma(h\cdot c) = \phi(h)\cdot \gamma(c), \qquad \b(h\cdot c) = \psi(h)\cdot \b(c)
\]
for any $c\in C$, and any $h\in H$.

We may, then, define the following action of the Hom-bialgebra $H_\phi^\psi$ on the Hom-coalgebra $C^\b$ as
\[
\rt: H_\phi^\psi \ot C^\b \to C^\b, \qquad h \rt c := \phi(h) \cdot \gamma(c). 
\]
Then, $(C^\b,\g)$ becomes a $H_\phi^\psi$-Hom-module. Indeed,
\begin{align}\label{aux-mod-coalg-action}
\begin{split}
& (h \bullet h') \rt \g(c) = \phi(hh') \rt \g(c) = \phi^2(hh') \cdot \g^2(c) = \\
& \phi^2(h) \cdot \big(\phi^2(h') \cdot \g^2(c)\big) = \phi(h) \rt\big(\phi(h') \cdot \g(c)\big) = \phi(h) \rt\big(h' \rt c\big). 
\end{split}
\end{align}

Similarly, setting the diagonal action as
\begin{align*}
&\rt: H_\phi^\psi \ot C^\b \ot C^\b \to C^\b \ot C^\b, \\ 
& h \rt (c\ot c') := \phi(\psi(h\ps{1})) \cdot \gamma(c) \ot \phi(\psi(h\ps{2})) \cdot \gamma(c') = h\ns{1}\rt c \ot h\ns{2}\rt c',
\end{align*}
it follows from 
\begin{align}\label{aux-mod-coalg-diag-act}
\begin{split}
& (h \bullet h') \rt \g(c \ot c') = \phi(hh') \rt \big(\g(c) \ot \g(c')\big) = \\
& \phi^2(\psi(h\ps{1}))\phi^2(\psi(h'\ps{1})) \cdot \g^2(c) \ot \phi^2(\psi(h\ps{2}))\phi^2(\psi(h'\ps{2})) \cdot \g^2(c') = \\
& \phi^2(\psi(h\ps{1})) \cdot \Big(\phi^2(\psi(h'\ps{1})) \cdot \g^2(c)\Big) \ot \phi^2(\psi(h\ps{2}))\cdot \Big(\phi^2(\psi(h'\ps{2})) \cdot \g^2(c') \Big) = \\
& \phi(h) \rt \big(h' \rt (c \ot c')\big)
\end{split}
\end{align}
that $(C^\b\ot C^\b, \g\ot \g)$ is a $H_\phi^\psi$-Hom-module. 

Furthermore, the comultiplication $\D_\b:(C^\b,\g) \to (C^\b \ot C^\b,\g\ot \g)$ becomes a morphism of $H_\phi^\psi$-Hom-modules. Indeed,
\begin{align}\label{aux-mod-coalg-comp}
\begin{split}
& h\rt \D_\b(c) = h\ns{1}\rt c\ns{1} \ot h\ns{2}\rt c\ns{2} = h\rt \big(\b(c\ps{1})\ot \b(c\ps{2})\big) =  \\
& \phi(\psi(h\ps{1})) \cdot \gamma(\b(c\ps{1})) \ot \phi(\psi(h\ps{2})) \cdot \gamma(\b(c\ps{2})) = \\
& \b(\phi(h\ps{1})\cdot \g(c\ps{1})) \ot \b(\phi(h\ps{2})\cdot \g(c\ps{2}))= \D_\b(h\rt c).
\end{split}
\end{align}
Similarly, $\ve_C:(C^\b,\g) \to (k,\Id)$ becomes a morphism of $H_\phi^\psi$-Hom-modules.

Accordingly, the definition of the module coalgebra symmetry for Hom-bialgebras reads as follows.

\begin{definition}\label{def-Hom-mod-coalg}
Let $(\C{H},\mu,\eta,\phi,\D,\ve,\psi)$ be a Hom-bialgebra, and $(\C{C},\D_\C{C},\ve_\C{C},\b)$ be a Hom-coalgebra together with a Hom-coalgebra endomorphism $\g:\C{C}\to \C{C}$ so that $(\C{C},\g)$ is a left $\C{H}$-Hom-module via $\rt:\C{H}\ot \C{C} \to \C{C}$. The pair $(\C{C},\g)$ is called a $\C{H}$-Hom-module coalgebra if the Hom-coalgebra structure maps are $\C{H}$-Hom-module morphisms, \ie
\begin{align}
& \b(h\rt c) = \psi(h) \rt \b(c), \label{Hom-mod-coalg-00} \\
& \D_\C{C}(h\rt c) = h\ns{1} \rt c\ns{1} \ot h\ns{2} \rt c\ns{2}, \label{Hom-mod-coalg-I} \\
& \ve_\C{C}(h\rt c) = \ve(h)\ve_\C{C}(c), \label{Hom-mod-coalg-II}
\end{align}
for any $h \in \C{H}$, and any $c\in \C{C}$.
\end{definition}

\subsection{Hom-comodule algebra}\label{subsect-comodule-alg}~

Let $(H,\mu,\eta,\D,\ve)$ be a bialgebra, and $(A,\mu_A,\eta_A)$ be an algebra so that $A$ is a (right) $H$-comodule algebra via $\nb:A \to A\ot H$. Let also $\phi,\psi:H\to H$ be two bialgebra homomorphsims, via which we construct the Hom-bialgebra $H_\phi^\psi$. Let, finally, $\a,\g:A\to A$ be two (commuting) algebra homomorphisms, so that,
\[
\nb(\gamma(a)) = \g(a\ps{0}) \ot \psi(a\ps{1}), \qquad \nb(\a(a)) = \a(a\ps{0}) \ot \phi(a\ps{1})
\]
for any $a\in A$. Then, setting
\[
\nb_{Hom}:A_\a \to A_\a\ot H_\phi^\psi, \qquad \nb_{Hom}(a) = a\ns{0} \ot a\ns{1} := \g(a\ps{0}) \ot \psi(a\ps{1}),
\]
we may endow $(A_\a,\g)$ with a (right) $H_\phi^\psi$-Hom-comodule structure. Indeed,
\begin{align}\label{comod-alg-compt}
\begin{split}
& ((\g \ot \D_{H_\phi^\psi}) \circ \nb_{Hom})(a) = (\g \ot \D_{H_\phi^\psi})(\g(a\ps{0}) \ot \psi(a\ps{1})) = \g^2(a\ps{0}) \ot \psi^2(a\ps{1}\ps{1}) \ot \psi^2(a\ps{1}\ps{2}) = \\
& \g^2(a\ps{0}\ps{0}) \ot \psi^2(a\ps{0}\ps{1}) \ot \psi^2(a\ps{1}) = \nb_{Hom}(\g(a\ps{0})) \ot \psi^2(a\ps{1}) = ((\nb_{Hom} \ot \psi)\circ \nb_{Hom}) (a). 
\end{split}
\end{align}

Similarly, defining the diagonal Hom-coaction by
\begin{align*}
& \nb_{Hom}^{diag}:A_\a \ot A_\a \to A_\a \ot A_\a \ot H_\phi^\psi, \\ &\nb_{Hom}^{diag}(a\ot a') := a\ns{0} \ot a'\ns{0} \ot a\ns{1} \bullet a'\ns{1} = \big(\g(a\ps{0}) \ot \g(a'\ps{0})\big) \ot \phi(\psi(a\ps{1})) \phi(\psi(a'\ps{1})),
\end{align*}
we may equip $(A_\a\ot A_\a, \g\ot \g)$ with a right $H_\phi^\psi$-Hom-comodule structure through 
\begin{align}\label{comod-alg-compt-II}
\begin{split}
& ((\g \ot \D_{H_\phi^\psi}) \circ \nb_{Hom}^{diag})(a\ot a') = (\g \ot \D_{H_\phi^\psi}) \Big(\g(a\ps{0}) \ot \g(a'\ps{0}) \ot \phi(\psi(a\ps{1})) \phi(\psi(a'\ps{1}))\Big) =\\
& \g^2(a\ps{0}) \ot \g^2(a'\ps{0}) \ot \phi(\psi^2(a\ps{1})) \phi(\psi^2(a'\ps{1})) \ot \phi(\psi^2(a\ps{2})) \phi(\psi^2(a'\ps{2})) = \\
& \g^2(a\ps{0}\ps{0}) \ot \g^2(a'\ps{0}\ps{0}) \ot \phi(\psi^2(a\ps{0}\ps{1})) \phi(\psi^2(a'\ps{0}\ps{1})) \ot \phi(\psi^2(a\ps{1})) \phi(\psi^2(a'\ps{1})) = \\
& \g^2\big(a\ps{0}\ps{0} \ot a'\ps{0}\ps{0}\big) \ot \phi(\psi^2(a\ps{0}\ps{1})) \phi(\psi^2(a'\ps{0}\ps{1})) \ot \phi(\psi^2(a\ps{1})) \phi(\psi^2(a'\ps{1})) = \\
& \nb_{Hom}^{diag}(\g(a\ps{0}) \ot \g(a'\ps{0})) \ot \phi(\psi^2(a\ps{1})) \phi(\psi^2(a'\ps{1})) = ((\nb_{Hom}^{diag} \ot \psi)\circ \nb_{Hom}^{diag}) (a \ot a').
\end{split}
\end{align}

Furthermore, the multiplication map $\mu_\a:(A_\a\ot A_\a, \g \ot \g) \to (A_\a,\g)$ is a map of $H_\phi^\psi$-Hom-comodules. Indeed,
\begin{align*}
& \left((\mu_\a\ot \Id)\circ\nb_{Hom}^{diag}\right)(a\ot a') = (\mu_\a\ot \Id)\Big(\g(a\ps{0}) \ot \g(a'\ps{0}) \ot \phi(\psi(a\ps{1}))\phi(\psi(a'\ps{1}))\Big) = \\
& \a(\g(a\ps{0}))\a(\g(a'\ps{0})) \ot \phi(\psi(a\ps{1}))\phi(\psi(a'\ps{1})) = \nb_{Hom}(\a(a)\a(a')) = (\nb_{Hom} \circ \mu_\a)(a \ot a').
\end{align*}
As for the unit map, we endow $(k,\Id)$ with the (right) $H_\phi^\psi$-Hom-comodule structure given by
\[
\nb_k: k \to k \ot H_\phi^\psi, \qquad \nb_k(r) := 1 \ot \eta(r),
\] 
to observe that
\begin{align*}
& \left((\eta_\a\ot \Id)\circ\nb_k\right)(r) = \a(\eta_A(1)) \ot \eta(r) = \g(\eta_A(1)) \ot \psi(\eta(r)) = (\nb_{Hom} \circ \eta_\a)(r).
\end{align*} 
That is, the unit map $\eta_\a:(k,\Id)\to (A_\a,\g)$ is also a map of $H_\phi^\psi$-Hom-comodules.

Accordingly, the module algebra symmetry for Hom-bialgebras is defined as follows; see also \cite[Def. 4.1]{Yau10}.

\begin{definition}
Let $(\C{H},\mu,\eta,\phi,\D,\ve,\psi)$ be a Hom-bialgebra, and $(\C{A},\mu_\C{A},\eta_\C{A},\a)$ be a Hom-algebra together with a Hom-algebra endomorphism $\g:\C{A}\to \C{A}$ so that $(\C{A},\g)$ is a right $\C{H}$-Hom-comodule via $\nb:\C{A} \to \C{A}\ot \C{H}$. The pair $(\C{A},\g)$ is called a $\C{H}$-Hom-comodule algebra if the Hom-algebra structure maps are $\C{H}$-Hom-comodule morphisms, \ie
\begin{align}
& \a(a)\ns{0} \ot \a(a)\ns{1} = \a(a\ns{0}) \ot \phi(a\ns{1}), \label{Hom-comod-alg-00}\\
& \nb(a\bullet a') = a\ns{0}\bullet a'\ns{0} \ot a\ns{1}\bullet a'\ns{1}, \label{Hom-comod-alg-I} \\
& \nb(\eta_\C{A}(1))= \eta_\C{A}(1)\ot \eta(1), \label{Hom-comod-alg-II}
\end{align}
for any $a,a'\in \C{A}$.
\end{definition}

\subsection{Hom-comodule coalgebra}\label{subsect-comodule-coalg}~

We shall consider, in the subsequent sections, the coaction of a Hom-Hopf algebra onto another Hom-Hopf algebra; the former being associative as an algebra, while the latter is coassociative as a coalgebra. As a result, we shall consider the following approach, via which we shall define the  ``comodule coalgebra'' symmetry for Hom-bialgebra.

Let $(H,\mu,\eta,\D,\ve)$ be a bialgebra, and $(C,\D_C,\ve_C)$ be a coalgebra so that $C$ is a (right) $H$-comodule coalgebra via $\nb:C \to C\ot H$. Let also $\phi,\psi:H\to H$ be two bialgebra homomorphsims, via which we construct the Hom-bialgebra $H_\phi^\psi$. Let, finally, $\b,\gamma:C\to C$ be two (commuting) coalgebra homomorphisms, so that,
\[
\nb(\gamma(c)) = \g(c\ps{0}) \ot \psi(c\ps{1}), \qquad \nb(\b(c)) = \b(c\ps{0}) \ot \phi(c\ps{1})
\]
for any $c\in C$. Then, setting
\[
\nb_{Hom}:C^\b \to C^\b\ot H_\phi^\psi, \qquad \nb_{Hom}(c) = c\ns{0} \ot c\ns{1} := \g(c\ps{0}) \ot \psi(c\ps{1}),
\]
we may endow $(C^\b,\g)$ with a (right) $H_\phi^\psi$-Hom-comodule structure. Indeed,
\begin{align}\label{comod-coalg-compt}
\begin{split}
& ((\g \ot \D_{H_\phi^\psi}) \circ \nb_{Hom})(c) = (\g \ot \D_{H_\phi^\psi})(\g(c\ps{0}) \ot \psi(c\ps{1})) = \g^2(c\ps{0}) \ot \psi^2(c\ps{1}\ps{1}) \ot \psi^2(c\ps{1}\ps{2}) = \\
& \g^2(c\ps{0}\ps{0}) \ot \psi^2(c\ps{0}\ps{1}) \ot \psi^2(c\ps{1}) = \nb_{Hom}(\g(c\ps{0})) \ot \psi^2(c\ps{1}) = ((\nb_{Hom} \ot \psi)\circ \nb_{Hom}) (c). 
\end{split}
\end{align}

Similarly, for the diagonal Hom-coaction defined by
\begin{align*}
& \nb_{Hom}^{diag}:C^\b \ot C^\b \to C^\b \ot C^\b \ot H_\phi^\psi, \\ &\nb_{Hom}^{diag}(c\ot c') := c\ns{0} \ot c'\ns{0} \ot c\ns{1} \bullet c'\ns{1} = \big(\g(c\ps{0}) \ot \g(c'\ps{0})\big) \ot \phi(\psi(c\ps{1})) \phi(\psi(c'\ps{1})),
\end{align*}
the pair $(C^\b\ot C^\b, \g\ot \g)$ becomes a right $H_\phi^\psi$-Hom-comodule. To this end, we observe that  
\begin{align}\label{comod-coalg-compt-II}
\begin{split}
& ((\g \ot \D_{H_\phi^\psi}) \circ \nb_{Hom}^{diag})(c\ot c') = (\g \ot \D_{H_\phi^\psi}) \Big(\g(c\ps{0}) \ot \g(c'\ps{0}) \ot \phi(\psi(c\ps{1})) \phi(\psi(c'\ps{1}))\Big) =\\
& \g^2(c\ps{0}) \ot \g^2(c'\ps{0}) \ot \phi(\psi^2(c\ps{1})) \phi(\psi^2(c'\ps{1})) \ot \phi(\psi^2(c\ps{2})) \phi(\psi^2(c'\ps{2})) = \\
& \g^2(c\ps{0}\ps{0}) \ot \g^2(c'\ps{0}\ps{0}) \ot \phi(\psi^2(c\ps{0}\ps{1})) \phi(\psi^2(c'\ps{0}\ps{1})) \ot \phi(\psi^2(c\ps{1})) \phi(\psi^2(c'\ps{1})) = \\
& \g^2\big(c\ps{0}\ps{0} \ot c'\ps{0}\ps{0}\big) \ot \phi(\psi^2(c\ps{0}\ps{1})) \phi(\psi^2(c'\ps{0}\ps{1})) \ot \phi(\psi^2(c\ps{1})) \phi(\psi^2(c'\ps{1})) = \\
& \nb_{Hom}^{diag}(\g(c\ps{0}) \ot \g(c'\ps{0})) \ot \phi(\psi^2(c\ps{1})) \phi(\psi^2(c'\ps{1})) = ((\nb_{Hom}^{diag} \ot \psi)\circ \nb_{Hom}^{diag}) (c \ot c').
\end{split}
\end{align}

Moreover, $\D_\b: (C^\b,\g) \to (C^\b\ot C^\b, \g\ot\g)$ is a morphism of (right) $H_\phi^\psi$-Hom-comodules. Indeed,
\begin{align*}
& (\nb_{Hom}^{diag}\circ \D_\b)(c) = \nb_{Hom}^{diag}(\b(c\ps{1})\ot \b(c\ps{2})) = \b\g(c\ps{1}\ps{0}) \ot \b\g(c\ps{2}\ps{0}) \ot \psi(\phi^2(c\ps{1}\ps{1})) \psi(\phi^2(c\ps{2}\ps{1})) = \\
& \b\g(c\ps{0}\ps{1}) \ot \b\g(c\ps{0}\ps{2}) \ot \psi(\phi^2(c\ps{1}))  = (\D_\b\ot \Id) (\g(c\ps{0}) \ot \psi(\phi^2(c\ps{1}))) = \left((\D_\b\ot \phi^2)\circ\nb_{Hom}\right)(c).
\end{align*} 
Quite similarly, endowing $(k,\Id)$ with the (right) $H_\phi^\psi$-Hom-comodule structure given by
\[
\nb_k: k \to k \ot H_\phi^\psi, \qquad \nb_k(r) := 1 \ot \eta(r),
\] 
the counit $\ve_C:(C^\b,\g)\to (k,\Id)$ becomes a morphism of (right) $H_\phi^\psi$-Hom-comodules. Indeed,
\[
(\nb_k\circ\ve_C)(c) = 1 \ot \eta(\ve_C(c)) = \ve_C(c\ps{0}) \ot \psi(c\ps{1}) = \ve_C(\g(c\ps{0})) \ot \psi(c\ps{1}) = ((\ve_C\ot \Id)\circ\nb_{Hom})(c).
\]

\begin{definition}\label{def-Hom-comod-coalg}
Let $(\C{H},\mu,\eta,\phi,\D,\ve,\psi)$ be a Hom-bialgebra, and $(\C{C},\D_\C{C},\ve_\C{C},\b)$ be a Hom-coalgebra together with a Hom-coalgebra endomorphism $\g:\C{C}\to \C{C}$ so that $(\C{C},\g)$ is a right $\C{H}$-Hom-comodule via $\nb:\C{C} \to \C{C}\ot \C{H}$. The pair $(\C{C},\g)$ is called a $\C{H}$-Hom-comodule coalgebra if the Hom-coalgebra structure maps are $\C{H}$-Hom-comodule morphisms, \ie
\begin{align}
& \b(c)\ns{0} \ot \b(c)\ns{1} = \b(c\ns{0}) \ot \phi(c\ns{1}), \label{Hom-comod-coalg-00}\\
& c\ns{0}\ns{1} \ot c\ns{0}\ns{2} \ot \phi^2(c\ns{1}) = c\ns{1}\ns{0} \ot c\ns{2}\ns{0} \ot c\ns{1}\ns{1} \bullet c\ns{2}\ns{1}, \label{Hom-comod-coalg-I} \\
& \ve_\C{C}(c\ns{0})c\ns{1} = \eta(\ve_\C{C}(c)), \label{Hom-comod-coalg-II}
\end{align}
for any $c\in \C{C}$.
\end{definition}

\section{Double cross product Hom-Hopf algebras}\label{sect-Double-cross-product-Hom-Hopf}

The notions of ``bicrossproduct Hom-Hopf algebra'' and ``double cross product Hom-Hopf algebra'' have already been developed, but only for the Hom-Hopf algebras of type $(\a,\a)$ and $(\a,\a^{-1})$. As such, the theory fails to cover the universal enveloping Hom-Hopf algebras of Hom-Lie algebras. To be able to develop a theory including this important class of Hom-Hopf algebras, we shall develop in the remaining sections the notions of the ``matched pairs of Hom-Hopf algebras'' (which leads to double cross product Hom-Hopf algebras) and the ``mutual pairs of Hom-Hopf algebras'' (that yields the bicrossproduct Hom-Hopf algebras) for the Hom-Hopf algebras of an arbitrary $(\a,\b)$-type.

In the present section we shall confine ourselves with the ``matched pairs of Hom-Hopf algebras'',  and hence the bicrossed product Hom-Hopf algebras, postponing the ``mutual pairs of Hom-Hopf algebras'' (and the bicrossproduct Hom-Hopf algebras) to the next section. It is this section that  we achieve to show that the universal enveloping Hom-Hopf algebras of a matched pair of Hom-Lie algebras form a matched pair of Hom-Hopf algebras.

\subsection{The double cross product construction}\label{subsect-double-cross-product-construction}~

Let us begin with the mutual pairs of Hom-Hopf algebras, and the double cross product Hom-Hopf algebra construction. To this end, we shall adopt \cite[Def. 2.11]{LuWang16}.

\begin{definition}
$(\C{U},\mu_\C{U}, \eta_\C{U},\phi,\D_\C{U},\ve_\C{U},\psi,S_\C{U})$ and $(\C{V},\mu_\C{V},\eta_\C{V},\a,\D_\C{V},\ve_\C{V},\b,S_\C{V})$ being two Hom-Hopf algebras, the pair $(\C{U},\C{V})$ is called a ``matched pair of Hom-Hopf algebras'' if
\begin{itemize}
\item[(i)] $(\C{U},\phi)$ is a left $\C{V}$-Hom-module coalgebra (via, say, $\rt:\C{V}\ot \C{U} \to \C{U}$) satisfying 
\begin{equation}\label{rt-phi-compatibility}
\phi(v\rt u) = \a(v)\rt \phi(u)
\end{equation} 
for any $u \in \C{U}$ and any $v\in \C{V}$, 
\item[(ii)] $(\C{V},\a)$ is a right $\C{U}$-Hom-module coalgebra (via, say, $\lt:\C{V}\ot \C{U} \to \C{V}$) satisfying 
\begin{equation}\label{lt-a-compatibility}
\a(v\lt u) = \a(v)\lt \phi(u)
\end{equation} 
for any $u \in \C{U}$ and any $v\in \C{V}$, 
\item[(iii)] for any $u,u'\in \C{U}$, and any $v,v'\in \C{V}$, 
\begin{align}
& v \rt (uu') = \big(\a^{-1}(\b^{-1}(v\ns{1})) \rt \psi^{-1}(u\ns{1})\big)\Big(\big(\a^{-2}(\b^{-1}(v\ns{2})) \lt \phi^{-1}(\psi^{-1}(u\ns{2}))\big)\rt u'\Big), \label{v-rt-uu'}\\
& (vv')\lt u = \Big(v \lt  \big(\a^{-1}(\b^{-1}(v'\ns{1})) \rt \phi^{-2}(\psi^{-1}(u\ns{1}))\big)\Big)\big(\b^{-1}(v'\ns{2}) \lt \phi^{-1}(\psi^{-1}(u\ns{2}))\big), \label{vv'-lt-u} \\
& v\ns{1}\lt u\ns{1} \ot v\ns{2}\rt u\ns{2} = v\ns{2}\lt u\ns{2} \ot v\ns{1}\rt u\ns{1}, \label{v-lt-u-ot-v-rt-u-switch} \\
& v \rt 1 = \ve_\C{V}(v)1, \qquad 1 \lt u = 1\ve_\C{U}(u). \label{actions-on-1}
\end{align}
\end{itemize}
\end{definition}

There is, then, a Hom-Hopf algebra structure on $\C{U}\ot \C{V}$, which is denoted by $\C{U}\bowtie \C{V}$ and is given by the following proposition. Compare with \cite[Prop. 2.12]{LuWang16}.

\begin{proposition}\label{prop-matched-pair-double-cross-prod}
Let $(\C{U},\mu_\C{U}, \eta_\C{U},\phi,\D_\C{U},\ve_\C{U},\psi,S_\C{U})$ and $(\C{V},\mu_\C{V},\eta_\C{V},\a,\D_\C{V},\ve_\C{V},\b,S_\C{V})$ be two Hom-Hopf algebras, so that $(\C{U},\C{V})$ is a matched pair of Hom-Hopf algebras. Then there is a Hom-Hopf algebra structure on $\big(\C{U}\bowtie \C{V} := \C{U}\ot \C{V}, \phi\ot \a, \psi\ot \b \big)$ given by
\begin{align}
& (u\ot v)(u'\ot v') = u\big(\a^{-1}(\b^{-1}(v\ns{1}))\rt \phi^{-1}(\psi^{-1}(u'\ns{1}))\big) \ot \big(\a^{-1}(\b^{-1}(v\ns{2}))\lt \phi^{-1}(\psi^{-1}(u'\ns{2}))\big)v', \label{dcp-multp}\\
& \D(u\ot v) = \big(u\ns{1}\ot v\ns{1}\big) \ot \big(u\ns{2}\ot v\ns{2}\big),\label{dcp-comultp} \\
& S(u\ot v) = \big(1\ot S(\a^{-1}(v))\big)\big(S(\phi^{-1}(u)) \ot 1\big),\label{dcp-antipode}
\end{align}
for any $u,u'\in \C{U}$, and any $v,v'\in \C{V}$.
\end{proposition}

\begin{proof}
In order to see that \eqref{dcp-multp} is Hom-associative, we note on one hand that 
\begin{align*}
& (\phi(u)\ot \a(v))\big((u'\ot v')(u''\ot v'')\big) = \\
& (\phi(u)\ot \a(v))\Big(u'\big(\a^{-1}(\b^{-1}(v'\ns{1}))\rt \phi^{-1}(\psi^{-1}(u''\ns{1}))\big) \ot \big(\a^{-1}(\b^{-1}(v'\ns{2}))\lt \phi^{-1}(\psi^{-1}(u''\ns{2}))\big)v''\Big) = \\
& \bigg(\phi(u)\Big(\b^{-1}(v\ns{1}) \rt \Big[\phi^{-1}(\psi^{-1}(u'\ns{1}))\big(\a^{-2}(\b^{-2}(v'\ns{1}\ns{1})) \rt \phi^{-2}(\psi^{-2}(u''\ns{1}\ns{1}))\big)\Big]\Big) \ot \\
& \Big(\b^{-1}(v\ns{2})\lt \Big[\phi^{-1}(\psi^{-1}(u'\ns{2}))\big(\a^{-2}(\b^{-2}(v'\ns{1}\ns{2})) \rt \phi^{-2}(\psi^{-2}(u''\ns{1}\ns{2}))\big)\Big]\Big)\times \\
&\Big[\big(\a^{-1}(\b^{-1}(v'\ns{2}))\lt \phi^{-1}(\psi^{-1}(u''\ns{2}))\big)v''\Big]\bigg),
\end{align*}
where, in the second equality we used \eqref{lt-a-compatibility}, and the latter is equivalent, in view of the Hom-associativity of the right Hom-action together with \eqref{v-rt-uu'}, to 
\begin{align*}
& \bigg(\phi(u)\Big(\big(\a^{-1}(\b^{-2}(v\ns{1}\ns{1})) \rt \phi^{-1}(\psi^{-2}(u'\ns{1}\ns{1}))\big) \times \\
&\Big[\big(\a^{-2}(\b^{-2}(v\ns{1}\ns{2})) \lt \phi^{-2}(\psi^{-2}(u'\ns{1}\ns{2}))\big) \rt\big(\a^{-2}(\b^{-2}(v'\ns{1}\ns{1})) \rt \phi^{-2}(\psi^{-2}(u''\ns{1}\ns{1}))\big)\Big]\Big) \ot \\
& \Big[\big(\a^{-1}(\b^{-1}(v\ns{2}))\lt \phi^{-1}(\psi^{-1}(u'\ns{2}))\big) \lt \big(\a^{-1}(\b^{-2}(v'\ns{1}\ns{2})) \rt \phi^{-1}(\psi^{-2}(u''\ns{1}\ns{2}))\big)\Big]\times \\
&\Big[\big(\a^{-1}(\b^{-1}(v'\ns{2}))\lt \phi^{-1}(\psi^{-1}(u''\ns{2}))\big)v''\Big]\bigg).
\end{align*}
Finally, imposing the Hom-associativity of $\C{U}$, we obtain
\begin{align*}
& (\phi(u)\ot \a(v))\big((u'\ot v')(u''\ot v'')\big) = \\
& \bigg(\Big[u\big(\a^{-1}(\b^{-2}(v\ns{1}\ns{1})) \rt \phi^{-1}(\psi^{-2}(u'\ns{1}\ns{1}))\big) \Big]\times \\
&\Big[\big(\a^{-1}(\b^{-2}(v\ns{1}\ns{2})) \lt \phi^{-1}(\psi^{-2}(u'\ns{1}\ns{2}))\big) \rt\big(\a^{-1}(\b^{-2}(v'\ns{1}\ns{1})) \rt \phi^{-1}(\psi^{-2}(u''\ns{1}\ns{1}))\big)\Big] \ot \\
& \Big[\big(\a^{-1}(\b^{-1}(v\ns{2}))\lt \phi^{-1}(\psi^{-1}(u'\ns{2}))\big) \lt \big(\a^{-1}(\b^{-2}(v'\ns{1}\ns{2})) \rt \phi^{-1}(\psi^{-2}(u''\ns{1}\ns{2}))\big)\Big]\times \\
&\Big[\big(\a^{-1}(\b^{-1}(v'\ns{2}))\lt \phi^{-1}(\psi^{-1}(u''\ns{2}))\big)v''\Big]\bigg).
\end{align*}
Now, on the other hand
\begin{align*}
& \big((u\ot v)(u'\ot v')\big)(\phi(u'')\ot \a(v'')) = \\
& \Big(u\big(\a^{-1}(\b^{-1}(v\ns{1}))\rt \phi^{-1}(\psi^{-1}(u'\ns{1}))\big) \ot \big(\a^{-1}(\b^{-1}(v\ns{2}))\lt \phi^{-1}(\psi^{-1}(u'\ns{2}))\big)v'\Big) (\phi(u'')\ot \a(v'')),
\end{align*}
where the latter implies, in view of \eqref{rt-phi-compatibility},
\begin{align*}
& \bigg(\Big[u\big(\a^{-1}(\b^{-1}(v\ns{1}))\rt \phi^{-1}(\psi^{-1}(u'\ns{1}))\big) \Big]\times\\
&\Big[\Big(\big(\a^{-2}(\b^{-2}(v\ns{2}\ns{1}))\lt \phi^{-2}(\psi^{-2}(u'\ns{2}\ns{1}))\big)\a^{-1}(\b^{-1}(v'\ns{1}))\Big) \rt \psi^{-1}(u''\ns{1})\Big] \ot \\
& \Big[\Big(\big(\a^{-2}(\b^{-2}(v\ns{2}\ns{2}))\lt \phi^{-2}(\psi^{-2}(u'\ns{2}\ns{2}))\big)\a^{-1}(\b^{-1}(v'\ns{2}))\Big) \lt \psi^{-1}(u''\ns{2})\Big]\a(v'')\bigg).
\end{align*}
Next, using the Hom-associativity of the Hom-action together with \eqref{vv'-lt-u}, we arrive at
\begin{align*}
& \bigg(\Big[u\big(\a^{-1}(\b^{-1}(v\ns{1}))\rt \phi^{-1}(\psi^{-1}(u'\ns{1}))\big) \Big]\times\\
&\Big[\big(\a^{-1}(\b^{-2}(v\ns{2}\ns{1}))\lt \phi^{-1}(\psi^{-2}(u'\ns{2}\ns{1}))\big) \rt \big(\a^{-1}(\b^{-1}(v'\ns{1}))\rt \phi^{-1}(\psi^{-1}(u''\ns{1}))\big)\Big] \ot \\
& \Big[\Big(\big(\a^{-2}(\b^{-2}(v\ns{2}\ns{2}))\lt \phi^{-2}(\psi^{-2}(u'\ns{2}\ns{2}))\big)\lt \big(\a^{-2}(\b^{-2}(v'\ns{2}\ns{1}))\rt \phi^{-2}(\psi^{-2}(u''\ns{2}\ns{1}))\big)\Big) \times \\
& \big(\a^{-1}(\b^{-2}(v'\ns{2}\ns{2}))\lt \phi^{-1}(\psi^{-2}(u''\ns{2}\ns{2}))\big)\Big]\a(v'')\bigg).
\end{align*}
Finally, in view of the Hom-associativity of $\C{V}$, we obtain
\begin{align*}
& \big((u\ot v)(u'\ot v')\big)(\phi(u'')\ot \a(v'')) = \\
& \bigg(\Big[u\big(\a^{-1}(\b^{-1}(v\ns{1}))\rt \phi^{-1}(\psi^{-1}(u'\ns{1}))\big) \Big]\times\\
&\Big[\big(\a^{-1}(\b^{-2}(v\ns{2}\ns{1}))\lt \phi^{-1}(\psi^{-2}(u'\ns{2}\ns{1}))\big) \rt \big(\a^{-1}(\b^{-1}(v'\ns{1}))\rt \phi^{-1}(\psi^{-1}(u''\ns{1}))\big)\Big] \ot \\
& \Big(\big(\a^{-1}(\b^{-2}(v\ns{2}\ns{2}))\lt \phi^{-1}(\psi^{-2}(u'\ns{2}\ns{2}))\big)\lt \big(\a^{-1}(\b^{-2}(v'\ns{2}\ns{1}))\rt \phi^{-1}(\psi^{-2}(u''\ns{2}\ns{1}))\big)\Big) \times \\
& \Big[\big(\a^{-1}(\b^{-2}(v'\ns{2}\ns{2}))\lt \phi^{-1}(\psi^{-2}(u''\ns{2}\ns{2}))\big)v''\Big]\bigg).
\end{align*}
As a result, the Hom-associativity of \eqref{dcp-multp} follows from the Hom-coassociativity of $\C{U}$ and the Hom-coassociativity of $\C{V}$.

As for \eqref{dcp-comultp} being multiplicative, we observe that
\begin{align*}
& \D\big((u\ot v)(u'\ot v')\big) = \\
& \D\Big(u\big(\a^{-1}(\b^{-1}(v\ns{1}))\rt \phi^{-1}(\psi^{-1}(u'\ns{1}))\big) \ot \big(\a^{-1}(\b^{-1}(v\ns{2}))\lt \phi^{-1}(\psi^{-1}(u'\ns{2}))\big)v'\Big) = \\
& \bigg(\Big(u\ns{1}\big(\a^{-1}(\b^{-1}(v\ns{1}\ns{1}))\rt \phi^{-1}(\psi^{-1}(u'\ns{1}\ns{1}))\big) \ot \big(\a^{-1}(\b^{-1}(v\ns{2}\ns{1}))\lt \phi^{-1}(\psi^{-1}(u'\ns{2}\ns{1}))\big)v'\ns{1}\Big) \ot \\
& \Big(u\ns{2}\big(\a^{-1}(\b^{-1}(v\ns{1}\ns{2}))\rt \phi^{-1}(\psi^{-1}(u'\ns{1}\ns{2}))\big) \ot \big(\a^{-1}(\b^{-1}(v\ns{2}\ns{2}))\lt \phi^{-1}(\psi^{-1}(u'\ns{2}\ns{2}))\big)v'\ns{2}\Big)\bigg),
\end{align*}
while on the other hand
\begin{align*}
& \D(u\ot v)\D(u'\ot v') = \\
& (u\ns{1}\ot v\ns{1})(u'\ns{1}\ot v'\ns{1}) \ot (u\ns{2}\ot v\ns{2})(u'\ns{2}\ot v'\ns{2}) = \\
& \Big(u\ns{1}\big(\a^{-1}(\b^{-1}(v\ns{1}\ns{1}))\rt \phi^{-1}(\psi^{-1}(u'\ns{1}\ns{1}))\big) \ot \big(\a^{-1}(\b^{-1}(v\ns{1}\ns{2}))\lt \phi^{-1}(\psi^{-1}(u'\ns{1}\ns{2}))\big)v'\ns{1}\Big) \ot \\
& \Big(u\ns{2}\big(\a^{-1}(\b^{-1}(v\ns{2}\ns{1}))\rt \phi^{-1}(\psi^{-1}(u'\ns{2}\ns{1}))\big) \ot \big(\a^{-1}(\b^{-1}(v\ns{2}\ns{2}))\lt \phi^{-1}(\psi^{-1}(u'\ns{2}\ns{2}))\big)v'\ns{2}\Big).
\end{align*}
On the other hand, it follows from the Hom-coassociativity of the comultiplications (both on $\C{U}$ and on $\C{V}$), we have
\begin{align}\label{flip-strategy}
\begin{split}
& v\ns{1}\ns{1} \ot v\ns{1}\ns{2} \ot v\ns{2}\ns{1} \ot v\ns{2}\ns{2} = \b(v\ns{1})\ot v\ns{2}\ns{1} \ot \b^{-1}(v\ns{2}\ns{2}\ns{1}) \ot \b^{-1}(v\ns{2}\ns{2}\ns{2} )= \\
& \b(v\ns{1})\ot \b^{-1}(v\ns{2}\ns{1}\ns{1}) \ot \b^{-1}(v\ns{2}\ns{1}\ns{2})\ot v\ns{2}\ns{2}.
\end{split}
\end{align}
As such,
\begin{align*}
& \D(u\ot v)\D(u'\ot v') = \\
& \Big(u\ns{1}\big(\a^{-1}(\b^{-1}(\b(v\ns{1})))\rt \phi^{-1}(\psi^{-1}(\psi(u'\ns{1})))\big) \ot \\
& \hspace{3cm} \big(\a^{-1}(\b^{-1}(\b^{-1}(v\ns{2}\ns{1}\ns{1})))\lt \phi^{-1}(\psi^{-1}(\psi^{-1}(u'\ns{2}\ns{1}\ns{1})))\big)v'\ns{1}\Big) \ot \\
& \Big(u\ns{2}\big(\a^{-1}(\b^{-1}(\b^{-1}(v\ns{2}\ns{1}\ns{2})))\rt \phi^{-1}(\psi^{-1}(\psi^{-1}(u'\ns{2}\ns{1}\ns{2})))\big) \ot  \\
& \hspace{5cm}\big(\a^{-1}(\b^{-1}(v\ns{2}\ns{2}))\lt \phi^{-1}(\psi^{-1}(u'\ns{2}\ns{2}))\big)v'\ns{2}\Big).
\end{align*}
Next, we make use of \eqref{v-lt-u-ot-v-rt-u-switch} to obtain
\begin{align*}
& \D(u\ot v)\D(u'\ot v') = \\
& \Big(u\ns{1}\big(\a^{-1}(\b^{-1}(\b(v\ns{1})))\rt \phi^{-1}(\psi^{-1}(\psi(u'\ns{1})))\big) \ot \\
& \hspace{3cm} \big(\a^{-1}(\b^{-1}(\b^{-1}(v\ns{2}\ns{1}\ns{2})))\lt \phi^{-1}(\psi^{-1}(\psi^{-1}(u'\ns{2}\ns{1}\ns{2})))\big)v'\ns{1}\Big) \ot \\
& \Big(u\ns{2}\big(\a^{-1}(\b^{-1}(\b^{-1}(v\ns{2}\ns{1}\ns{1})))\rt \phi^{-1}(\psi^{-1}(\psi^{-1}(u'\ns{2}\ns{1}\ns{1})))\big) \ot  \\
& \hspace{5cm}\big(\a^{-1}(\b^{-1}(v\ns{2}\ns{2}))\lt \phi^{-1}(\psi^{-1}(u'\ns{2}\ns{2}))\big)v'\ns{2}\Big),
\end{align*}
where the latter is equivalent, in view of \eqref{flip-strategy}, to
\begin{align*}
& \bigg(\Big(u\ns{1}\big(\a^{-1}(\b^{-1}(v\ns{1}\ns{1}))\rt \phi^{-1}(\psi^{-1}(u'\ns{1}\ns{1}))\big) \ot \big(\a^{-1}(\b^{-1}(v\ns{2}\ns{1}))\lt \phi^{-1}(\psi^{-1}(u'\ns{2}\ns{1}))\big)v'\ns{1}\Big) \ot \\
& \Big(u\ns{2}\big(\a^{-1}(\b^{-1}(v\ns{1}\ns{2}))\rt \phi^{-1}(\psi^{-1}(u'\ns{1}\ns{2}))\big) \ot \big(\a^{-1}(\b^{-1}(v\ns{2}\ns{2}))\lt \phi^{-1}(\psi^{-1}(u'\ns{2}\ns{2}))\big)v'\ns{2}\Big)\bigg) = \\
& \D\big((u\ot v)(u'\ot v')\big).
\end{align*}
Thus the multiplicativity of \eqref{dcp-comultp}. Finally we consider \eqref{dcp-antipode}, the antipode. To this end we observe that
\begin{align*}
& S\big(\phi(u\ns{1}) \ot \a(v\ns{1})\big) \big(\phi(u\ns{2}) \ot \a(v\ns{2})\big)  = \Big[\big(1\ot S(v\ns{1})\big)\big(S(u\ns{1})\ot 1\big)\Big]\big(\phi(u\ns{2}) \ot \a(v\ns{2})\big) = \\
& \big(1\ot S(\a(v\ns{1}))\big)\Big[\big(S(u\ns{1})\ot 1)\big)\big(u\ns{2} \ot v\ns{2}\big)\Big] = \\
& \big(1\ot S(\a(v\ns{1}))\big)\Big(S(u\ns{1})\big(\a^{-1}(\b^{-1}(1))\rt \phi^{-1}(\psi^{-1}(u\ns{2}\ns{1}))\big) \ot \big(\a^{-1}(\b^{-1}(1))\lt \phi^{-1}(\psi^{-1}(u\ns{2}\ns{2}))\big)v\ns{2}\Big) = \\
& \big(1\ot S(\a(v\ns{1}))\big)\Big(S(u\ns{1})\psi^{-1}(u\ns{2}\ns{1}) \ot \ve(u\ns{2}\ns{2})1\bullet v\ns{2}\Big) = \\
& \big(1\ot S(\a(v\ns{1}))\big)\Big(S(u\ns{1})u\ns{2}\ot \a(v\ns{2})\Big) = \ve(u)\big(1\ot S(\a(v\ns{1}))\big)\big(1\ot \a(v\ns{2})\big)  = \\
& \ve(u) \big(1 \ve(v\ns{1}\ns{1})\ot S(\a(\b^{-1}(v\ns{1}\ns{2}))) \a(v\ns{2})\big) = \ve(u) \big(1\ot S(\a(v\ns{1})) \a(v\ns{2})\big) = \ve(u)\ve(v) \big(1\ot 1\big),
\end{align*}
where on the fourth and the seventh equalities we used \eqref{actions-on-1}. The claim then follows.
\end{proof}

\begin{remark}
We note that in the case that $\a = \b: \C{V}\to \C{V}$ and $\phi = \psi:\C{U}\to \C{U}$, \eqref{dcp-multp} reduces to the one given in \cite[Prop. 2.12]{LuWang16}, while \eqref{v-rt-uu'} and \eqref{vv'-lt-u}  become \cite[(2.6)]{LuWang16} and \cite[(2.5)]{LuWang16} respectively.
\end{remark}

\subsection{Universal enveloping Hom-Hopf algebras of matched pairs of Hom-Lie algebras}\label{subsect-univ-envlp-Hom-Hopf-Hom-Lie}~

We shall first recall the matched pair construction for a pair $(\G{g},\G{h})$ of (multiplicative) Hom-Lie algebras from \cite{ShenBai14}. Then, we shall sketch briefly the universal enveloping Hom-algebra construction from \cite{Laur-GengMakhTele18}. Finally, we shall observe that the universal enveloping Hom-algebras of a matched pair of Hom-Lie algebras form a mutual pair of Hom-Hopf algebras.

Let us begin with the representation of a (multiplicative) Hom-Lie algebra from \cite{Shen12}.

\begin{definition}
Let $(\G{g},[\,,\,],\phi)$ be a (multiplicative) Hom-Lie algebra, and $V$ be a vector space equipped with a linear map $\g:V\to V$. Then, the pair $(V,\g)$ is called a $\G{g}$-module if there is a linear mapping $\G{g}\ot V \to V$ such that
\begin{itemize}
\item[(i)] $\g(\xi\cdot v) = \phi(\xi) \cdot \g(v)$, and
\item[(ii)] $[\xi,\xi'] \cdot \g(v) = \phi(\xi) \cdot (\xi' \cdot v) - \phi(\xi') \cdot (\xi \cdot v)$
\end{itemize}
for any $\xi,\xi' \in \G{g}$, and any $v \in V$.
\end{definition}

We can, then, recall from \cite{ShenBai14} the definition of a matched pair of Hom-Lie algebras.

\begin{definition}
Let $(\G{g},[\,,\,],\phi)$ and $(\G{h},[\,,\,],\a)$ be two (multiplicative) Hom-Lie algebras such that $(\G{g},\phi)$ is a $\G{h}$-module via
\[
\rt:\G{h} \ot \G{g} \to \G{g}, \qquad \eta\ot \xi \mapsto \eta \rt \xi,
\]
and that $(\G{h},\a)$ is a $\G{g}$-module via
\[
\lt:\G{h} \ot \G{g} \to \G{h}, \qquad \eta\ot \xi \mapsto \eta \lt \xi.
\]
Then the pair $(\G{g},\G{h})$ is called a ``matched pair of Hom-Lie algebras'' if 
\begin{align}
& \a(\eta) \rt [\xi,\xi'] = [\eta\rt \xi, \phi(\xi')] + [\phi(\xi),\eta\rt \xi'] + (\eta \lt \xi) \rt \phi(\xi') - (\eta \lt \xi') \rt \phi(\xi), \label{matched-pair-Hom-Lie-alg-I}\\ 
& [\eta,\eta'] \lt \phi(\xi) = [\a(\eta), \eta'\lt \xi] + [\eta \lt \xi,\a(\eta')] + \a(\eta) \lt (\eta' \rt \xi) - \a(\eta') \lt (\eta \rt \xi),\label{matched-pair-Hom-Lie-alg-II}
\end{align}
for any $\xi,\xi' \in \G{g}$, and any $\eta,\eta' \in \G{h}$.
\end{definition}
It is then showed in \cite[Thm. 3.1]{ShenBai14} that $(\G{g} \bowtie \G{h}, \phi\times \a)$, where $\G{g}\bowtie \G{h} := \G{g}\oplus \G{h}$, is a Hom-Lie algebra via the bracket 
\[
[(\xi,\eta),(\xi',\eta')] = \Big([\xi,\xi'] + \eta \rt \xi' - \eta' \rt \xi, [\eta,\eta'] + \eta \lt \xi' - \eta' \lt \xi\Big)
\]
for any $\xi,\xi'\in \G{g}$, and any $\eta,\eta' \in \G{h}$, if and only if $(\G{g},\G{h})$ is a matched pair of Hom-Lie algebras.

On the rest of this subsection we shall consider the pair $(\C{U}(\G{g}),\C{U}(\G{h}))$ of Hom-Hopf algebras given two Hom-Lie algebras $(\G{g},\phi)$ and $(\G{h},\a)$ so that $(\G{g},\G{h})$ is a matched pair of Hom-Lie algebras.

We shall be concerned first with the (right) Hom-action of the Hom-algebra $(\B{T}^\G{g} / \C{I}^\G{g}, \vee, {\bf 1}, \G{a})$ on $(\G{h},\a)$. To this end, we let
\begin{equation}\label{Tg-action-I}
\eta \lt {\bf 1} := \a(\eta), \qquad \eta \lt (\vp_1,s,\xi) := \eta \lt \phi^s(\xi),
\end{equation}
and extend it to $\B{T}^\G{g}$ via
\begin{equation}\label{Tg-action-II}
\a(\eta)\lt (\vp\vee \vp') := (\eta\lt \vp)\lt \G{a}(\vp'')
\end{equation}
for any $\vp:= (\vp,s_1,\ldots,s_n,\xi_1,\ldots,\xi_n) \in B_n\ot \G{g}^{\ot \,n}$, and any $\vp':= (\vp',r_1,\ldots,r_m,\xi'_1,\ldots,\xi'_m) \in B_m\ot \G{g}^{\ot \,m}$. 

\begin{proposition}\label{prop-right-Tg-Hom-act-on-h}
Given two Hom-Lie algebras $(\G{g},\phi)$ and $(\G{h},\a)$, let $(\G{h},\a)$ be a (right) Hom-module over $(\G{g},\phi)$. Then, $(\G{h},\a)$ is a (right) Hom-module over the Hom-algebra $(\B{T}^\G{g} / \C{I}^\G{g}, \vee, {\bf 1}, \G{a})$.
\end{proposition}

\begin{proof}
Let us first note inductively from \eqref{Tg-action-I} and \eqref{Tg-action-II} that
\begin{equation}\label{Tg-action-III}
\a(\eta \lt \vp) = \a(\eta)\lt \G{a}(\vp)
\end{equation}
for any $\eta\in \G{h}$, and any $\vp \in \B{T}^\G{g}$, and that
\begin{align*}
& \a(\eta) \lt \Big[\G{a}(\vp)\vee (\vp' \vee \vp'')\Big] = (\eta\lt \G{a}(\vp)) \lt \G{a}(\vp' \vee \vp'') = \\
& \a(\a^{-1}(\eta) \lt \vp) \lt \Big[\G{a}(\vp') \vee \G{a}(\vp'')\Big] = \Big[(\a^{-1}(\eta) \lt \vp) \lt \G{a}(\vp')\Big] \lt \G{a}^2(\vp'') = \\
& \Big[\eta\lt (\vp \vee \vp')\Big] \lt \G{a}^2(\vp'') = \a(\eta) \lt \Big[(\vp\vee \vp' )\vee \G{a}(\vp'')\Big],
\end{align*}
for any $\vp,\vp',\vp'' \in \B{T}^\G{g}$. Here, we used \eqref{Tg-action-II} on the first, third, fouth, and the fifth equalities. On the second equality we used \eqref{Tg-action-III} as well as \cite[Lemma 4.2]{Laur-GengMakhTele18}. 

As a result, the operation thus defined pass to the quotient $\B{T}^\G{g} / \C{I}^\G{g}$ as a Hom-action.
\end{proof}

Moreover, the action pass also to the quotient by $\C{J}^\G{g}$.

\begin{corollary}\label{coroll-right-Tg-Hom-act-on-h}
Given two Hom-Lie algebras $(\G{g},\phi)$ and $(\G{h},\a)$, let $(\G{h},\a)$ be a (right) Hom-module over $(\G{g},\phi)$. Then, $(\G{h},\a)$ is a (right) Hom-module over the universal enveloping Hom-Hopf algebra $(\C{U}(\G{g}), \vee, {\bf 1}, \phi)$.
\end{corollary}

\begin{proof}
It suffaces to show that the (right) Hom-action over $\B{T}^\G{g} / \C{I}^\G{g}$ passes to the quotient by $\C{J}^\G{g}$. Indeed, we have
\[
\eta \lt (\vp_1,0,\phi^s(\xi)) = \eta \lt \phi^0(\phi^s(\xi)) = \eta \lt \phi^s(\xi) = \eta \lt (\vp_1,s,\xi),
\]
and
\begin{align*}
& \a(\eta)\lt (\vp_1,0,[\xi_1,\xi_2]) = \a(\eta) \lt [\xi_1,\xi_2] = \\
& (\eta \lt \xi_1) \lt \phi(\xi_2) - (\eta \lt \xi_2) \lt \phi(\xi_1) = \\
& \eta\lt \Big[(\vp_1,0,\xi_1) \vee (\vp_1,0,\xi_2)\Big] - \eta\lt \Big[(\vp_1,0,\xi_2) \vee (\vp_1,0,\xi_1)\Big] = \\
& \eta \lt \Big[(\vp_2,0,0,\xi_1,\xi_2) - (\vp_2,0,0,\xi_2,\xi_1)\Big].
\end{align*}
\end{proof}

We shall now proceed to a (left) Hom-action of $(\G{h},\a)$ on the universal enveloping Hom-Hopf algebra $(\C{U}(\G{g}), \vee, {\bf 1}, \phi, \D,\ve,\Id,S)$. To this end, we let
\begin{equation}\label{eqn-h-Hom-action-I}
\eta\rt{\bf 1} := 0, \qquad \eta\rt (\vp_1,s,\xi) := (\vp_1, s, \a^{-s}(\eta)\rt \xi),
\end{equation}
and extend it on $\B{T}^\G{g}$ according to \eqref{v-rt-uu'}, that is,
\begin{equation}\label{eqn-h-Hom-action-II}
\eta \rt (\vp \vee \vp') := (\a^{-1}(\eta)\rt \vp) \vee \G{a}(\vp') + \G{a}(\vp\ns{1}) \vee \Big[\Big(\a^{-2}(\eta)\lt \G{a}^{-1}(\vp\ns{2})\Big)\rt \vp'\Big]
\end{equation}
for any $\vp:= (\vp,s_1,\ldots,s_n,\xi_1,\ldots,\xi_n) \in B_n\ot \G{g}^{\ot \,n}$, and any $\vp':= (\vp',r_1,\ldots,r_m,\xi'_1,\ldots,\xi'_m) \in B_m\ot \G{g}^{\ot \,m}$. The right Hom-action here, is the one discussed in Proposition \ref{prop-right-Tg-Hom-act-on-h}, and we note also that $\G{a}:\B{T}^\G{g} \to \B{T}^\G{g}$ is invertible on $\D(\vp) = \vp\ns{1}\ot \vp\ns{2}$ for any $\vp\in B_n\ot\G{g}^{\ot\,n}$. 

 We shall need the following compatibility of the (right) Hom-action on $(\G{h},\a)$ discussed in Proposition \ref{prop-right-Tg-Hom-act-on-h} and Corollary \ref{coroll-right-Tg-Hom-act-on-h}.

\begin{lemma}\label{lemma-right-action-on-bracket}
Given two Hom-Lie algebras $(\G{g},\phi)$ and $(\G{h},\a)$, let $(\G{h},\a)$ be a (right) Hom-module over $(\G{g},\phi)$ satisfying \eqref{matched-pair-Hom-Lie-alg-II}, and $(\G{g},\phi)$ be a (left) Hom-module over the Hom-Lie algebra $(\G{h},\a)$. Then, the (right) Hom-action of the Hom algebra $(\B{T}^\G{g} / \C{I}^\G{g}, \vee, {\bf 1}, \G{a})$ on $(\G{h},\a)$ satisfies
\[
[\eta,\eta'] \lt \G{a}(\vp) = [\eta \lt \vp\ns{1}, \eta'\lt \vp\ns{2}] + \a(\eta) \lt (\eta' \rt \vp) - \a(\eta') \lt (\eta \rt \vp)
\]
for any $\eta,\eta' \in \G{h}$, and any $\vp \in \B{T}^\G{g} / \C{I}^\G{g}$.
\end{lemma}

\begin{proof}
Let us begin with
\begin{align*}
& [\eta,\eta'] \lt \G{a}(\vp_1,s,\xi) = [\eta,\eta'] \lt \phi^{s+1}(\xi) = \\
& [\a(\eta), \eta'\lt \phi^s(\xi)] + [\eta\lt \phi^s(\xi), \a(\eta')] + \a(\eta)\lt \Big(\eta'\rt \phi^s(\xi)\Big) - \a(\eta')\lt \Big(\eta\rt \phi^s(\xi)\Big) = \\
& [\eta \lt {\bf 1}, \eta'\lt (\vp_1,s,\xi)] + [\eta\lt (\vp_1,s,\xi), \eta' \lt {\bf 1}] +  \a(\eta)\lt \phi^s(\a^{-s}(\eta')\rt \xi) - \a(\eta')\lt \phi^s(\a^{-s}(\eta)\rt \xi) = \\
& [\eta \lt (\vp_1,s,\xi)\ns{1}, \eta'\lt (\vp_1,s,\xi)\ns{2}] + \a(\eta)\lt \Big(\eta'\rt (\vp_1,s,\xi)\Big) - \a(\eta')\lt \Big(\eta\rt (\vp_1,s,\xi)\Big), 
\end{align*}
for any $\eta,\eta'\in \G{h}$, and any $\xi\in \G{g}$. Then, inductively 
\begin{align*}
& [\eta,\eta'] \lt \G{a}(\vp\vee \vp') = \Big([\a^{-1}(\eta),\a^{-1}(\eta')] \lt \G{a}(\vp)\Big) \lt \G{a}^2(\vp') = \\
& \Big\{[\a^{-1}(\eta)\lt \vp\ns{1},\a^{-1}(\eta')\lt\vp\ns{2}]  + \eta \lt \Big(\a^{-1}(\eta')\rt\vp\Big) - \eta' \lt \Big(\a^{-1}(\eta)\rt\vp\Big)\Big\} \lt \G{a}^2(\vp') = \\
& [\Big(\a^{-1}(\eta)\lt \vp\ns{1}\Big)\lt \G{a}(\vp'\ns{1}), \Big(\a^{-1}(\eta')\lt\vp\ns{2}\Big) \lt \G{a}(\vp'\ns{2})] + \\
& \a(\a^{-1}(\eta)\lt \vp\ns{1}) \lt \Big[\Big(\a^{-1}(\eta')\lt \vp\ns{2}\Big) \rt \G{a}(\vp')\Big] - \a(\a^{-1}(\eta')\lt \vp\ns{1}) \lt \Big[\Big(\a^{-1}(\eta)\lt \vp\ns{2}\Big) \rt \G{a}(\vp')\Big]  + \\
& \Big[\eta \lt \Big(\a^{-1}(\eta')\rt\vp\Big)\Big] \lt \G{a}^2(\vp') - \Big[\eta' \lt \Big(\a^{-1}(\eta)\rt\vp\Big)\Big]\lt \G{a}^2(\vp') = \\
& [\eta\lt \Big(\vp\ns{1}\vee \vp'\ns{1}\Big), \eta'\lt \Big(\vp\ns{2}\vee \vp'\ns{2}\Big)] + \\
& \a(\eta) \lt  \Big\{\G{a}(\vp\ns{1}) \vee\Big[\Big(\a^{-2}(\eta')\lt \G{a}^{-1}(\vp\ns{2})\Big) \rt \vp'\Big]\Big\} - \a(\eta') \lt  \Big\{\G{a}(\vp\ns{1}) \vee\Big[\Big(\a^{-2}(\eta)\lt \G{a}^{-1}(\vp\ns{2})\Big) \rt \vp'\Big]\Big\} + \\
& \a(\eta) \lt \Big[\Big(\a^{-1}(\eta')\rt\vp\Big) \vee\G{a}(\vp')\Big] - \a(\eta') \lt \Big[\Big(\a^{-1}(\eta)\rt\vp\Big) \vee\G{a}(\vp')\Big] = \\
& [\eta \lt (\vp\vee \vp')\ns{1},\eta'\lt (\vp\vee \vp')\ns{2}]  + \a(\eta)\lt \Big(\eta' \rt (\vp\vee \vp')\Big) - \a(\eta')\lt \Big(\eta \rt (\vp\vee \vp')\Big)
\end{align*}
for any $\eta, \eta' \in \G{h}$, and any $\vp,\vp'\in \B{T}^\G{g}/ \C{I}^\G{g}$.
\end{proof}

\begin{lemma}\label{lemma-left-action-coprod}
Given two Hom-Lie algebras $(\G{g},\phi)$ and $(\G{h},\a)$, let $(\G{h},\a)$ be a (right) Hom-module over $(\G{g},\phi)$, and $(\G{g},\phi)$ be a (left) Hom-module over the Hom-Lie algebra $(\G{h},\a)$. Then, the (left) Hom-action on the Hom algebra $(\B{T}^\G{g} / \C{I}^\G{g},\G{a})$ satisfies
\begin{equation}\label{comultp-on-eta-vp}
\D(\eta\rt \vp) = \eta\rt \vp\ns{1}\ot \G{a}(\vp\ns{2}) + \G{a}(\vp\ns{1}) \ot \eta\rt \vp\ns{2},
\end{equation}
for any $\eta\in \G{h}$, and any $\vp \in \B{T}^\G{g} / \C{I}^\G{g}$.
\end{lemma}

\begin{proof}
Indeed, starting from
\begin{align*}
& \D(\eta \rt (\vp_1,s,\xi)) = \D(\vp_1,s,\a^{-s}(\eta)\rt\xi) = \\
& (\vp_1,s,\a^{-s}(\eta)\rt\xi) \ot {\bf 1} + {\bf 1}  \ot (\vp_1,s,\a^{-s}(\eta)\rt\xi) = \\
& \eta\rt (\vp_1,s,\xi)\ns{1} \ot \G{a}\Big((\vp_1,s,\xi)\ns{2}\Big) + \G{a}\Big((\vp_1,s,\xi)\ns{1}\Big) \ot \eta\rt (\vp_1,s,\xi)\ns{2},
\end{align*}
we find inductively, in view of the commutativity and the coassociativity of the coproduct,
\begin{align*}
& \D(\eta\rt (\vp\vee \vp')) = \\
&\D\Big(\Big[\a^{-1}(\eta) \rt \vp\Big]\vee \G{a}(\vp')\Big) + \D\Big(\G{a}(\vp\ns{1})\vee \Big[\Big(\a^{-2}(\eta)\lt \G{a}^{-1}(\vp\ns{2})\Big) \rt \vp'\Big]\Big) = \\
& \Big(\a^{-1}(\eta) \rt \vp\Big)\ns{1}\vee \G{a}(\vp'\ns{1}) \ot \Big(\a^{-1}(\eta) \rt \vp\Big)\ns{2}\vee \G{a}(\vp'\ns{2}) + \\
& \G{a}(\vp\ns{1}\ns{1})\vee \Big[\Big(\a^{-2}(\eta)\lt \G{a}^{-1}(\vp\ns{2})\Big) \rt \vp'\Big]\ns{1} \ot \G{a}(\vp\ns{1}\ns{2})\vee \Big[\Big(\a^{-2}(\eta)\lt \G{a}^{-1}(\vp\ns{2})\Big) \rt \vp'\Big]\ns{2} = \\
& \Big(\a^{-1}(\eta) \rt \vp\ns{1}\Big)\vee \G{a}(\vp'\ns{1}) \ot \G{a}(\vp\ns{2})\vee \G{a}(\vp'\ns{2}) + \\
& \hspace{2cm} \G{a}(\vp\ns{1})\vee \G{a}(\vp'\ns{1}) \ot \Big(\a^{-1}(\eta) \rt \vp\ns{2}\Big)\vee \G{a}(\vp'\ns{2}) + \\
& \G{a}(\vp\ns{1}\ns{1})\vee \Big[\Big(\a^{-2}(\eta)\lt \G{a}^{-1}(\vp\ns{2})\Big) \rt \vp'\ns{1}\Big] \ot \G{a}(\vp\ns{1}\ns{2})\vee \G{a}(\vp'\ns{2}) + \\
& \hspace{2cm}\G{a}(\vp\ns{1}\ns{1})\vee \G{a}(\vp'\ns{1}) \ot \G{a}(\vp\ns{1}\ns{2})\vee \Big[\Big(\a^{-2}(\eta)\lt \G{a}^{-1}(\vp\ns{2})\Big) \rt \vp'\ns{2}\Big] = \\
& \eta\rt (\vp\ns{1}\vee \vp'\ns{1}) \ot \Big( \G{a}(\vp\ns{2})\vee \G{a}(\vp'\ns{2})\Big) + \Big( \G{a}(\vp\ns{1})\vee \G{a}(\vp'\ns{1})\Big)  \ot \eta\rt (\vp\ns{2}\vee \vp'\ns{2}),
\end{align*}
for any $\eta\in \G{h}$, and any $\vp, \vp' \in \B{T}^\G{g} / \C{I}^\G{g}$.
\end{proof}

\begin{proposition}
Given two Hom-Lie algebras $(\G{g},\phi)$ and $(\G{h},\a)$, let $(\G{g},\phi)$ be a (left) Hom-module over $(\G{h},\a)$, and $(\G{h},\a)$ be a (right) Hom-module over $(\G{g},\phi)$ satisfying \eqref{matched-pair-Hom-Lie-alg-II}. Then, $(\B{T}^\G{g} / \C{I}^\G{g}, \G{a})$ is a (left) Hom-module over the Hom-Lie algebra $(\G{h},\a)$.
\end{proposition}

\begin{proof}
We shall observe that this operation do pass on the quotient, and that it is indeed a Hom-Lie algebra Hom-action. As for the former claim, we first need to notice (inductively) from \eqref{eqn-h-Hom-action-I} and \eqref{eqn-h-Hom-action-II} that
\begin{equation}\label{Hom-mod-gamma-property}
\G{a}(\eta\rt \vp) = \a(\eta) \rt \G{a}(\vp)
\end{equation}
for any $\vp:= (\vp,s_1,\ldots,s_n,\xi_1,\ldots,\xi_n) \in B_n\ot \G{g}^{\ot \,n}$. Then,
\begin{align*}
& \eta \rt \Big[\G{a}(\vp) \vee (\vp'\vee \vp''\Big)] = \\
& \Big(\a^{-1}(\eta)\rt \G{a}(\vp)\Big)\vee \Big(\G{a}(\vp')\vee \G{a}(\vp'')\Big) + \G{a}^2(\vp\ns{1}) \vee \Big[\Big(\a^{-2}(\eta)\lt \vp\ns{2}\Big)\rt (\vp' \vee \vp'')\Big] = \\
& \Big(\a^{-1}(\eta)\rt \G{a}(\vp)\Big)\vee \Big(\G{a}(\vp')\vee \G{a}(\vp'')\Big) + \G{a}^2(\vp\ns{1}) \vee \bigg\{\Big[\a^{-1}\Big(\a^{-2}(\eta)\lt \vp\ns{2}\Big) \rt \vp'\Big]\vee \G{a}(\vp'') + \\
&  \G{a}(\vp'\ns{1}) \vee \bigg(\Big[\a^{-2}\Big(\a^{-2}(\eta)\lt \vp\ns{2}\Big) \lt \G{a}^{-1}(\vp'\ns{2})\Big] \rt \vp''\bigg)\bigg\} = \\
& \Big(\a^{-1}(\eta)\rt \G{a}(\vp)\Big)\vee \Big(\G{a}(\vp')\vee \G{a}(\vp'')\Big) + \G{a}^2(\vp\ns{1}) \vee \bigg\{\Big[\Big(\a^{-3}(\eta)\lt \G{a}^{-1}(\vp\ns{2})\Big) \rt \vp'\Big]\vee \G{a}(\vp'') + \\
&  \G{a}(\vp'\ns{1}) \vee \bigg(\Big[\a^{-2}\Big(\Big(\a^{-2}(\eta)\lt \vp\ns{2}\Big) \lt \G{a}(\vp'\ns{2})\Big)\Big] \rt \vp''\bigg)\bigg\} = \\
& \G{a}\Big(\a^{-2}(\eta)\rt \vp\Big)\vee \Big(\G{a}(\vp')\vee \G{a}(\vp'')\Big) + \G{a}^2(\vp\ns{1}) \vee \bigg\{\Big[\Big(\a^{-3}(\eta)\lt \G{a}^{-1}(\vp\ns{2})\Big) \rt \vp'\Big]\vee \G{a}(\vp'') + \\
&  \G{a}(\vp'\ns{1}) \vee \bigg(\Big[\a^{-2}\Big(\a^{-1}(\eta)\lt \Big(\vp\ns{2}\vee \vp'\ns{2}\Big) \Big)\Big] \rt \vp''\bigg)\bigg\} =
\end{align*}
where in the first equality we needed \cite[Eqn. (9)]{Laur-GengMakhTele18}, on the third equality we used \eqref{Tg-action-III}, and finally on the fourth equality we needed \eqref{Tg-action-II} and \eqref{Hom-mod-gamma-property}. On the other hand,
\begin{align*}
& \eta \rt \Big[(\vp\vee \vp') \vee \G{a}(\vp'')\Big] = \\
& \Big(\a^{-1}(\eta) \rt (\vp\vee \vp')\Big) \vee \G{a}^2(\vp'') + \\
& \hspace{2cm} \G{a}(\vp\ns{1}\vee \vp'\ns{1}) \vee\Big[\Big(\a^{-2}(\eta)\lt \Big(\G{a}^{-1}(\vp\ns{2})\vee \G{a}^{-1}(\vp'\ns{2})\Big)\Big)\rt \G{a}(\vp'')\Big] = \\
& \bigg[\Big(\a^{-2}(\eta) \rt \vp\Big)\vee \G{a}(\vp') + \G{a}(\vp\ns{1})\vee \Big[\Big(\a^{-3}(\eta)\lt \G{a}^{-1}(\vp\ns{2})\Big) \rt \vp'\Big]\bigg]\vee \G{a}^2(\vp'') + \\
&  \hspace{2cm} \G{a}(\vp\ns{1}\vee \vp'\ns{1}) \vee\Big[\Big(\a^{-2}\Big(\eta \lt \Big[\G{a}(\vp\ns{2})\vee \G{a}(\vp'\ns{2})\Big]\Big)\Big)\rt \G{a}(\vp'')\Big] = \\
& \bigg[\Big(\a^{-2}(\eta) \rt \vp\Big)\vee \G{a}(\vp')+ \G{a}(\vp\ns{1})\vee \Big[\Big(\a^{-3}(\eta)\lt \G{a}^{-1}(\vp\ns{2})\Big) \rt \vp'\Big]\bigg]\vee \G{a}^2(\vp'') + \\
&  \hspace{2cm} \G{a}(\vp\ns{1}\vee \vp'\ns{1}) \vee \G{a}\Big(\Big[\a^{-2}\Big(\a^{-1}(\eta) \lt (\vp\ns{2}\vee \vp'\ns{2})\Big)\Big]\rt \vp''\Big)
\end{align*}
where we needed \cite[Lemma 4.2 \& Lemma 4.6]{Laur-GengMakhTele18} in the first equality, while in the second equality we used \eqref{Tg-action-III}, and for the third equality we used \eqref{Hom-mod-gamma-property}. As a result, if
\[
\eta \rt \Big[\G{a}(\vp) \vee (\vp'\vee \vp'')\Big] = \G{a}(A)\vee (B\vee C)
\]
for some $A,B,C \in \B{T}^\G{g}$, then
\[
\eta \rt \Big[(\vp\vee \vp') \vee \G{a}(\vp'')\Big] = (A\vee B)\vee \G{a}(C),
\]
or in other words $\C{I}^\G{g}$ is closed under this operation, and hence it passes onto the quotient $\B{T}^\G{g} / \C{I}^\G{g}$.

Next, we shall observe that 
\[
[\eta,\eta'] \rt \G{a}(\vp) = \a(\eta) \rt (\eta' \rt \vp) - \a(\eta') \rt (\eta \rt \vp)
\]
for any $\eta,\eta'\in \G{h}$, and any $\vp\in \B{T}^\G{g} / \C{I}^\G{g}$.

Indeed, if $\vp = (\vp_1,s,\xi)$, then
\begin{align*}
& [\eta,\eta'] \rt \G{a}(\vp_1,s,\xi) =  [\eta,\eta'] \rt (\vp_1,s,\phi(\xi)) =  (\vp_1,s,\a^{-s}([\eta,\eta'])\rt \phi(\xi)) = \\
& (\vp_1,s,[\a^{-s}(\eta),\a^{-s}(\eta')]\rt \phi(\xi)) = \\
& \Big(\vp_1,s, \a^{-s+1}(\eta) \rt \Big[\a^{-s}(\eta')\rt \xi\Big]\Big) - \Big(\vp_1,s, \a^{-s+1}(\eta') \rt \Big[\a^{-s}(\eta)\rt \xi\Big]\Big) = \\
& \a(\eta) \rt \Big(\vp_1,s, \a^{-s}(\eta')\rt \xi\Big) - \a(\eta') \rt \Big(\vp_1,s, \a^{-s}(\eta)\rt \xi\Big) = \\
& \a(\eta) \rt \Big[\eta'\rt (\vp_1,s,\xi)\Big] - \a(\eta') \rt \Big[\eta\rt (\vp_1,s,\xi)\Big],
\end{align*}
for any $\eta,\eta' \in \G{h}$, and any $\xi\in \G{g}$. Inductively, on the one hand we have
\begin{align*}
& [\eta,\eta'] \rt \G{a}(\vp\vee \vp') = \\
& \Big(\a^{-1}([\eta,\eta']) \rt \G{a}(\vp)\Big) \vee \G{a}^2(\vp') + \G{a}^2(\vp\ns{1}) \vee \Big[\Big(\a^{-2}([\eta,\eta']) \lt \vp\ns{2}\Big) \rt \G{a}(\vp')\Big], 
\end{align*}
while on the other hand
\begin{align*}
& \a(\eta) \rt \Big(\eta' \rt (\vp\vee\vp')\Big) = \\
& \a(\eta) \rt \bigg[\Big(\a^{-1}(\eta')\rt \vp\Big)\vee \G{a}(\vp') + \G{a}(\vp\ns{1})\vee\Big[\Big(\a^{-2}(\eta')\lt \G{a}^{-1}(\vp\ns{2})\Big)\rt \vp'\Big]\bigg] = \\
& \a(\eta) \rt \Big[\Big(\a^{-1}(\eta')\rt \vp\Big)\vee \G{a}(\vp')\Big] + \a(\eta) \rt \Big\{\G{a}(\vp\ns{1})\vee\Big[\Big(\a^{-2}(\eta')\lt \G{a}^{-1}(\vp\ns{2})\Big)\rt \vp'\Big]\Big\} = \\
& \Big[\eta\rt\Big(\a^{-1}(\eta')\rt \vp\Big)\Big]\vee \G{a}^2(\vp') + \\
&\hspace{2cm}\G{a}\Big((\a^{-1}(\eta')\rt \vp)\ns{1}\Big) \vee \Big[\Big(\eta^{-1}(\eta)\lt \G{a}^{-1}\Big((\a^{-1}(\eta')\rt \vp)\ns{2}\Big)\Big) \rt \G{a}(\vp')\Big] + \\
& \Big(\eta\rt \G{a}(\vp\ns{1})\Big) \vee \Big[\Big(\a^{-1}(\eta')\lt \vp\ns{2}\Big)\rt \G{a}(\vp')\Big] + \\
& \hspace{2cm}\G{a}^2(\vp\ns{1}\ns{1}) \vee \Big\{\Big(\a^{-1}(\eta)\lt \vp\ns{1}\ns{2}\Big) \rt \Big[\Big(\a^{-2}(\eta')\lt \G{a}^{-1}(\vp\ns{2})\Big)\rt \vp'\Big]\Big\},
\end{align*}
for any $\eta,\eta' \in \G{h}$, and any $\vp,\vp'\in \B{T}^\G{g} / \C{I}^\G{g}$.

It then follows at once from Lemma \ref{lemma-left-action-coprod} that
\begin{align*}
& \a(\eta) \rt \Big(\eta' \rt (\vp\vee\vp')\Big) = \\
& \hspace{2cm}\Big[\eta\rt\Big(\a^{-1}(\eta')\rt \vp\Big)\Big]\vee \G{a}^2(\vp') + \\
&\hspace{4cm}\Big(\eta'\rt \G{a}(\vp\ns{1})\Big) \vee \Big[\Big(\a^{-1}(\eta)\lt \vp\ns{2}\Big) \rt \G{a}(\vp')\Big] + \\
&\G{a}^2(\vp\ns{1}) \vee \Big[\Big(\a^{-1}(\eta)\lt \Big(\a^{-2}(\eta')\rt \G{a}^{-1}(\vp\ns{2})\Big)\Big) \rt \G{a}(\vp')\Big] + \\
& \hspace{4cm}\Big(\eta\rt \G{a}(\vp\ns{1})\Big) \vee \Big[\Big(\a^{-1}(\eta')\lt \vp\ns{2}\Big)\rt \G{a}(\vp')\Big] + \\
& \G{a}^2(\vp\ns{1}) \vee \Big\{\Big(\a^{-1}(\eta)\lt \vp\ns{2}\Big) \rt \Big[\Big(\a^{-2}(\eta')\lt \G{a}^{-1}(\vp\ns{3})\Big)\rt \vp'\Big]\Big\}.
\end{align*}
As a result,
\begin{align*}
& \a(\eta) \rt \Big(\eta' \rt (\vp\vee\vp')\Big) - \a(\eta') \rt \Big(\eta \rt (\vp\vee\vp')\Big) = \\
& \hspace{1cm}\Big([\a^{-1}(\eta),\a^{-1}(\eta') ]\rt \G{a}(\vp)\Big) \vee \G{a}^2(\vp') + \\
&\hspace{2cm} \G{a}^2(\vp\ns{1})\vee \Big([\a^{-2}(\eta)\lt \G{a}^{-1}(\vp\ns{2}), \a^{-2}(\eta')\lt \G{a}^{-1}(\vp\ns{3})] \rt \G{a}(\vp')\Big) + \\
& \G{a}^2(\vp\ns{1}) \vee \Big[\Big(\a^{-1}(\eta)\lt \Big(\a^{-2}(\eta')\rt \G{a}^{-1}(\vp\ns{2})\Big) - \a^{-1}(\eta')\lt \Big(\a^{-2}(\eta)\rt \G{a}^{-1}(\vp\ns{2})\Big)\Big) \rt \G{a}(\vp')\Big] = \\
& \hspace{1cm}\Big([\a^{-1}(\eta),\a^{-1}(\eta') ]\rt \G{a}(\vp)\Big) \vee \G{a}^2(\vp') + \\
& \hspace{2cm} \G{a}^2(\vp\ns{1})\vee \Big[\Big([\a^{-2}(\eta), \a^{-2}(\eta')]\lt \G{a}^{-1}(\vp\ns{2})\Big) \rt \G{a}(\vp')\Big] = \\
& [\eta,\eta'] \rt \G{a}(\vp\vee \vp'),
\end{align*}
where on the second equality we used Lemma \ref{lemma-right-action-on-bracket}.
\end{proof}

\begin{corollary}
Given a matched pair of Hom-Lie algebras $(\G{g},\phi)$ and $(\G{h},\a)$, the universal enveloping Hom-Hopf algebra $(\C{U}(\G{g}), \phi)$ is a (left) Hom-module over the Hom-Lie algebra $(\G{h},\a)$.
\end{corollary}

\begin{proof}
It suffices to observe that the ideal $\C{J}^\G{g}$ is closed under the Hom-action of $(\G{h},\a)$. Indeed,
\begin{align*}
& \eta\rt \Big[(\vp_1,s,\xi) - (\vp_1,0,\phi^s(\xi))\Big] = \eta\rt(\vp_1,s,\xi) - \eta\rt(\vp_1,0,\phi^s(\xi)) = \\
& (\vp_1,s,\a^{-s}(\eta)\rt \xi) - (\vp_1,0,\eta\rt\phi^s(\xi)) = (\vp_1,s,\a^{-s}(\eta)\rt \xi) - (\vp_1,0,\phi^s(\a^{-s}(\eta)\rt\xi)) \in \C{J}^\G{g},
\end{align*}
as well as
\begin{align*}
& \eta \rt \Big\{(\vp_2,0,0,\xi_1,\xi_2) - (\vp_2,0,0,\xi_2,\xi_1) - (\vp_1,0,[\xi_1,\xi_2])\Big\} = \\
& \eta \rt \Big\{(\vp_1,0,\xi_1)\vee (\vp_1,0,\xi_2) - (\vp_1,0,\xi_2)\vee (\vp_1,0,\xi_2) - (\vp_1,0,[\xi_1,\xi_2])\Big\} = \\
& \Big(\a^{-1}(\eta)\rt (\vp_1,0,\xi_1)\Big)\vee \G{a}(\vp_1,0,\xi_2) + \G{a}\Big((\vp_1,0,\xi_1)\ns{1}\Big) \vee \Big[\Big(\a^{-2}(\eta)\lt \G{a}^{-1}\Big((\vp_1,0,\xi_1)\ns{2}\Big)\Big) \rt (\vp_1,0,\xi_2)\Big] + \\
& -\Big(\a^{-1}(\eta)\rt (\vp_1,0,\xi_2)\Big)\vee \G{a}(\vp_1,0,\xi_1) - \G{a}\Big((\vp_1,0,\xi_2)\ns{1}\Big) \vee \Big[\Big(\a^{-2}(\eta)\lt \G{a}^{-1}\Big((\vp_1,0,\xi_2)\ns{2}\Big)\Big) \rt (\vp_1,0,\xi_1)\Big] + \\
& - (\vp_1,0,\eta\rt [\xi_1,\xi_2]) = \\
& (\vp_1,0,\a^{-1}(\eta)\rt \xi_1)\vee (\vp_1,0,\phi(\xi_2)) + (\vp_1,0,\phi(\xi_1)) \vee \Big(\a^{-1}(\eta) \rt (\vp_1,0,\xi_2)\Big) + \\
&\hspace{2cm} \Big(\a^{-1}(\eta)\lt (\vp_1,0,\xi_1)\Big) \rt (\vp_1,0,\phi(\xi_2)) - (\vp_1,0,\a^{-1}(\eta)\rt \xi_2)\vee (\vp_1,0,\phi(\xi_1)) + \\
& - (\vp_1,0,\phi(\xi_2)) \vee \Big(\a^{-1}(\eta) \rt (\vp_1,0,\xi_1)\Big) - \Big(\a^{-1}(\eta)\lt (\vp_1,0,\xi_2)\Big) \rt (\vp_1,0,\phi(\xi_1)) + \\
& - (\vp_1,0,\eta\rt [\xi_1,\xi_2]) = \\
& (\vp_2,0,0,\a^{-1}(\eta)\rt \xi_1,\phi(\xi_2)) + (\vp_2,0,0,\phi(\xi_1),\a^{-1}(\eta) \rt \xi_2) + \\
&\hspace{2cm}  (\vp_1,0,\Big(\a^{-1}(\eta)\lt \xi_1\Big) \rt\phi(\xi_2)) - (\vp_2,0,0,\a^{-1}(\eta)\rt \xi_2,\phi(\xi_1)) + \\
& - (\vp_2,0,0,\phi(\xi_2),\a^{-1}(\eta) \rt \xi_1)  -  (\vp_1,0,\Big(\a^{-1}(\eta)\lt \xi_2\Big)\rt \phi(\xi_1)) + \\
& - (\vp_1,0,\eta\rt [\xi_1,\xi_2]) = \\
& (\vp_2,0,0,\a^{-1}(\eta)\rt \xi_1,\phi(\xi_2)) - (\vp_2,0,0,\phi(\xi_2),\a^{-1}(\eta) \rt \xi_1) - (\vp_1,0,\eta\rt [\a^{-1}(\eta) \rt\xi_1,\phi(\xi_2)]) + \\
&  (\vp_2,0,0,\phi(\xi_1),\a^{-1}(\eta)\rt \xi_2) - (\vp_2,0,0,\a^{-1}(\eta) \rt \xi_2,\phi(\xi_1)) - (\vp_1,0,\eta\rt [\phi(\xi_1), \a^{-1}(\eta) \rt\xi_2]) \in \C{J}^\G{g},
\end{align*}
where the last equality is a result of \eqref{matched-pair-Hom-Lie-alg-I}.
\end{proof}

So far, we have collected a (right) $(\C{U}(\G{g}), \vee, {\bf 1},\phi)$ Hom-action on $(\G{h},\a)$, and a (left) $(\G{h},\a)$ Hom-action on $(\C{U}(\G{g}), \vee, {\bf 1},\phi)$. We shall now take a step further towards the $(\C{U}(\G{h}), \vee,{\bf 1},\a)$ Hom-action on $(\C{U}(\G{g}),\phi)$. 

\begin{proposition}\label{prop-universal-Hom-Hopf-mod-coalg}
Given a matched pair of Hom-Lie algebras $(\G{g},\phi)$ and $(\G{h},\a)$, the universal enveloping Hom-Hopf algebra $(\C{U}(\G{g}), \vee,{\bf 1},\phi,\D,\ve,\Id,S)$ is a (left) Hom-module coalgebra over the universal enveloping Hom-Hopf algebra $(\C{U}(\G{h}), \vee,{\bf 1},\a,\D,\ve,\Id,S)$.
\end{proposition}

\begin{proof}
It follows at once from Proposition \ref{prop-right-Tg-Hom-act-on-h} and Corollary \ref{coroll-right-Tg-Hom-act-on-h} that the (left) $(\G{h},\a)$ Hom-action on $(\C{U}(\G{g}),\phi)$ extends to a (left) $(\C{U}(\G{h}),\vee,{\bf 1},\a)$ Hom-action. As a result, it suffices to show that the Hom-module coalgebra compatibility conditions \eqref{Hom-mod-coalg-00},\eqref{Hom-mod-coalg-I}, and \eqref{Hom-mod-coalg-II} are satisfied. Furthermore, \eqref{Hom-mod-coalg-00} and \eqref{Hom-mod-coalg-II}  being straightforward, we shall be content with \eqref{Hom-mod-coalg-I}, that is,
\[
\D(\Om\rt \vp) = \Om\ns{1}\rt \vp\ns{1} \ot \Om\ns{2}\rt \vp\ns{2}
\]
for any $\Om \in \C{U}(\G{h})$, and any $\vp \in \C{U}(\G{g})$. 

To begin with, if $\Om = (\vp_1,s,\eta)$, then the result follows from \eqref{comultp-on-eta-vp}. Then inductively
\begin{align*}
& \D((\Om\vee \Om')\rt \vp) = \D(\a(\Om) \rt \Big(\Om' \rt \phi^{-1}(\vp)\Big) )= \\
& \a(\Om\ns{1}) \rt \Big(\Om' \rt \phi^{-1}(\vp)\Big)\ns{1}  \ot \a(\Om\ns{2}) \rt \Big(\Om' \rt \phi^{-1}(\vp)\Big)\ns{2} = \\
& \a(\Om\ns{1}) \rt \Big(\Om'\ns{1} \rt \phi^{-1}(\vp\ns{1})\Big)  \ot \a(\Om\ns{2}) \rt \Big(\Om'\ns{2} \rt \phi^{-1}(\vp\ns{2})\Big)   = \\
& (\Om\vee \Om')\ns{1}\rt \vp\ns{1} \ot (\Om\vee \Om')\ns{2}\rt \vp\ns{2}.
\end{align*}
\end{proof}

Finally we achieve the main result of the present subsection.

\begin{proposition}\label{prop-univ-envlp-mutual-pair}
Given a matched pair of Hom-Lie algebras $(\G{g},\phi)$ and $(\G{h},\a)$, the universal enveloping Hom-Hopf algebras $(\C{U}(\G{g}), \vee,{\bf 1},\phi,\D,\ve,\Id,S)$ and $(\C{U}(\G{h}), \vee,{\bf 1},\a,\D,\ve,\Id,S)$ form a matched pair of Hom-Hopf algebras.
\end{proposition}

\begin{proof}
As a result of Proposition \ref{prop-universal-Hom-Hopf-mod-coalg} now, we have $(\C{U}(\G{g}), \vee,{\bf 1},\phi,\D,\ve,\Id,S)$ as a left Hom-module coalgebra over $(\C{U}(\G{h}), \vee,{\bf 1},\a,\D,\ve,\Id,S)$, and similarly $(\C{U}(\G{h}), \vee,{\bf 1},\a,\D,\ve,\Id,S)$ as a right Hom-module coalgebra over $(\C{U}(\G{g}), \vee,{\bf 1},\phi,\D,\ve,\Id,S)$. Hence, we need to show that the compatibility conditions \eqref{v-rt-uu'}-\eqref{actions-on-1} for the mutual pair of Hom-Hopf algebras hold.

Among these, \eqref{v-lt-u-ot-v-rt-u-switch} follows immediately from the cocommutativity of the coproduct, and \eqref{actions-on-1} from the definition. The remaining conditions \eqref{v-rt-uu'} and \eqref{vv'-lt-u}, that is,
\[
\Om \rt (\vp\vee \vp') = \Big(\a^{-1}(\Om\ns{1}) \rt \vp\ns{1}\Big) \vee\Big[\Big(\a^{-2}(\Om\ns{2}) \lt \phi^{-1}(\vp\ns{2})\Big)\rt \vp'\Big],
\]
and
\[
(\Om\vee\Om')\lt \vp = \Big[\Om\lt \Big(\a^{-1}(\Om'\ns{1})\rt \phi^{-2}(\vp\ns{1})\Big)\Big]\vee\Big(\Om'\ns{2} \lt \phi^{-1}(\vp\ns{2})\Big)
\]
for any $\vp,\vp' \in \C{U}(\G{g})$, and any $\Om,\Om' \in \C{U}(\G{h})$, are similar. As such, it suffices to observe the former.

Let us note that if $\Om = (\vp_1,s,\eta)$, then the result follows already from \eqref{eqn-h-Hom-action-II}. Then, inductively
\begin{align*}
& (\Om\vee \Om') \rt (\vp\vee \vp') = \a(\Om)\rt \Big(\Om'\rt \phi^{-1}(\vp\vee \vp')\Big) = \\
& \a(\Om)\rt \bigg\{\Big(\a^{-1}(\Om'\ns{1})\rt \phi^{-1}(\vp\ns{1})\Big)\vee \Big[\Big(\a^{-2}(\Om'\ns{2}) \lt \phi^{-2}(\vp\ns{2})\Big)\rt \phi^{-1}(\vp')\Big]\bigg\} = \\
& \Big[\Om\ns{1}\rt \Big(\a^{-1}(\Om'\ns{1})\rt \phi^{-1}(\vp\ns{1})\Big)\Big] \vee \\
&  \bigg\{\Big[\a^{-1}(\Om)\lt \Big(\a^{-2}(\Om'\ns{2})\rt \phi^{-2}(\vp\ns{2})\Big)\Big]\rt \Big[\Big(\a^{-2}(\Om'\ns{3}) \lt \phi^{-2}(\vp\ns{3})\Big)\rt \phi^{-1}(\vp')\Big]\bigg\} = \\
& \Big[\Big(\a^{-1}(\Om\ns{1}) \vee\a^{-1}(\Om'\ns{1})\Big)\rt \vp\ns{1}\Big] \vee  \bigg\{\Big[\Big(\a^{-2}(\Om\ns{2}) \vee\a^{-2}(\Om'\ns{2})\Big)\lt \phi^{-1}(\vp\ns{2})\Big]\rt \vp'\bigg\},
\end{align*}
for any $\vp,\vp' \in \C{U}(\G{g})$, and any $\Om,\Om' \in \C{U}(\G{h})$. The claim thus follows.
\end{proof}

\section{Bicrossproduct Hom-Hopf algebras}\label{sect-bicrossproduct-Hom-Hopf}

\subsection{The bicrossproduct construction}\label{subsect-bicrossproduct-construction}~

In this subsection we shall rediscuss \cite[Thm. 2.6]{LuWang16} in the presence of the Hom-Hopf algebras of type $(\a,\b)$. To this end, we shall first focus on the underlying Hom-algebra structure of a bicrossproduct Hom-Hopf algebra. Similarly, we next consider the underlying  Hom-coalgebra structure, and then finally we discuss the compatibility conditions necessary to be able to proceed towards a Hom-Hopf algebra.

\begin{proposition}\label{prop-bicrossprod-alg}
Let $(\C{F},\mu_\C{F}, \eta_\C{F},\a,\D_\C{F},\ve_\C{F},\b)$ and $(\C{U},\mu_\C{U}, \eta_\C{U},\phi,\D_\C{U},\ve_\C{U},\psi)$ be two Hom-bialgebras, so that $(\C{F},\b)$ is a (left) $\C{U}$-Hom-module algebra via, say $\rt:\C{U}\ot \C{F} \to \C{F}$, satisfying 
\begin{equation}\label{rt-f-comp-0}
\b(u \rt f) = \phi(u)\rt \b(f).
\end{equation}  
Then, there is a (unital) Hom-algebra structure on $\big(\C{F}\rtimes \C{U} := \C{F}\ot \C{U}, \b\ot \phi\big)$ given by
\begin{equation}\label{Hom-mod-alg-multp}
(f,u)\ast (f',u') := \Big(\a^{-1}(\b(f))\star \big(\phi^{-1}(\psi^{-1}(u\ns{1})) \rt \a^{-1}(f')\big), \psi^{-1}(u\ns{2}) \bullet u'\Big),
\end{equation}
and
\begin{equation}\label{Hom-mod-alg-rtimes-unit}
\eta_{\rtimes}: k \to \C{F}\rtimes \C{U}, \qquad \eta_{\rtimes}(1) := \Big(\eta_\C{F}(1), \eta_\C{U}(1)\Big).
\end{equation}
\end{proposition}

\begin{proof}
On one hand we have
\begin{align*}
&\Big[(f,u)\ast(f',u')\Big]\ast(\b(f''),\phi(u'')) =  \\
& \Big(\a^{-1}(\b(f))\star \big(\phi^{-1}(\psi^{-1}(u\ns{1})) \rt \a^{-1}(f')\big), \psi^{-1}(u\ns{2}) \bullet u'\Big) \ast(\b(f''),\phi(u'')) = \\
& \bigg(\Big[\a^{-2}(\b^2(f))\star \Big(\psi^{-2}(u\ns{1}) \rt \a^{-2}(\b(f'))\Big)\Big]\star \Big(\big(\phi^{-1}(\psi^{-2}(u\ns{2}\ns{1}))\bullet \phi^{-1}(\psi^{-1}(u'\ns{1}))\big) \rt \a^{-1}(\b(f''))\Big), \\
& \Big[\psi^{-2}(u\ns{2}\ns{2})\bullet \psi^{-1}(u'\ns{2})\Big]\bullet \phi(u'')\bigg),
\end{align*}
while on the other hand
\begin{align*}
& (\b(f),\phi(u))\ast \Big[(f',u')\ast(f'',u'')\Big] = \\
& (\b(f),\phi(u))\ast \Big(\a^{-1}(\b(f'))\star \big(\phi^{-1}(\psi^{-1}(u'\ns{1})) \rt \a^{-1}(f'')\big), \psi^{-1}(u'\ns{2}) \bullet u''\Big) = \\
& \bigg(\a^{-1}(\b^2(f)) \star \Big\{\psi^{-1}(u\ns{1}) \rt \Big[\a^{-2}(\b(f'))\star \Big(\phi^{-1}(\psi^{-2}(u'\ns{1})) \rt \a^{-2}(f'')\Big)\Big]\Big\}, \\
& \phi(\psi^{-1}(u\ns{2})) \bullet \Big[\psi^{-1}(u'\ns{2}) \bullet u''\Big]\bigg) = \\
& \bigg(\a^{-1}(\b^2(f)) \star \Big\{\Big(\psi^{-3}(u\ns{1}\ns{1}) \rt \a^{-2}(\b(f'))\Big) \star \Big[\psi^{-3}(u\ns{1}\ns{2}) \rt\Big(\phi^{-1}(\psi^{-2}(u'\ns{1})) \rt \a^{-2}(f'')\Big)\Big]\Big\}, \\
& \psi^{-1}(u\ns{2}) \bullet \Big[\psi^{-1}(u'\ns{2}) \bullet u''\Big]\bigg) = \\
& \bigg(\a^{-1}(\b^2(f)) \star \Big\{\Big(\psi^{-3}(u\ns{1}\ns{1}) \rt \a^{-2}(\b(f'))\Big) \star \Big[\Big(\phi^{-1}(\psi^{-3}(u\ns{1}\ns{2})) \bullet\phi^{-1}(\psi^{-2}(u'\ns{1}))\Big) \rt \a^{-2}(f'')\Big]\Big\}, \\
& \psi^{-1}(u\ns{2}) \bullet \Big[\psi^{-1}(u'\ns{2}) \bullet u''\Big]\bigg) = \\
& \bigg(\Big\{\a^{-2}(\b^2(f)) \star \Big(\psi^{-3}(u\ns{1}\ns{1}) \rt \a^{-2}(\b(f'))\Big)\Big\} \star \Big[\Big(\phi^{-1}(\psi^{-2}(u\ns{1}\ns{2})) \bullet \phi^{-1}(\psi^{-1}(u'\ns{1}))\Big) \rt \a^{-1}(\b(f''))\Big], \\
& \psi^{-1}(u\ns{2}) \bullet \Big[\psi^{-1}(u'\ns{2}) \bullet u''\Big]\bigg).
\end{align*}
The equality then, follows at once from \eqref{Hom-coassoc}. As for the unitality, we simply observe
\begin{align*}
& (f,u)\ast (\eta_\C{F}(1),\eta_\C{U}(1)) = \Big(\a^{-1}(\b(f))\star \big(\phi^{-1}(\psi^{-1}(u\ns{1})) \rt \a^{-1}(\eta_\C{F}(1))\big), \psi^{-1}(u\ns{2}) \bullet \eta_\C{U}(1)\Big) = \\
& \Big(\a^{-1}(\b(f))\star \big(\phi^{-1}(\psi^{-1}(u\ns{1})) \rt \eta_\C{F}(1)\big), \psi^{-1}(u\ns{2}) \bullet \eta_\C{U}(1)\Big) = \\
& \Big(\a^{-1}(\b(f))\star \ve(\phi^{-1}(\psi^{-1}(u\ns{1})))\eta_\C{F}(1), \phi(\psi^{-1}(u\ns{2})) \Big) = \\
& \Big(\a^{-1}(\b(f))\star \eta_\C{F}(1), \phi(u) \Big) = (\b(f), \phi(u) ).
\end{align*}
Similarly, we may obtain
\[
(\eta_\C{F}(1),\eta_\C{U}(1)) \ast (f,u) = (\b(f), \phi(u) ).
\]
\end{proof}

Next comes the Hom-coalgebra structure, as promised.

\begin{proposition}\label{prop-bicrossprod-coalg}
Let $(\C{F},\mu_\C{F}, \eta_\C{F},\a,\D_\C{F},\ve_\C{F},\b)$ and $(\C{U},\mu_\C{U}, \eta_\C{U},\phi,\D_\C{U},\ve_\C{U},\psi)$ be two Hom-bialgebras, so that $(\C{U},\phi)$ is a (right) $\C{F}$-Hom-comodule coalgebra via, say $\nb:\C{U}\to \C{U} \ot \C{F}$, satisfying 
\begin{equation}\label{rt-f-comp-0}
\nb(\phi(u)) = \phi(u\ns{0})\ot \b(u\ns{1}).
\end{equation}  
Then, there is a (counital) Hom-coalgebra structure on $\big(\C{F}\cl \C{U} := \C{F}\ot \C{U}, \a\ot \psi\big)$ given by
\begin{equation}\label{Hom-crossed-coalg-comultp}
\D_{\cl}(f,u) := \Big(\a(\b^{-1}(f\ns{1})), \phi^{-1}(u\ns{1}\ns{0})\Big) \ot \Big(\b^{-1}(f\ns{2})\star\a^{-2}(u\ns{1}\ns{1}), u\ns{2}\Big),
\end{equation}
and
\begin{equation}\label{Hom-mod-alg-rtimes-unit}
\ve_{\cl}: \C{F}\cl \C{U}\to k, \qquad \ve_{\cl}(f, u) := \ve_\C{F}(f) \ve_\C{U}(u).
\end{equation}
\end{proposition}

\begin{proof}
Let us begin with the Hom-coassociativity \eqref{Hom-coassoc}. On one hand we have
\begin{align*}
& \D_{\cl}\Big(\a(\b^{-1}(f\ns{1})), \phi^{-1}(u\ns{1}\ns{0})\Big) \ot (\a\ot \psi)\Big(\b^{-1}(f\ns{2}) \star\a^{-2}(u\ns{1}\ns{1}), u\ns{2}\Big) = \\
& \Big(\a^2(\b^{-2}(f\ns{1}\ns{1})), \phi^{-2}(u\ns{1}\ns{0}\ns{1}\ns{0})\Big) \ot \\
& \hspace{2cm} \Big(\a(\b^{-2}(f\ns{1}\ns{2}))\star \a^{-2}(\b^{-1}(u\ns{1}\ns{0}\ns{1}\ns{1})), \phi^{-1}(u\ns{1}\ns{0}\ns{2})\Big) \ot  \\
&\hspace{7cm} \Big(\a(\b^{-1}(f\ns{2}))\star \a^{-1}(u\ns{1}\ns{1}), \psi(u\ns{2})\Big) = \\
& \Big(\a^2(\b^{-2}(f\ns{1}\ns{1})), \phi^{-1}(u\ns{1}\ns{1}\ns{0})\Big) \ot \\
& \hspace{2cm} \Big(\a(\b^{-2}(f\ns{1}\ns{2}))\star \a^{-2}(\b^{-1}(u\ns{1}\ns{1}\ns{1}\ns{1})), \phi^{-1}(u\ns{1}\ns{2}\ns{0})\Big) \ot  \\
&\hspace{4cm} \Big(\a(\b^{-1}(f\ns{2}))\star \Big[\a^{-3}(\b^{-1}(u\ns{1}\ns{1}\ns{1}\ns{2}))\star \a^{-3}(u\ns{1}\ns{2}\ns{1})\Big], \psi(u\ns{2})\Big) = \\
& \Big(\a^2(\b^{-1}(f\ns{1})), \phi^{-1}(\psi(u\ns{1}\ns{0}))\Big) \ot \\
& \hspace{2cm} \Big(\a(\b^{-2}(f\ns{2}\ns{1}))\star \a^{-1}(\b^{-1}(u\ns{1}\ns{1}\ns{1})), \phi^{-1}(u\ns{2}\ns{1}\ns{0})\Big) \ot  \\
&\hspace{4cm} \Big(\a(\b^{-2}(f\ns{2}\ns{2}))\star \Big[\a^{-2}(\b^{-1}(u\ns{1}\ns{1}\ns{2}))\star \a^{-3}(u\ns{2}\ns{1}\ns{1})\Big], u\ns{2}\ns{2}\Big) 
\end{align*}
where on the second equality we used the Hom-comodule coalgebra compatibility \eqref{Hom-comod-coalg-I}, while the third equality followed from \eqref{Hom-coassoc} and \eqref{Hom-comodule-compt}. Similarly, on the other hand, 
\begin{align*}
& (\a\ot \psi)\Big(\a(\b^{-1}(f\ns{1})), \phi^{-1}(u\ns{1}\ns{0})\Big) \ot \D_{\cl}\Big(\b^{-1}(f\ns{2}) \star \a^{-2}(u\ns{1}\ns{1}), u\ns{2}\Big) = \\
& \Big(\a^2(\b^{-1}(f\ns{1})), \phi^{-1}(\psi(u\ns{1}\ns{0}))\Big) \ot \\
& \hspace{2cm} \Big(\a^{-1}(\b^{-2}(f\ns{2}\ns{1})) \star \a^{-1}(\b^{-1}(u\ns{1}\ns{1}\ns{1})), \phi^{-1}(u\ns{2}\ns{1}\ns{0})\Big) \ot \\
& \hspace{4cm} \Big(\Big[\b^{-2}(f\ns{2}\ns{2}) \star \a^{-2}(\b^{-1}(u\ns{1}\ns{1}\ns{2}))\Big] \star \a^{-2}(u\ns{2}\ns{1}\ns{1}), u\ns{2}\ns{2}\Big).
\end{align*}
The Hom-coassociativity of \eqref{Hom-crossed-coalg-comultp} now follows at once from the Hom-associativity on $\C{F}$. As for the counitality, it suffices to recall the individual counitalities on $\C{U}$, and on $\C{F}$, as well as the counitality of the $\C{F}$-Hom-coaction on $(\C{U},\phi)$.
\end{proof}

\begin{definition}
$(\C{U},\mu_\C{U}, \eta_\C{U},\phi,\D_\C{U},\ve_\C{U},\psi,S_\C{U})$ and $(\C{F},\mu_\C{F},\eta_\C{F},\a,\D_\C{F},\ve_\C{F},\b,S_\C{F})$ being two Hom-Hopf algebras, the pair $(\C{F},\C{U})$ is called a ``mutual pair of Hom-Hopf algebras'' if
\begin{itemize}
\item[(i)] $(\C{F},\b)$ is a left $\C{U}$-Hom-module algebra (via, say, $\rt:\C{U}\ot \C{F} \to \C{F}$) satisfying 
\begin{equation}\label{rt-f-comp}
\b(u\rt f) = \phi(u)\rt \b(f)
\end{equation} 
for any $u \in \C{U}$ and any $f\in \C{F}$, 
\item[(ii)] $(\C{U},\phi)$ is a right $\C{F}$-Hom-comodule coalgebra (via, say, $\nb:\C{U}\to \C{U} \ot \C{F}$) satisfying 
\begin{equation}\label{lt-f-comp}
\nb(\phi(u)) = \phi(u\ns{0})\ot \b(u\ns{1})
\end{equation} 
for any $u \in \C{U}$, and 
\item[(iii)] for any $u,u'\in \C{U}$, any $f\in \C{F}$, 
\begin{align}
& \D_\C{F}(u\rt f) = \psi^{-1}(u\ns{1}\ns{0})\rt f\ns{1} \ot \a^{-4}\b^3(u\ns{1}\ns{1})\star \big[\phi(\psi^{-2}(u\ns{2}))\rt \a^{-1}(f\ns{2})\big], \label{comp-I}\\
& \ve_\C{F}(u\rt f) = \ve_\C{U}(u)\ve_\C{F}(f), \label{comp-II}\\
& \nb(u\bullet u') = \psi^{-1}(u\ns{1}\ns{0}) \bullet u'\ns{0} \ot \a^{-2}(\b(u\ns{1}\ns{1}))\star \big[\phi^{-1}(u\ns{2})\rt \a^{-1}(u'\ns{1})\big], \label{comp-III}\\
& u\ns{1}\ns{0} \ot \a^{-2}\b^{2}(u\ns{1}\ns{1})\star (u\ns{2} \rt f) = u\ns{2}\ns{0} \ot (u\ns{1} \rt f) \star \a^{-2}\b^{2}(u\ns{2}\ns{1}), \label{comp-IV}
\end{align}
\end{itemize}
\end{definition}

There is, then, a Hom-Hopf algebra structure on $\C{F}\ot \C{U}:=\C{F}\acl \C{U}$, which is given by the following proposition, compare with \cite[Thm. 2.6]{LuWang16}. 

\begin{proposition}\label{prop-bicrossed-prod-Hom-Hopf}
Let $(\C{F},\mu_\C{F}, \eta_\C{F},\a,\D_\C{F},\ve_\C{F},\b,S_\C{F})$ and $(\C{U},\mu_\C{U}, \eta_\C{U},\phi,\D_\C{U},\ve_\C{U},\psi,S_\C{U})$ be two Hom-Hopf algebras, so that $(\C{F},\C{U})$ is a mutual pair of Hom-Hopf algebras. Then there is a Hom-Hopf algebra structure on $(\C{F}\acl \C{U} := \C{F}\ot \C{U}, \b\ot \phi,\a\ot \psi)$ given by
\begin{align}
& (f,u)\ast (f',u') = \Big(\a^{-1}(\b(f))\star \big(\phi^{-1}(\psi^{-1}(u\ns{1})) \rt \a^{-1}(f')\big), \psi^{-1}(u\ns{2}) \bullet u'\Big), \label{bcp-multp}\\
& \eta_{\acl}: k \to \C{F}\acl \C{U}, \qquad \eta_{\acl}(1) := \eta_\C{F}(1)\ot \eta_\C{U}(1), \\
& \D_{\acl}(f\ot u) = \Big(\a(\b^{-1}(f\ns{1})), \phi^{-1}(u\ns{1}\ns{0})\Big) \ot \Big(\b^{-1}(f\ns{2})\star \a^{-2}(u\ns{1}\ns{1}), u\ns{2}\Big),\label{bcp-comultp} \\
& \ve_{\acl}: \C{F}\acl \C{U}\to k, \qquad \ve_{\acl}(f, u) := \ve_\C{F}(f) \ve_\C{U}(u), \\
& S_{\acl}(f\ot u) = \Big(1,S_\C{U}(\phi^{-2}(u\ns{0}))\Big)\ast \Big(S_\C{F}(\a^{-1}(\b^{-1}(f))\star \a^{-2}(\b^{-1}(u\ns{1}))), 1\Big),\label{bcp-antipode}
\end{align}
for any $u,u'\in \C{U}$, and any $f,f'\in \C{F}$.
\end{proposition}

\begin{proof}
We shall first check the condition (iii)(1)-(iii)(9) of Definition \ref{def-Hom-bialg-Hom-Hopf}, among which, those except (iii)(2) are straightforward. As for (iii)(2), we begin with
\begin{align*}
& \D_{\acl}\Big((f,u)\ast (f',u')\Big) = \\
& \D_{\acl}\Big(\a^{-1}(\b(f))\star \big(\phi^{-1}(\psi^{-1}(u\ns{1})) \rt \a^{-1}(f')\big), \psi^{-1}(u\ns{2}) \bullet u'\Big) = \\
& \Big(f\ns{1}\star \a\big(\b^{-1}\big(\big[\phi^{-1}(\psi^{-1}(u\ns{1}))\rt \a^{-1}(f')\big]\ns{1}\big)\big), \phi^{-1}\big((\psi^{-1}(u\ns{2})\bullet u')\ns{1}\ns{0}\big)\Big) \ot \\
& \Big(\Big[\a^{-1}(f\ns{2})\star \b^{-1}\big(\big[\phi^{-1}(\psi^{-1}(u\ns{1}))\rt \a^{-1}(f')\big]\ns{2}\big)\Big]\star \a^{-2}\big((\psi^{-1}(u\ns{2})\bullet u')\ns{1}\ns{1}\big), (\psi^{-1}(u\ns{2})\bullet u')\ns{2}\Big),
\end{align*}
and using \eqref{comp-I} we continue with
\begin{align*}
& = \Big(f\ns{1}\star \a\big(\b^{-1}\big(\big[\phi^{-1}(\psi^{-2}(u\ns{1}\ns{1}\ns{0}))\rt \a^{-1}(f'\ns{1})\big]\big)\big), \phi^{-1}\big((\psi^{-1}(u\ns{2})\bullet u')\ns{1}\ns{0}\big)\Big) \ot \\
& \Big(\Big[\a^{-1}(f\ns{2})\star \b^{-1}\big(\a^{-5}(\b^2(u\ns{1}\ns{1}\ns{1})) \star\big[\psi^{-3}(u\ns{1}\ns{2})\rt \a^{-2}(f'\ns{2})\big]\big)\Big]\star \a^{-2}\big((\psi^{-1}(u\ns{2})\bullet u')\ns{1}\ns{1}\big), \\
&\hspace{2cm} (\psi^{-1}(u\ns{2})\bullet u')\ns{2}\Big) \\
& = \Big(f\ns{1}\star \big[\phi^{-2}(\psi^{-1}(u\ns{1}\ns{1}\ns{0}))\rt \b^{-1}(f'\ns{1})\big], \phi^{-1}\big((\psi^{-1}(u\ns{2}\ns{1})\bullet u'\ns{1})\ns{0}\big)\Big) \ot \\
& \Big(\Big[\a^{-1}(f\ns{2})\star \Big\{\a^{-5}(\b(u\ns{1}\ns{1}\ns{1})) \star\big[\phi^{-1}(\psi^{-3}(u\ns{1}\ns{2}))\rt \a^{-2}(\b^{-1}(f'\ns{2}))\big]\Big\}\Big]\star \\
& \hspace{10cm} \a^{-2}\big((\psi^{-1}(u\ns{2}\ns{1})\bullet u'\ns{1})\ns{1}\big), \\
&\hspace{2cm} (\psi^{-1}(u\ns{2}\ns{2})\bullet u'\ns{2})\Big).
\end{align*}
On the next move we employ \eqref{comp-III} and the Hom-associativity on $\C{F}$ to arrive at
\begin{align*}
& \D_{\acl}\Big((f,u)\ast (f',u')\Big) = \\
&\Big(f\ns{1}\star \big[\phi^{-2}(\psi^{-1}(u\ns{1}\ns{1}\ns{0}))\rt \b^{-1}(f'\ns{1})\big], \phi^{-1}(\psi^{-2}(u\ns{2}\ns{1}\ns{1}\ns{0}))\bullet \phi^{-1}(u'\ns{1}\ns{0})\Big) \ot \\
& \Big(\Big[\a^{-1}(f\ns{2})\star \Big\{\a^{-5}(\b(u\ns{1}\ns{1}\ns{1})) \star\big[\phi^{-1}(\psi^{-3}(u\ns{1}\ns{2}))\rt \a^{-2}(\b^{-1}(f'\ns{2}))\big]\Big\}\Big]\star \\
& \hspace{5cm} \Big\{\a^{-5}(\b(u\ns{2}\ns{1}\ns{1}\ns{1}))\star \big[\phi^{-1}(\psi^{-3}(u\ns{2}\ns{1}\ns{2}))\rt \a^{-3}(u'\ns{1}\ns{0})\big]\Big\}, \\
&\hspace{2cm} (\psi^{-1}(u\ns{2}\ns{2})\bullet u'\ns{2})\Big) = \\
&\Big(f\ns{1}\star \big[\phi^{-2}(\psi^{-1}(u\ns{1}\ns{1}\ns{0}))\rt \b^{-1}(f'\ns{1})\big], \phi^{-1}(\psi^{-2}(u\ns{2}\ns{1}\ns{1}\ns{0}))\bullet \phi^{-1}(u'\ns{1}\ns{0})\Big) \ot \\
& \Big(\Big[\a^{-1}(f\ns{2})\star \a^{-4}(\b(u\ns{1}\ns{1}\ns{1})) \Big] \star \\
&\Big\{ \Big(\big[\phi^{-1}(\psi^{-3}(u\ns{1}\ns{2}))\rt \a^{-2}(\b^{-1}(f'\ns{2}))\big]\star \a^{-6}(\b(u\ns{2}\ns{1}\ns{1}\ns{1}))\Big)\star  \big[\phi^{-1}(\psi^{-3}(u\ns{2}\ns{1}\ns{2}))\rt \a^{-3}(u'\ns{1}\ns{1})\big]\Big\}, \\
&\hspace{2cm} (\psi^{-1}(u\ns{2}\ns{2})\bullet u'\ns{2})\Big).
\end{align*}
Next, we recall \eqref{flip-strategy} and then the Hom-coassociativity on $\C{U}$ to proceed along
\begin{align*}
& \D_{\acl}\Big((f,u)\ast (f',u')\Big) = \\
& \Big(f\ns{1}\star \big[\phi^{-2}(u\ns{1}\ns{0})\rt \b^{-1}(f'\ns{1})\big], \phi^{-1}(\psi^{-3}(u\ns{2}\ns{1}\ns{2}\ns{1}\ns{0}))\bullet \phi^{-1}(u'\ns{1}\ns{0})\Big) \ot \\
& \Big(\Big[\a^{-1}(f\ns{2})\star \a^{-3}(\b(u\ns{1}\ns{1})) \Big] \star \\
&\hspace{1cm}\Big\{ \Big(\big[\phi^{-1}(\psi^{-4}(u\ns{2}\ns{1}\ns{1}))\rt \a^{-2}(\b^{-1}(f'\ns{2}))\big]\star \a^{-7}(\b(u\ns{2}\ns{1}\ns{2}\ns{1}\ns{1}))\Big)\star \\
&\hspace{7cm} \big[\phi^{-1}(\psi^{-4}(u\ns{2}\ns{1}\ns{2}\ns{2}))\rt \a^{-3}(u'\ns{1}\ns{1})\big]\Big\}, \\
&\hspace{2cm} (\psi^{-1}(u\ns{2}\ns{2})\bullet u'\ns{2})\Big) = \\
& \Big(f\ns{1}\star \big[\phi^{-2}(u\ns{1}\ns{0})\rt \b^{-1}(f'\ns{1})\big], \phi^{-1}(\psi^{-3}(u\ns{2}\ns{1}\ns{1}\ns{2}\ns{0}))\bullet \phi^{-1}(u'\ns{1}\ns{0})\Big) \ot \\
& \Big(\Big[\a^{-1}(f\ns{2})\star \a^{-3}(\b(u\ns{1}\ns{1})) \Big] \star \\
&\hspace{1cm}\Big\{ \Big(\big[\phi^{-1}(\psi^{-5}(u\ns{2}\ns{1}\ns{1}\ns{1}))\rt \a^{-2}(\b^{-1}(f'\ns{2}))\big]\star \a^{-7}(\b(u\ns{2}\ns{1}\ns{1}\ns{2}\ns{1}))\Big)\star \\
&\hspace{7cm} \big[\phi^{-1}(\psi^{-3}(u\ns{2}\ns{1}\ns{2}))\rt \a^{-3}(u'\ns{1}\ns{1})\big]\Big\}, \\
&\hspace{2cm} (\psi^{-1}(u\ns{2}\ns{2})\bullet u'\ns{2})\Big).
\end{align*}
At this point, we note that \eqref{comp-IV} implies
\begin{align*}
& \phi^{-1}\psi^{-3}(u''\ns{2}\ns{0}) \ot \Big[\phi^{-1}\psi^{-5}(u''\ns{1})\rt f''\Big]\star \a^{-7}\b(u''\ns{2}\ns{1}) = \\
& \psi^2\Big(\phi^{-1}\psi^{-5}(u''\ns{2})\ns{0}\Big) \ot \Big[\phi^{-1}\psi^{-5}(u''\ns{1})\rt f''\Big]\star \a^{-2}\b^2\Big(\phi^{-1}\psi^{-5}(u''\ns{2})\ns{1}\Big) = \\
& \psi^2\Big(\phi^{-1}\psi^{-5}(u''\ns{1})\ns{0}\Big) \ot \a^{-2}\b^2\Big(\phi^{-1}\psi^{-5}(u''\ns{1})\ns{1}\Big) \star \Big[\phi^{-1}\psi^{-5}(u''\ns{2})\rt f''\Big]  = \\
& \phi^{-1}\psi^{-3}(u''\ns{1}\ns{0}) \ot \a^{-7}\b(u''\ns{1}\ns{1}) \star \Big[\phi^{-1}\psi^{-5}(u''\ns{2})\rt f''\Big].
\end{align*}
Next, using \eqref{comp-IV} this way, as well as the Hom-coassociativity on $\C{U}$ once again, we obtain
\begin{align*}
& \D_{\acl}\Big((f,u)\ast (f',u')\Big) = \\
& \Big(f\ns{1}\star \big[\phi^{-2}(u\ns{1}\ns{0})\rt \b^{-1}(f'\ns{1})\big], \phi^{-1}(\psi^{-3}(u\ns{2}\ns{1}\ns{1}\ns{1}\ns{0}))\bullet \phi^{-1}(u'\ns{1}\ns{0})\Big) \ot \\
& \Big(\Big[\a^{-1}(f\ns{2})\star \a^{-3}(\b(u\ns{1}\ns{1})) \Big] \star \\
&\hspace{1cm}\Big\{ \Big(\a^{-7}(\b(u\ns{2}\ns{1}\ns{1}\ns{1}\ns{1})) \star \big[\phi^{-1}(\psi^{-5}(u\ns{2}\ns{1}\ns{1}\ns{2}))\rt \a^{-2}(\b^{-1}(f'\ns{2}))\big]\Big)\star \\
&\hspace{7cm} \big[\phi^{-1}(\psi^{-3}(u\ns{2}\ns{1}\ns{2}))\rt \a^{-3}(u'\ns{1}\ns{1})\big]\Big\}, \\
&\hspace{2cm} (\psi^{-1}(u\ns{2}\ns{2})\bullet u'\ns{2})\Big) = \\
& \Big(f\ns{1}\star \big[\phi^{-2}(u\ns{1}\ns{0})\rt \b^{-1}(f'\ns{1})\big], \phi^{-1}(\psi^{-2}(u\ns{2}\ns{1}\ns{1}\ns{0}))\bullet \phi^{-1}(u'\ns{1}\ns{0})\Big) \ot \\
& \Big(\Big[\a^{-1}(f\ns{2})\star \a^{-3}(\b(u\ns{1}\ns{1})) \Big] \star \\
&\hspace{1cm}\Big\{ \Big(\a^{-6}(\b(u\ns{2}\ns{1}\ns{1}\ns{1})) \star \big[\phi^{-1}(\psi^{-5}(u\ns{2}\ns{1}\ns{2}\ns{1}))\rt \a^{-2}(\b^{-1}(f'\ns{2}))\big]\Big)\star \\
&\hspace{7cm} \big[\phi^{-1}(\psi^{-4}(u\ns{2}\ns{1}\ns{2}\ns{2}))\rt \a^{-3}(u'\ns{1}\ns{1})\big]\Big\}, \\
&\hspace{2cm} (\psi^{-1}(u\ns{2}\ns{2})\bullet u'\ns{2})\Big),
\end{align*}
where the latter is the same as, in view of \eqref{flip-strategy},
\begin{align*}
& = \Big(f\ns{1}\star \big[\phi^{-2}(\psi^{-1}(u\ns{1}\ns{1}\ns{0}))\rt \b^{-1}(f'\ns{1})\big], \phi^{-1}(\psi^{-1}(u\ns{1}\ns{2}\ns{0}))\bullet \phi^{-1}(u'\ns{1}\ns{0})\Big) \ot \\
& \Big(\Big[\a^{-1}(f\ns{2})\star \a^{-4}(\b(u\ns{1}\ns{1}\ns{1})) \Big] \star \\
&\hspace{1cm}\Big\{ \Big(\a^{-5}(\b(u\ns{1}\ns{2}\ns{1})) \star \big[\phi^{-1}(\psi^{-4}(u\ns{2}\ns{1}\ns{1}))\rt \a^{-2}(\b^{-1}(f'\ns{2}))\big]\Big)\star \\
&\hspace{7cm} \big[\phi^{-1}(\psi^{-3}(u\ns{2}\ns{1}\ns{2}))\rt \a^{-3}(u'\ns{1}\ns{1})\big]\Big\}, \\
&\hspace{2cm} (\psi^{-1}(u\ns{2}\ns{2})\bullet u'\ns{2})\Big).
\end{align*}
We now use the Hom-associativity of $\C{F}$ once again to arrive at
\begin{align*}
& \D_{\acl}\Big((f,u)\ast (f',u')\Big) = \\
& \Big(f\ns{1}\star \big[\phi^{-2}(\psi^{-1}(u\ns{1}\ns{1}\ns{0}))\rt \b^{-1}(f'\ns{1})\big], \phi^{-1}(\psi^{-1}(u\ns{1}\ns{2}\ns{0}))\bullet \phi^{-1}(u'\ns{1}\ns{0})\Big) \ot \\
& \Big(\Big\{\a^{-1}(f\ns{2})\star\Big[\a^{-5}(\b(u\ns{1}\ns{1}\ns{1})) \star \a^{-5}(\b(u\ns{1}\ns{2}\ns{1})) \Big]\Big\} \star \\
&\hspace{1cm}\Big\{ \Big( \big[\phi^{-1}(\psi^{-3}(u\ns{2}\ns{1}\ns{1}))\rt \a^{-1}(\b^{-1}(f'\ns{2}))\big]\Big)\star\big[\phi^{-1}(\psi^{-3}(u\ns{2}\ns{1}\ns{2}))\rt \a^{-3}(u'\ns{1}\ns{1})\big]\Big\}, \\
&\hspace{2cm} (\psi^{-1}(u\ns{2}\ns{2})\bullet u'\ns{2})\Big),
\end{align*}
Finally, the Hom-module algebra compatibility \eqref{Hom-mod-alg-I} together with the  Hom-comodule coalgebra compatibility \eqref{Hom-comod-coalg-I} yield
\begin{align*}
& \D_{\acl}\Big((f,u)\ast (f',u')\Big) = \\
& \Big(f\ns{1}\star \big[\phi^{-2}(\psi^{-1}(u\ns{1}\ns{0}\ns{1}))\rt \b^{-1}(f'\ns{1})\big], \phi^{-1}(\psi^{-1}(u\ns{1}\ns{0}\ns{2}))\bullet \phi^{-1}(u'\ns{1}\ns{0})\Big) \ot \\
& \Big(\Big\{\a^{-1}(f\ns{2})\star\a^{-3}(\b(u\ns{1}\ns{1}))\Big\} \star \Big\{ \phi^{-1}(\psi^{-1}(u\ns{2}\ns{1}))\rt \big[\a^{-1}(\b^{-1}(f'\ns{2})) \star \a^{-3}(u'\ns{1}\ns{1})\big]\Big\}, \\
&\hspace{2cm} (\psi^{-1}(u\ns{2}\ns{2})\bullet u'\ns{2})\Big) = \\
& \Big(\a(\b^{-1}(f\ns{1})), \phi^{-1}(u\ns{1}\ns{0})\Big) \ast \Big(\a(\b^{-1}(f'\ns{1})), \phi^{-1}(u'\ns{1}\ns{0})\Big) \ot \\
& \Big(\b^{-1}(f\ns{2})\star \a^{-2}(u\ns{1}\ns{1}), u\ns{2}\Big) \ast \Big(\b^{-1}(f'\ns{2})\star \a^{-2}(u'\ns{1}\ns{1}), u'\ns{2}\Big) = \\
& \D_{\acl}(f,u) \ast \D_{\acl}(f',u').
\end{align*}
As for the antipode, we first observe that
\begin{align*}
& (\Id\ot S_{\acl})\circ \D_{\acl} (f,u) = \\
& \Big(\a(\b^{-1}(f\ns{1})), \phi^{-1}(u\ns{1}\ns{0})\Big) \ast S_{\acl}\Big(\b^{-1}(f\ns{2})\star \a^{-2}(u\ns{1}\ns{1}), u\ns{2}\Big) = \\
& \Big(\a(\b^{-1}(f\ns{1})), \phi^{-1}(u\ns{1}\ns{0})\Big) \ast  \\
&\bigg[\Big(1,S_\C{U}(\phi^{-2}(u\ns{2}\ns{0}))\Big) \ast \Big(S_\C{F}\big(\big[\a^{-1}(\b^{-2}(f\ns{2}))\star \a^{-3}(\b^{-1}(u\ns{1}\ns{1}))\big]\star \a^{-2}(\b^{-1}(u\ns{2}\ns{1}))\big),1\Big)\bigg].
\end{align*}
Then, in view of the Hom-associativity on $\C{F}$ and on $\C{F}\acl \C{U}$, we may proceed along
\begin{align*}
& =\bigg[\Big(\a(\b^{-2}(f\ns{1})), \phi^{-2}(u\ns{1}\ns{0})\Big) \ast  \Big(1,S_\C{U}(\phi^{-2}(u\ns{2}\ns{0}))\Big) \bigg]  \ast \\
&\Big(S_\C{F}\big(\big[\a^{-1}(\b^{-1}(f\ns{2}))\star \a^{-3}(u\ns{1}\ns{1})\big]\star \a^{-2}(u\ns{2}\ns{1})\big),1\Big) = \\
& =\bigg[\Big(\a(\b^{-2}(f\ns{1})), \phi^{-2}(u\ns{1}\ns{0})\Big) \ast  \Big(1,S_\C{U}(\phi^{-2}(u\ns{2}\ns{0}))\Big) \bigg]  \ast \\
&\Big(S_\C{F}\big(\b^{-1}(f\ns{2}) \star\big[\a^{-3}(u\ns{1}\ns{1}\star u\ns{2}\ns{1})\big]\big),1\Big).
\end{align*}
Next, recalling once again the Hom-comodule coalgebra compatibilities \eqref{Hom-comod-coalg-I} and \eqref{Hom-comod-coalg-II}, we may write
\begin{align*}
& (\Id\ot S_{\acl})\circ \D_{\acl} (f,u) = \\
& \bigg[\Big(\a(\b^{-2}(f\ns{1})), \phi^{-2}(u\ns{0}\ns{1})\Big) \ast  \Big(1,S_\C{U}(\phi^{-2}(u\ns{0}\ns{2}))\Big) \bigg]  \ast \Big(S_\C{F}\big(\b^{-1}(f\ns{2}) \star\a^{-1}(u\ns{1})\big),1\Big) = \\
& \Big(\a(\b^{-1}(f\ns{1})), \phi^{-2}(u\ns{0}\ns{1})\bullet S_\C{U}(\phi^{-2}(u\ns{0}\ns{2}))\Big) \ast \Big(S_\C{F}\big(\b^{-1}(f\ns{2}) \star\a^{-1}(u\ns{1})\big),1\Big) =\\
& \Big(\a(\b^{-1}(f\ns{1})), 1\Big) \ast \Big(S_\C{F}\big(\a(\b^{-1}(f\ns{2})) \big),1\Big) \,\ve_\C{U}(u) = \\
& \Big(f\ns{1} \star S_\C{F}\big(f\ns{2}\big)\Big) \,\ve_\C{U}(u) = \ve_\C{F}(f)\ve_\C{U}(u) \,\Big(1,1\Big).
\end{align*}
Similarly,
\begin{align*}
& (S_{\acl} \ot \Id)\circ \D_{\acl} (f,u) = \\
& S_{\acl}\Big(\a(\b^{-1}(f\ns{1})), \phi^{-1}(u\ns{1}\ns{0})\Big) \ast \Big(\b^{-1}(f\ns{2})\star \a^{-2}(u\ns{1}\ns{1}), u\ns{2}\Big) = \\
& \bigg[\Big(1,S_\C{U}\big(\phi^{-3}(u\ns{1}\ns{0}\ns{0})\big)\Big) \ast \Big(S_\C{F}\big(\b^{-2}(f\ns{1})\star \a^{-2}(\b^{-2}(u\ns{1}\ns{0}\ns{1}))\big), 1\Big)\bigg] \ast \\
& \Big(\b^{-1}(f\ns{2})\star \a^{-2}(u\ns{1}\ns{1}), u\ns{2}\Big) =  \\
& \Big(1,S_\C{U}\big(\phi^{-2}(u\ns{1}\ns{0}\ns{0})\big)\Big) \ast \\
& \bigg[\Big(S_\C{F}\big(\b^{-2}(f\ns{1})\star \a^{-2}(\b^{-2}(u\ns{1}\ns{0}\ns{1}))\big), 1\Big) \ast  \Big(\b^{-2}(f\ns{2})\star \a^{-2}(\b^{-1}(u\ns{1}\ns{1})), \phi^{-1}(u\ns{2})\Big)\bigg] = \\
& \Big(1,S_\C{U}\big(\phi^{-1}(u\ns{1}\ns{0})\big)\Big) \ast \\
& \bigg[\Big(S_\C{F}\big(\b^{-2}(f\ns{1})\star \a^{-2}(\b^{-2}(u\ns{1}\ns{1}\ns{1}))\big), 1\Big) \ast  \Big(\b^{-2}(f\ns{2})\star \a^{-2}(\b^{-2}(u\ns{1}\ns{1}\ns{2})), \phi^{-1}(u\ns{2})\Big)\bigg] =\\
 & \Big(1,S_\C{U}\big(\phi^{-1}(u\ns{1}\ns{0})\big)\Big) \ast \\
& \bigg[\Big(S_\C{F}\big(\a^{-1}(\b^{-1}(f\ns{1}))\star \a^{-3}(\b^{-1}(u\ns{1}\ns{1}\ns{1}))\big) \star [\a^{-1}(\b^{-1}(f\ns{2}))\star \a^{-3}(\b^{-1}(u\ns{1}\ns{1}\ns{2}))], u\ns{2}\Big)\bigg] = \\
& \Big(1,S_\C{U}\big(u\ns{1}\big)\Big) \ast  \Big(1, u\ns{2}\Big)\, \ve_\C{F}(f) = \Big(1,S_\C{U}\big(u\ns{1}\bullet u\ns{2}\big)\Big)\, \ve_\C{F}(f) = \ve_\C{F}(f)\ve_\C{F}(u) \, \Big(1,1\Big),
\end{align*}
where on the fourth equality we used the Hom-comodule coassociativity of the $\C{F}$-Hom-coaction on $\C{U}$.
\end{proof}

\subsection{Semidualization}\label{subsect-semidualization}~

Starting with a Hom-associative analogue of \cite[Prop. 1.6.11]{Majid-book}, for Hom-module coalgebra / Hom-module algebra duality, we show in the present subsection how a double cross product Hom-Hopf algebra gives rise to a bicrossproduct Hom-Hopf algebra. As a result, we construct  the Hom-Lie-Hopf algebras.

\begin{lemma}
Let $(\C{V},\mu,\eta,\phi,\D,\ve,\psi)$ be a Hom-bialgebra, and let $(\C{U},\g)$ be a (unital) left $\C{V}$-Hom-module, so that, $\g(v\rt u) = \phi(v)\rt \g(u)$. Then, $(\C{U},\g)$ is a (counital) right $\C{V}^\circ$-Hom-comodule via
\[
\nb_{Hom}^\circ:\C{U}\to \C{U}\ot \C{V}^\circ, \qquad \nb_{Hom}^\circ(u) := u\ns{0}\ot u\ns{1}
\]
given by
\begin{equation}\label{action-coaction-duality}
u\ns{0} \langle u\ns{1},v\rangle = \phi^{-2}(v)\rt u.
\end{equation}
Moreover, $\g:\C{U}\to \C{U}$ satisfies $\g(u)\ns{0} \ot \g(u)\ns{1} = \g(u\ns{0}) \ot (\phi^{-1})^\ast(u\ns{1})$ for any $u \in \C{U}$.
\end{lemma}

\begin{proof}
We shall observe that \eqref{action-coaction-duality} indeed defines a right coaction. To this end, we note that 
\begin{align*}
& \g(u\ns{0})\langle \D(u\ns{1}), v \ot v' \rangle = \g(u\ns{0})\langle u\ns{1}, \phi^{-2}(v\bullet v') \rangle = \g(\phi^{-4}(v\bullet v')\rt u) = \\
& \phi^{-3}(v\bullet v') \rt \g(u) = \phi^{-2}(v) \rt (\phi^{-3}(v')\rt u) = \phi^{-2}(v) \rt u\ns{0} \langle u\ns{1}, \phi^{-1}(v')\rangle = \\
& u\ns{0}\ns{0}\langle u\ns{0}\ns{1},v\rangle \langle u\ns{1}, \phi^{-1}(v')\rangle,
\end{align*}
where the third equality follows from the assumption, and the fourth equality is a result of the Hom-associativity of the action. The equality thus obtained is equivalent to that of 
\[
\g(u\ns{0})\ot \D(u\ns{1}) = u\ns{0}\ns{0} \ot u\ns{0}\ns{1} \ot (\phi^{-1})^\ast(u\ns{1}),
\]
which is the claim. The Hom-counitality follows similarly (from the Hom-unitality of the action).

As for the second claim, it suffices to observe that
\[
\g(u)\ns{0} \langle \g(u)\ns{1}, v \rangle = \phi^{-2}(v)\rt \g(u) = \g(\phi^{-3}(v)\rt u) = u\ns{0}\langle u\ns{1}, \phi^{-1}(v)\rangle.
\]
\end{proof}

\begin{proposition}
Let $(\C{V},\mu,\eta,\phi,\D,\ve,\psi)$ be a Hom-bialgebra, and let $(\C{U},\D_\C{U},\ve_\C{U},\b)$ be a Hom-coalgebra together with a Hom-coalgebra endomorphism $\g:\C{U}\to \C{U}$ via which $(\C{U},\g)$ is a left $\C{V}$-Hom-module coalgebra with respect to the action $\rt:\C{V}\ot \C{U} \to \C{U}$ that satisfies $\g(v\rt u) = \phi(v)\rt \g(u)$. Then, $(\C{U},\g)$ is a right $\C{V}^\circ$-Hom-comodule coalgebra.
\end{proposition}

\begin{proof}
Let us first observe that 
\begin{align*}
& \b(u)\ns{0} \langle \b(u)\ns{1},  v\rangle = \phi^{-2}(v) \rt \b(u) = \b(\psi^{-1}(\phi^{-2}(v)) \rt u) = \\
& \b(u\ns{0})\langle u\ns{1}, \psi^{-1}(v)\rangle =  \b(u\ns{0})\langle {\left(\psi^{-1}\right)}^\ast(u\ns{1}), v\rangle,
\end{align*}
where the second equality follows from \eqref{Hom-mod-coalg-00}. Hence, \eqref{Hom-comod-coalg-00} follows. 

As for \eqref{Hom-comod-coalg-I} we have,
\begin{align*}
& u\ns{0}\ns{1} \ot u\ns{0}\ns{2} \langle u\ns{1}, v \rangle = \D_\C{U}(\phi^{-2}(v)\rt u) = \phi^{-2}(v\ns{1}) \rt u\ns{1} \ot \phi^{-2}(v\ns{2}) \rt u\ns{2} = \\
& u\ns{1}\ns{0} \langle u\ns{1}\ns{1} , v\ns{1}\rangle \ot u\ns{2}\ns{0} \langle u\ns{2}\ns{2} , v\ns{2}\rangle =  u\ns{1}\ns{0} \ot u\ns{2}\ns{0} \langle u\ns{1}\ns{1} \bullet u\ns{2}\ns{1}, \psi^2(v)\rangle
\end{align*}
where the second equality is a result of \eqref{Hom-mod-coalg-I}. Similarly, \eqref{Hom-comod-coalg-II} follows directly from \eqref{Hom-mod-coalg-II}.
\end{proof}

\begin{lemma}
Let $(\C{U},\mu,\eta,\phi,\D,\ve,\psi)$ be a Hom-bialgebra, and let $(\C{V},\g)$ be a (unital) right $\C{U}$-Hom-module, so that, $\g(v\lt u) = \g(v)\lt \phi(u)$. Then, $(\C{V}^\circ,(\g^{-1})^\ast)$ is a (unital) left $\C{U}$-Hom-module via
\[
\overset{\circ}{\rt}:\C{U}\ot \C{V}^\circ \to \C{V}^\circ
\]
given by
\begin{equation}\label{action-coaction-duality-II}
\langle u\overset{\circ}{\rt} f,v\rangle = \langle f,\g^{-2}(v)\lt \phi^{-2}(u)\rangle.
\end{equation}
Moreover, $(\g^{-1})^\ast:\C{V}^\circ\to \C{V}^\circ$ satisfies $(\g^{-1})^\ast(u \overset{\circ}{\rt} f) = \phi(u) \overset{\circ}{\rt}(\g^{-1})^\ast(f)$ for any $u \in \C{U}$, and any $f\in \C{V}^\circ$.
\end{lemma}

\begin{proof}
We see at once that
\begin{align*}
& \langle (u\bullet u') \overset{\circ}{\rt} (\g^{-1})^\ast(f), v \rangle = \langle  (\g^{-1})^\ast(f), (\g^{-2})(v) \lt \phi^{-2}(u\bullet u')\rangle = \langle  f, (\g^{-1})((\g^{-2})(v) \lt \phi^{-2}(u\bullet u'))\rangle = \\
& \langle  f, (\g^{-3})(v) \lt \phi^{-3}(u\bullet u')\rangle = \langle  f, ((\g^{-4})(v) \lt \phi^{-3}(u)) \lt \phi^{-2}(u')\rangle = \\
& \langle  u' \overset{\circ}{\rt} f, \g^2((\g^{-4})(v) \lt \phi^{-3}(u))\rangle = \langle  u' \overset{\circ}{\rt} f, (\g^{-2})(v) \lt \phi^{-1}(u)\rangle  = \langle \phi(u)\overset{\circ}{\rt} (u' \overset{\circ}{\rt} f), v\rangle,  
\end{align*}
where the third and the sixth equalities follow from the assumption, and on the fourth equality we used the Hom-module associativity. As for the unitality, we observe that
\[
\langle \eta(1) \overset{\circ}{\rt} f, v \rangle = \langle f, \g^{-2}(v)\lt \phi^{-2}(\eta(1))\rangle = \langle f, \g^{-2}(v)\lt \eta(1)\rangle = \langle f, \g^{-1}(v)\rangle  = \langle (\g^{-1})^\ast(f), v\rangle.
\]
The first claim then follows. 

The second claim, on the other hand, follows from
\begin{align*}
& \langle (\g^{-1})^\ast(u \overset{\circ}{\rt} f), v\rangle = \langle u \overset{\circ}{\rt} f, \g^{-1}(v)\rangle = \langle  f, \g^{-3}(v) \lt \phi^{-2}(u)\rangle = \langle  f, \g^{-1}(\g^{-2}(v)\lt \phi^{-1}(u)) \rangle  = \\
& \langle (\g^{-1})^\ast( f), \g^{-2}(v)\lt \phi^{-1}(u)\rangle = \langle \phi(u) \overset{\circ}{\rt} (\g^{-1})^\ast( f), v\rangle.
\end{align*}
\end{proof}

\begin{proposition}
Let $(\C{U},\mu,\eta,\phi,\D,\ve,\psi)$ be a Hom-bialgebra, and let $(\C{V},\D_\C{V},\ve_\C{V},\b)$ be a Hom-coalgebra together with a Hom-coalgebra endomorphism $\g:\C{V}\to \C{V}$ so that $(\C{V},\g)$ is a right $\C{U}$-Hom-module coalgebra via $\lt:\C{V}\ot \C{U} \to \C{V}$ so that $\g(v\lt u) = \g(v)\lt \phi(u)$. Then, $(\C{V}^\circ,(\g^{-1})^\ast)$ is a left $\C{U}$-Hom-module algebra.
\end{proposition}

\begin{proof}
We begin with \eqref{Hom-mod-alg-00}. Indeed,
\begin{align*}
& \langle (\b^{-1})^\ast(u \overset{\circ}{\rt}  f), v\rangle  = \langle u \overset{\circ}{\rt}  f, \b^{-1}(v)\rangle = \langle  f, \g^{-2}(\b^{-1}(v))\lt \phi^{-2}(u)\rangle = \langle  f, \b^{-1}(\g^{-2}(v) \lt \psi(\phi^{-2}(u)))\rangle = \\
& \langle (\b^{-1})^\ast(f), \g^{-2}(v)\lt \phi^{-2}(\psi(u))\rangle = \langle \psi(u) \overset{\circ}{\rt} (\b^{-1})^\ast(f), v \rangle 
\end{align*}
for any $u\in \C{U}$, any $v\in \C{V}$, and any $f\in \C{V}^\circ$.

As for \eqref{Hom-mod-alg-I} we proceed along
\begin{align*}
& \langle \psi^2(u) \overset{\circ}{\rt} (f \star g), v\rangle = \langle f \star g, \g^{-2}(v)\lt \phi^{-2}(\psi^2(u))\rangle =  \\
& \langle f \ot g, \b^{-2}(\g^{-2}(v\ns{1})\lt \phi^{-2}(\psi^2(u\ns{1}))) \ot \b^{-2}(\g^{-2}(v\ns{2})\lt \phi^{-2}(\psi^2(u\ns{2})))\rangle  = \\
& \langle f , \b^{-2}(\g^{-2}(v\ns{1})\lt \phi^{-2}(\psi^2(u\ns{1}))) \rangle  \langle  g, \b^{-2}(\g^{-2}(v\ns{2})\lt \phi^{-2}(\psi^2(u\ns{2})))\rangle = \\
& \langle f , \b^{-2}(\g^{-2}(v\ns{1}))\lt \phi^{-2}(u\ns{1}) \rangle  \langle  g, \b^{-2}(\g^{-2}(v\ns{2}))\lt \phi^{-2}(u\ns{2})\rangle = \\
& \langle (u\ns{1} \overset{\circ}{\rt} f)\star (u\ns{2}\overset{\circ}{\rt} g) , v\rangle,
\end{align*}
where the third equality follows from \eqref{Hom-mod-coalg-I} for right actions. Similarly, \eqref{Hom-mod-coalg-II} implies \eqref{Hom-mod-alg-II}.
\end{proof}

We conclude with the following semidualization result.

\begin{proposition}\label{prop-semidual}
Let $(\C{U},\mu_\C{U}, \eta_\C{U},\phi,\D_\C{U},\ve_\C{U},\psi,S_\C{U})$ and $(\C{V},\mu_\C{V},\eta_\C{V},\a,\D_\C{V},\ve_\C{V},\b,S_\C{V})$ be two Hom-Hopf algebras, so that $\a^4\b^{-2}=\Id_\C{V} $, and that $\phi^4\psi^{-2} = \Id_\C{U}$. Then, $(\C{U},\C{V})$ is a matched pair of Hom-Hopf algebras, if and only if $(\C{V}^\circ,\C{U})$ is a mutual pair of Hom-Hopf algebras.
\end{proposition}

\begin{proof}
Let $(\C{U},(\C{V})$ be matched pair of Hom-Hopf algebras, and let us consider the dual Hom-Hopf algebra $(\C{V}^\circ, \G{m},\G{e},\G{a}, {\bf \D},\epsilon,\G{b},{\bf S})$. We shall begin with \eqref{comp-I}. For any $u\in \C{U}$, and any $f\in \C{V}^\circ$, and any $v,v'\in \C{V}$, we have on one hand
\begin{align*}
& \langle {\bf \D}(u\overset{\circ}{\rt} f) , v\ot v'\rangle = \langle u\overset{\circ}{\rt} f,  \a^{-2}(v\bullet v')\rangle = \langle  f,  \a^{-4}(v\bullet v') \lt \phi^{-2}(u)\rangle = \\
& \langle  f,  \Big(\a^{-4}(v)\bullet \a^{-4}(v') \Big)\lt \phi^{-2}(u)\rangle,   
\end{align*}
where on the first equality we used \eqref{delta}, and on the second equality \eqref{action-coaction-duality-II}. On the other hand,
\begin{align*}
& \langle \psi^{-1}(u\ns{1}\ns{0})\overset{\circ}{\rt} f\ns{1} , v\rangle \langle \G{a}^{-4}\G{b}^3(u\ns{1}\ns{1})\star \big[\phi(\psi^{-2}(u\ns{2}))\overset{\circ}{\rt} \G{a}^{-1}(f\ns{2})\big], v'\rangle = \\
& \langle  f\ns{1} , \a^{-2}(v) \lt \phi^{-2}\psi^{-1}(u\ns{1}\ns{0})\rangle \langle \G{a}^{-4}\G{b}^3(u\ns{1}\ns{1}), \b^{-2}(v'\ns{1})\rangle \langle \phi(\psi^{-2}(u\ns{2}))\overset{\circ}{\rt} \G{a}^{-1}(f\ns{2}), \b^{-2}(v'\ns{2})\rangle 
= \\
& \langle  f\ns{1} , \a^{-2}(v) \lt \phi^{-2}\psi^{-1}(u\ns{1}\ns{0})\rangle \langle u\ns{1}\ns{1}, \a^{-3}\b^2(v'\ns{1})\rangle\langle f\ns{2}, \a^{-2}\b^{-1}(v'\ns{2}) \lt \phi^{-1}(\psi^{-1}(u\ns{2}))\rangle = \\
& \langle f\ns{1}, \a^{-2}(v)\lt \phi^{-2}\psi^{-1}(u\ns{1}\ns{0}) \rangle \langle u\ns{1}\ns{1}, \a(v'\ns{1})\rangle\langle f\ns{2}, \a^{-2}\b^{-1}(v'\ns{2}) \lt \phi^{-1}(\psi^{-1}(u\ns{2}))\rangle = \\
& \langle f\ns{1}, \a^{-2}(v)\lt \Big(\a^{-3}\b^{-1}(v'\ns{1}) \rt \phi^{-2}(\psi^{-1}(u\ns{1}))\Big) \rangle\langle f\ns{2}, \a^{-2}\b^{-1}(v'\ns{2}) \lt \phi^{-1}(\psi^{-1}(u\ns{2}))\rangle = \\
& \langle  f,  \Big[\a^{-4}(v)\lt \Big(\a^{-5}\b^{-1}(v'\ns{1}) \rt \phi^{-4}(\psi^{-1}(u\ns{1}))\Big)\Big]\bullet \Big(\a^{-4}\b^{-1}(v'\ns{2}) \lt \phi^{-3}(\psi^{-1}(u\ns{2}))\Big)\rangle,
\end{align*}
where we used \eqref{dual-multp-star} on the first equality, \eqref{action-coaction-duality-II} on the second and the fourth equalities, and \eqref{delta} on the fifth equality. Accordingly we see that \eqref{vv'-lt-u} holds if and only if \eqref{comp-I} holds.

Similarly, for any $u,u'\in \C{U}$, and any $v\in \C{V}$, on one hand we have
\begin{align*}
& (u\bullet u')\ns{0}\langle  (u\bullet u')\ns{1}, v\rangle = \a^{-2}(v) \rt (u\bullet u'),
\end{align*}
where we used \eqref{action-coaction-duality-II}, and on the other hand,
\begin{align*}
& \psi^{-1}(u\ns{1}\ns{0}) \bullet u'\ns{0} \langle \G{a}^{-2}(\G{b}(u\ns{1}\ns{1}))\star \big[\phi^{-1}(u\ns{2})\overset{\circ}{\rt} \G{a}^{-1}(u'\ns{1})\big] , v\rangle = \\
& \psi^{-1}(u\ns{1}\ns{0}) \bullet u'\ns{0} \langle \G{a}^{-2}(\G{b}(u\ns{1}\ns{1})) , \b^{-2}(v\ns{1})\rangle \langle \phi^{-1}(u\ns{2})\overset{\circ}{\rt} \G{a}^{-1}(u'\ns{1}) , \b^{-2}(v\ns{2})\rangle = \\
& \psi^{-1}(u\ns{1}\ns{0}) \bullet u'\ns{0} \langle u\ns{1}\ns{1} , \a^{-1}(v\ns{1})\rangle \langle  u'\ns{1} , \a^{-2}\b^{-1}(v\ns{2}) \lt \phi^{-3}\psi(u\ns{2})\rangle = \\
& \Big(\a^{-3}\b^{-1}(v\ns{1}) \rt \psi^{-1}(u\ns{1})\Big) \bullet u'\ns{0} \langle  u'\ns{1} , \a^{-2}\b^{-1}(v\ns{2}) \lt \phi^{-3}\psi(u\ns{2})\rangle = \\
& \Big(\a^{-3}\b^{-1}(v\ns{1}) \rt \psi^{-1}(u\ns{1})\Big) \bullet \Big[\Big(\a^{-4}\b^{-1}(v\ns{2}) \lt \phi^{-5}\psi(u\ns{2})\Big) \rt u'\Big] = \\
&\Big(\a^{-3}\b^{-1}(v\ns{1}) \rt \psi^{-1}(u\ns{1})\Big) \bullet \Big[\Big(\a^{-4}\b^{-1}(v\ns{2}) \lt \phi^{-1}\psi^{-1}(u\ns{2})\Big) \rt u'\Big],
\end{align*}
where on the first equality we used \eqref{dual-multp-star}, and \eqref{action-coaction-duality-II} on the second, the third, and the fifth equalities. Hence, \eqref{v-rt-uu'} holds if and only if \eqref{comp-III} holds.

Finally, for any $u\in \C{U}$, any $f\in \C{V}^\circ$, any $v\in \C{V}$, and any $m\geq 0$, we have on one hand
\begin{align*}
&  u\ns{1}\ns{0} \langle \G{a}^{-2}\G{b}^2(u\ns{1}\ns{1})\star (u\ns{2} \overset{\circ}{\rt} f), v \rangle = u\ns{1}\ns{0} \langle \G{a}^{-2}\G{b}^2(u\ns{1}\ns{1}), \b^{-2}(v\ns{1}) \rangle\langle u\ns{2} \overset{\circ}{\rt} f, \b^{-2}(v\ns{2}) \rangle = \\
& \a^{-4}(v\ns{1}) \rt u\ns{1} \langle f , \a^{-2}\b^{-2}(v\ns{2}) \lt \phi^{-2}(u\ns{2})\rangle = \\
& \phi^2\Big(\a^{-6}(v\ns{1}) \rt \phi^{-2}(u\ns{1})\Big) \langle f , \a^{-2}\b^{-2}(v\ns{2}) \lt \phi^{-2}(u\ns{2})\rangle = \\
& \phi^2\Big(\a^{-6}(v\ns{1}) \rt \phi^{-2}(u\ns{1})\Big) \langle f , \a^{-6}(v\ns{2}) \lt \phi^{-2}(u\ns{2})\rangle,
\end{align*}
where on the first equality we used \eqref{dual-multp-star}, and on the second and the third equalities \eqref{action-coaction-duality-II}. On the other hand,
\begin{align*}
& u\ns{2}\ns{0} \langle (u\ns{1} \overset{\circ}{\rt} f) \star \G{a}^{-2}\G{b}^{2}(u\ns{2}\ns{1}), v \rangle = u\ns{2}\ns{0} \langle u\ns{1} \overset{\circ}{\rt} f , \b^{-2}(v\ns{1}) \rangle\langle  \G{a}^{-2}\G{b}^{2}(u\ns{2}\ns{1}), \b^{-2}(v\ns{2}) \rangle = \\
& u\ns{2}\ns{0} \langle f , \a^{-2}\b^{-2}(v\ns{1})\lt \phi^{-2}(u\ns{1}) \rangle\langle  u\ns{2}\ns{1}, \a^{-2}(v\ns{2}) \rangle = \\
& \a^{-4}(v\ns{2}) \rt u\ns{2} \langle f , \a^{-2}\b^{-2}(v\ns{1})\lt \phi^{-2}(u\ns{1}) \rangle  = \\
& \phi^2\Big(\a^{-6}(v\ns{2}) \rt \phi^{-2}(u\ns{2})\Big) \langle f , \a^{-2}\b^{-2}(v\ns{1})\lt \phi^{-2}(u\ns{1}) \rangle =\\
& \phi^2\Big(\a^{-6}(v\ns{2}) \rt \phi^{-2}(u\ns{2})\Big) \langle f , \a^{-6}(v\ns{1})\lt \phi^{-2}(u\ns{1}) \rangle. 
\end{align*}
Accordingly, \eqref{v-lt-u-ot-v-rt-u-switch} holds if and only if \eqref{comp-IV} holds.
\end{proof}

We can now state our main result as a direct corollary of Proposition \ref{prop-univ-envlp-mutual-pair}, Proposition \ref{prop-bicrossed-prod-Hom-Hopf}, and Proposition \ref{prop-semidual}.

\begin{corollary}\label{coroll-Hom-Lie-Hopf}
To any matched pair of Hom-Lie algebras $(\G{g},\phi)$ and $(\G{h},\a)$, so that $\phi^4 = \Id_\G{g}$ and $\a^4 = \Id_\G{h}$, there corresponds a bicrossed product Hom-Hopf algebra 
\[
\Big(\C{U}(\G{h})^\circ\acl \C{U}(\G{g}), \ast, \ve\ot {\bf 1},(\a^{-1})^\ast\ot\phi,\D_{\acl},\ve_{\acl},\Id_{\C{U}(\G{h})}^\ast\ot \Id_{\C{U}(\G{g})},S_{\acl}\Big).
\]
\end{corollary}

\bibliographystyle{plain}
\bibliography{references}{}

\begin{thebibliography}{10}

\bibitem{Abe-book}
E.~Abe.
\newblock {\em Hopf algebras}, volume~74 of {\em Cambridge Tracts in
  Mathematics}.
\newblock Cambridge University Press, 1980.
\newblock Translated from the Japanese by Hisae Kinoshita and Hiroko Tanaka.

\bibitem{CaenGoyv11}
S.~Caenepeel and I.~Goyvaerts.
\newblock Monoidal {H}om-{H}opf algebras.
\newblock {\em Comm. Algebra}, 39(6):2216--2240, 2011.

\bibitem{ConnMosc}
A.~Connes and H.~Moscovici.
\newblock Background independent geometry and {H}opf cyclic cohomology.

\bibitem{ConnMosc98}
A.~Connes and H.~Moscovici.
\newblock Hopf algebras, cyclic cohomology and the transverse index theorem.
\newblock {\em Comm. Math. Phys.}, 198(1):199--246, 1998.

\bibitem{Fuks-book}
D.~B. Fuks.
\newblock {\em Cohomology of infinite-dimensional {L}ie algebras}.
\newblock Contemporary Soviet Mathematics. Consultants Bureau, 1986.
\newblock Translated from the Russian by A. B. Sosinskii.

\bibitem{GelfFuks70}
I.~M. Gelfand and D.~B. Fuks.
\newblock Cohomologies of the {L}ie algebra of formal vector fields.
\newblock {\em Izv. Akad. Nauk SSSR Ser. Mat.}, 34:322--337, 1970.

\bibitem{Gohr10}
A.~Gohr.
\newblock On hom-algebras with surjective twisting.
\newblock {\em J. Algebra}, 324(7):1483--1491, 2010.

\bibitem{HartLarsSilv06}
J.~T. Hartwig, D.~Larsson, and S.~D. Silvestrov.
\newblock Deformations of {L}ie algebras using {$\sigma$}-derivations.
\newblock {\em J. Algebra}, 295(2):314--361, 2006.

\bibitem{HassShapSutl15}
M.~Hassanzadeh, I.~Shapiro, and S.~S\"utl\"u.
\newblock Cyclic homology for {H}om-associative algebras.
\newblock {\em J. Geom. Phys.}, 98:40--56, 2015.

\bibitem{Hoch59}
G.~Hochschild.
\newblock Algebraic {L}ie algebras and representative functions.
\newblock {\em Illinois J. Math.}, 3:499--523, 1959.

\bibitem{Hochschild-book}
G.~P. Hochschild.
\newblock {\em Basic theory of algebraic groups and {L}ie algebras}, volume~75
  of {\em Graduate Texts in Mathematics}.
\newblock Springer-Verlag, 1981.

\bibitem{Jako75}
N.~N. Jakovlev.
\newblock The cohomology of the {W}itt algebra.
\newblock {\em Funkcional. Anal. i Prilo\v{z}en.}, 9(3):95--96, 1975.

\bibitem{Kac68}
G.~I. Kac.
\newblock Extensions of groups to ring groups.
\newblock {\em Math. USSR Sb.}, 5:451--474, 1968.

\bibitem{Laur-GengMakhTele18}
C.~Laurent-Gengoux, A.~Makhlouf, and J.~Teles.
\newblock Universal algebra of a {H}om-{L}ie algebra and group-like elements.
\newblock {\em J. Pure Appl. Algebra}, 222(5):1139--1163, 2018.

\bibitem{LuWang16}
D.~Lu and S.~Wang.
\newblock The {D}rinfel'd double versus the {H}eisenberg double for
  {H}om-{H}opf algebras.
\newblock {\em J. Algebra Appl.}, 15(4):1650059, 24 pp., 2016.

\bibitem{Maji90}
S.~Majid.
\newblock Physics for algebraists: noncommutative and noncocommutative {H}opf
  algebras by a bicrossproduct construction.
\newblock {\em J. Algebra}, 130(1):17--64, 1990.

\bibitem{Majid-book}
S.~Majid.
\newblock {\em Foundations of quantum group theory}.
\newblock Cambridge University Press, 1995.

\bibitem{MakhPana15}
A.~Makhlouf and F.~Panaite.
\newblock Hom-{L}-{R}-smash products, {H}om-diagonal crossed products and the
  {D}rinfeld double of a {H}om-{H}opf algebra.
\newblock {\em J. Algebra}, 441:314--343, 2015.

\bibitem{MakhPana16}
A.~Makhlouf and F.~Panaite.
\newblock Twisting operators, twisted tensor products and smash products for
  hom-associative algebras.
\newblock {\em Glasg. Math. J.}, 58(3):513--538, 2016.

\bibitem{MakhSilv09}
A.~Makhlouf and S.~Silvestrov.
\newblock Hom-{L}ie admissible {H}om-coalgebras and {H}om-{H}opf algebras.
\newblock In {\em Generalized {L}ie theory in mathematics, physics and beyond},
  pages 189--206. Springer, Berlin, 2009.

\bibitem{MakhSilv10-II}
A.~Makhlouf and S.~Silvestrov.
\newblock Hom-algebras and {H}om-coalgebras.
\newblock {\em J. Algebra Appl.}, 9(4):553--589, 2010.

\bibitem{MakhSilv08}
A.~Makhlouf and S.~D. Silvestrov.
\newblock Hom-algebra structures.
\newblock {\em J. Gen. Lie Theory Appl.}, 2(2):51--64, 2008.

\bibitem{MoscRang09}
H.~Moscovici and B.~Rangipour.
\newblock Hopf algebras of primitive {L}ie pseudogroups and {H}opf cyclic
  cohomology.
\newblock {\em Adv. Math.}, 220(3):706--790, 2009.

\bibitem{RangSutl-I-arxiv}
B.~Rangipour and S.~S\"utl\"u.
\newblock {L}ie-{H}opf algebras and their {H}opf cyclic cohomology.
\newblock {\tt arXiv:math/1012.4827}, 2010.

\bibitem{RangSutl-III}
B.~Rangipour and S.~Sütlü.
\newblock S{AYD} {M}odules over {L}ie-{H}opf {A}lgebras.
\newblock {\em Comm. Math. Phys.}, 316(1):199--236, 2012.

\bibitem{RangSutl}
B.~Rangipour and S.~Sütlü.
\newblock A van {E}st isomorphism for bicrossed product {H}opf algebras.
\newblock {\em Comm. Math. Phys.}, 311(2):491--511, 2012.

\bibitem{Shen12}
Y.~Sheng.
\newblock Representations of hom-{L}ie algebras.
\newblock {\em Algebr. Represent. Theory}, 15(6):1081--1098, 2012.

\bibitem{ShenBai14}
Y.~Sheng and C.~Bai.
\newblock A new approach to hom-{L}ie bialgebras.
\newblock {\em J. Algebra}, 399:232--250, 2014.

\bibitem{Yau08}
D.~Yau.
\newblock Enveloping algebras of {H}om-{L}ie algebras.
\newblock {\em J. Gen. Lie Theory Appl.}, 2(2):95--108, 2008.

\bibitem{Yau08-II}
D.~Yau.
\newblock Module {H}om-algebras.
\newblock {\tt arXiv:math/0812.4695}, 2008.

\bibitem{Yau09}
D.~Yau.
\newblock Hom-algebras and homology.
\newblock {\em J. Lie Theory}, 19(2):409--421, 2009.

\bibitem{Yau10}
D.~Yau.
\newblock Hom-bialgebras and comodule {H}om-algebras.
\newblock {\em Int. Electron. J. Algebra}, 8:45--64, 2010.

\end{thebibliography}

\end{document}